\documentclass[11pt,reqno]{amsart}
\usepackage{amsmath,amsthm,amssymb,enumerate}
\usepackage{hyperref}
\usepackage{mathrsfs}
\usepackage{nicefrac}
\usepackage{mhequ}
\usepackage{comment}

\usepackage{color}

\usepackage[a4paper,innermargin=1.1in,outermargin=1.1in,bottom=1.5in,marginparwidth=1.5in,marginparsep=3mm]{geometry}

\date{\today}
\numberwithin{equation}{section}

\newtheorem{theorem}{Theorem}[section]
\newtheorem{lemma}[theorem]{Lemma}
\newtheorem{proposition}[theorem]{Proposition}
\newtheorem{corollary}[theorem]{Corollary}

\theoremstyle{definition}
\newtheorem{assumption}[theorem]{Assumption}
\newtheorem{definition}[theorem]{Definition}

\theoremstyle{remark}
\newtheorem{remark}[theorem]{Remark}

\newcommand\bR{\mathbb{R}}
\newcommand\bP{\mathbb{P}}

\newcommand\cF{\mathcal{F}}

\newcommand{\bT}{\mathbb{T}}
\newcommand{\bN}{\mathbb{N}}
\newcommand{\wto}{\rightharpoonup}

\newcommand{\tr}{\tilde{r}}

\newcommand*{\beq}{\begin{equation}}
\newcommand*{\eeq}{\end{equation}}
\newcommand{\E}{\mathbb{E}}

\newcommand{\bit}{\begin{itemize}}
\newcommand{\eit}{\end{itemize}}
\newcommand{\supp}{\text{supp}\,}

\newcommand{\D}{\partial}

\newcommand{\tu}{{\tilde{u}}}

\newcommand{\tb}{{\tilde{b}}}

\newcommand{\eps}{\varepsilon}
\newcommand{\T}{\mathbb{T}}

\newcommand{\fra}{\mathfrak{a}}
\newcommand{\tfra}{{\tilde{\mathfrak{a}}}}

\renewcommand{\le}{\lesssim}
\newcommand{\tsigma}{{\tilde\sigma}}

\newcommand{\txi}{{\tilde\xi}}
\newcommand{\N}{\mathbb{N}}
\newcommand{\ta}{\tilde{a}}

\usepackage{mathtools}

\makeatletter
\DeclareRobustCommand\widecheck[1]{{\mathpalette\@widecheck{#1}}}
\def\@widecheck#1#2{%
    \setbox\z@\hbox{\m@th$#1#2$}%
    \setbox\tw@\hbox{\m@th$#1%
       \widehat{%
          \vrule\@width\z@\@height\ht\z@
          \vrule\@height\z@\@width\wd\z@}$}%
    \dp\tw@-\ht\z@
    \@tempdima\ht\z@ \advance\@tempdima2\ht\tw@ \divide\@tempdima\thr@@
    \setbox\tw@\hbox{%
       \raise\@tempdima\hbox{\scalebox{1}[-1]{\lower\@tempdima\box
\tw@}}}%
    {\ooalign{\box\tw@ \cr \box\z@}}}
\makeatother

\DeclareMathOperator*{\esssup}{ess\,sup}

\begin{document}
\title{Nonlinear diffusion equations with nonlinear gradient noise}
\author{Konstantinos Dareiotis and  Benjamin Gess}
\begin{abstract}
We prove the existence and uniqueness of entropy solutions for nonlinear diffusion equations with nonlinear conservative gradient noise. As particular applications our results include stochastic porous media equations, as well as the one-dimensional stochastic mean curvature flow in graph form.
\end{abstract}
\maketitle

\section{Introduction}

In this work we consider stochastic partial differential equations of the type
\begin{equs}             \label{eq:equation-stratonovich}
\begin{aligned}
du &= \Big(  \Delta \Phi(u) + \nabla \cdot G(x,u) \Big) \, dt + \sum_{k=1}^\infty\left( \nabla \cdot \sigma^k(x, u)\right)  \circ d\beta^k(t) \quad\text{on }(0,T)\times\T^d
\\
u(0,x)&=\xi(x),
\end{aligned}
\end{equs} 
where $\T^d$ is the $d$-dimensional torus, $\beta^k$ are independent $\bR$-valued Brownian motions, $\Phi:\bR\to\bR$ is a monotone function (cf.\ Assumption \ref{as:A} below) and the coefficients $G: \bT^d \times \bR \to \bR^d$, $\sigma^k :\bT^d \times \bR \to \bR^d$ are regular enough (cf.\ Assumption \ref{as:sigma} below). The main results of this work are the existence and uniqueness of entropy solutions to \eqref{eq:equation-stratonovich} (Theorem \ref{thm:main-theorem} below) and the stability of \eqref{eq:equation-stratonovich} with respect to $\Phi$ (Theorem \ref{thm:uniqueness} below).

Stochastic partial differential equations of the type \eqref{eq:equation-stratonovich} arise as limits of interacting particle systems driven by common noise, with notable relation to the theory of mean field games \cite{LL06,LL06-2,LL07}, in the graph formulation of the stochastic mean curvature/curve shortening flow \cite{KO82,SY04,DLN01,ESR12} and as simplified approximating models of fluctuations in non-equilibrium statistical physics \cite{DSZ16}. We refer to \cite{FG18} and the references therein for more details on these applications. In particular, the results of this work imply the well-posedness of the stochastic mean curvature flow in one spatial dimension with spatially inhomogeneous noise, in the graph form, 
\begin{equs}                \label{eq:mean-curvature}
du =   \frac{\D^2_{xx} u}{1+|\D_x u|^2}  \, dt + \sum_{k=1}^\infty h^k(x)\sqrt{1+|\D_xu|^2} \circ d \beta^k(t),
\end{equs}
and thus extend the works \cite{ESR12,GR17} which were restricted to noise either satisfying a smallness condition or being independent of the spatial variable. For an alternative approach to stochastic mean curvature with spatially inhomogeneous noise based on stochastic viscosity solutions see \cite{Soug1, Soug2, Soug3} and the references therein. 

Generalized stochastic porous medium equations of the type
\begin{equs}             \label{eq:intro_PME}
\begin{aligned}
du(t,x) &= \Delta \Phi(u(t,x)) \, dt + B(u)dW_t
\end{aligned}
\end{equs}  
have attracted considerable interest and their well-posedness has been obtained for several classes of nonlinearities $\Phi$, diffusion operators $B$, boundary conditions and lower order perturbations. We refer to the monographs \cite{P75,KR79,PR07,LR15,BDPR16} for a detailed account on these developments and to \cite{Witt,GT14,BR15,BR17,DHV16,G12,GH18,FG18-2} and the references therein for recent contributions. While \textit{linear} gradient noise (cf.\ e.g.\ \cite{DG17,T18,MR18}), that is, $\sigma^k(x,u)=h^k(x)u$ in \eqref{eq:equation-stratonovich} to some extent can be treated by these methods, the \textit{nonlinear} structure of the gradient noise in \eqref{eq:equation-stratonovich} requires entirely different techniques. 
Only in recent years, in a series of works \cite{LPS13,LPS14,GS17-2,GS14-2,GS16-2,GS14} a kinetic approach to (simpler versions of) \eqref{eq:equation-stratonovich} was developed based on rough path methods, cf. also \cite{KarlsenRisebro,SougGess,Gassiat1,Gassiat2}.  for numerical methods and regularity/qualitative properties of the solutions. In the most recent contribution \cite{FG18} the path-by-path well-posedness of kinetic solutions to  \eqref{eq:equation-stratonovich}, with $\Phi(u)=u|u|^{m-1}$ for $m \in (0,\infty)$ (fast and slow diffusion),  was proved for the first time for non-negative initial data, while for sign-changing  data the uniqueness was restricted to the case $m >2$. As it is well-known from the theory of rough paths, such path-by-path methods require stronger regularity assumptions on the diffusion coefficients than what would be expected based on probabilistic methods. More precisely, when applied to \eqref{eq:equation-stratonovich}, the results of \cite{FG18} require 
 $\sigma^k(x,u) \in C^{\gamma}_b(\T^d \times \bR)\quad\forall k\in\N,$
for some $\gamma > 5$. Moreover, the construction of kinetic solutions presented in \cite{FG18} relies on the fractional Sobolev regularity of the solutions, which is available only in the particular case $\Phi(u)=|u|^{m-1}u$, $m\in (0,\infty)$. 

The key aims of the current work are to obtain well-posedness without sign restrictions on the initial data that covers the full spectrum of $m$ for the slow diffusion ($m>1$), to relax the regularity assumptions on the diffusion coefficients $\sigma^k$, and to treat a general class of diffusion nonlinearities $\Phi$. These aims are achieved by developing a  probabilistic entropy approach to \eqref{eq:equation-stratonovich} leading to the relaxed regularity assumption (cf.\ Assumption \ref{as:sigma} below for details) $\sigma^k(x,u) \in C^3_b(\T^d \times \bR)\quad\forall k\in\N$.
The treatment of general diffusion nonlinearites $\Phi$ is achieved by using quantified compactness in order to prove stability of \eqref{eq:equation-stratonovich} with respect to variations in $\Phi$. Based on this, the strong convergence of approximations can be shown, without relying on the compactness arguments from \cite{FG18} which were restricted to the case $\Phi(u)=|u|^{m-1}u$. In particular, this generalization allows the application to the stochastic mean curvature flow. The proof of stability relies on entropy techniques and a careful control of the errors arising in the corresponding doubling the variables argument which was  initiated in \cite{DareiotisGerencserGess} and is disjoint from the kinetic techniques put forward in \cite{FG18}.

The structure of the article is as follows. In Section 2 we formulate our main results concerning equations of porous medium type. In Section 3 we gather some lemmata that are needed for the proof of our main results. In Section 4, we prove the main estimates in $L_1(\bT^d)$ leading to uniqueness and stability and in Section 5 we show existence and uniqueness for non-degenerate equations.  In Section 6 we use the results if the two previous sections in order to prove our main theorem. Finally, in Section 7, we explain the modifications that need to be done in the proof of Theorem \ref{thm:main-theorem} in order to obtain existence and uniqueness of solutions of equation \eqref{eq:mean-curvature}. 

\subsection{Notation} We fix a filtered probability space $(\Omega, ( \mathcal{F}_t)_{ 
\in [0,T]}, \bP )$ carrying a sequence $(\beta^k(t))_{k \in \bN, t \in [0,T]}$ of independent, one-dimensional, $(\mathcal{F}_t)$-Wiener processes.
We introduce the notations $\Omega_T=\Omega\times[0,T]$, $Q_T=[0,T]\times\T^d$.
Lebesgue and Sobolev spaces are denoted in the usual way by $L_p$ and $W^k_p$, respectively.
When a function space is given on $\Omega$ or $\Omega_T$,
we understand it to be defined with respect to $\cF:=\cF_T$ and the predictable $\sigma$-algebra, respectively. 
In all other cases the usual Borel $\sigma$-algebra will be used. Moreover, throughout the whole article we fix a constant $m >1$. 

We fix a non-negative smooth function $\rho:\bR\to \bR$ which is bounded by $2$, supported in $(0,1)$, integrates to $1$ and, for $\theta>0$, we set $\rho_\theta(r)=\theta^{-1}\rho(\theta^{-1}r)$.
When smoothing in time by convolution with $\rho_\theta$, the property that $\rho$ is supported on positive times will be crucial.
For spatial regularisation this fact will be irrelevant, but for the sake of simplicity, we often use $\rho_\theta^{\otimes d}$ for smoothing in space as well. In the proofs of lemmas/theorems/propositions, we will often use the notation $a \le b$ which means $a \leq N b$ for a constant $N$ which depends only on the parameters stated in the corresponding   lemma/theorem/proposition. For a function  $g: \bT^d \times \bR \to \bR$ we will often use the notation  
\begin{align*}
 [ g ](x,r):= \int_0^r g(x,s) \, ds.
\end{align*}
If $g$ does not depend on $x \in \bT^d$, then we will write $[g](r)$. For a function  $g$ on $\bT^d \times \bR$, we will write $g_r$, $\D_r g$ for the derivative of $g$ with respect to the real variable $r \in \bR$ and $g_{x_i}$, $\D_{x_i}g$ for the partial derivatives of $g$ in the periodic variable $x \in \bT^d$. If $\gamma=(\gamma_1,...,\gamma_d) \in (\bN \cup \{0\})^d$ is a multi-index, we will write $\D^\gamma_xg:= \D^{\gamma_1}_{x_1}...\D^{\gamma_d}_{x_d} g$. For $\beta \in (0,1)$, $C^\beta$ will denote the usual H\"older spaces and $[ \cdot]_{C^\beta}$ will denote the usual semi-norm.  In addition, the summation convention with respect to integer valued indices will be in use. In particular, expressions of the form $ a^i b^i, f^i \D_{x_i}$ and $f^i_{x_i}$ will stand for $\sum_i a^i b^i, 
\sum_i f^i \D_{x_i}$ and $\sum_i f^i_{x_i}$ respectively, unless otherwise stated.
Finally, when confusion does not arise, in integrals we will drop some of the integration variables from the integrands for notational convenience. 
\section{Formulation and main results} \label{sec:formulation}

For $i,j \in \{1,...,d\}$, let us set 
\begin{equ}
a^{ij}(x,r)= \frac{1}{2}\sum_{k=1}^\infty\sigma^{ik}_r(x,r) \sigma^{jk}_r(x,r),\qquad b^i(x,r)= \sum_{k=1}^\infty \sigma^{ik}_r (x,r) \sum_{j=1}^d\sigma_{x_j}^{jk}(x, r),
\end{equ}
and 
\begin{equs}
f^i(x,r)= G^i(x,r)-\frac{1}{2}b^i(x,r).
\end{equs}
With this notation we rewrite \eqref{eq:equation-stratonovich} in It\^o form 
\begin{equs}         \label{eq:main-equation}
\begin{aligned}
du &= \left( \Delta \Phi(u)  + \D_{x_i} \left(a^{ij}(x,u)\D_{x_j} u + b^i(x,u)+  f^i(x,u) \right) \right) \, dt 
\\
& + \D_{x_i} \sigma^{ik}(x,u)\,  d\beta^k(t) 
\\
u(0)&= \xi.
\end{aligned}
\end{equs}
\begin{remark}
Formally, we have 
\begin{equs}
\D_{x_i} \sigma^{ik}(x,u) \circ d\beta^k(t) =& \D_{x_i} \sigma^{ik}(x,u)  d\beta^k(t)
\\
+&   \D_{x_i} \left(a^{ij}(x,u)\D_{x_j} u \right)+ \frac{1}{2}\D_{x_i} b^i(x,u) \, dt.
\end{equs}
In \eqref{eq:main-equation} we add  $b^i/2$ and then we subtract it from $G^i$ in order to make cancellations with terms coming from the It\^o correction when applying It\^o's formula apparent.  Despite the fact that $\D_{x_i} b^i$ and $\D_{x_i} f^i$ are of the same nature, they will be treated slightly differently to exploit these cancellations. 
 
\end{remark}
We will often write  $\Pi(\Phi, \xi)$ to address equation \eqref{eq:main-equation} with initial condition $\xi$ and nonlinearity $\Phi$. 
To formulate the assumptions on $\Phi$ let us set 
\begin{equs}
\fra(r)= \sqrt{\Phi'(r)}.
\end{equs}

\begin{assumption}\label{as:A}
The following hold:
\begin{enumerate}[(a)]
\item \label{as:A first}
The function $\Phi:\bR\mapsto\bR$ is differentiable, strictly increasing and odd. The function $\fra$ is differentiable away from the origin, and satisfies the bounds
\begin{equ}\label{eq:as fra}
|\fra(0)|\leq K,
\quad |\fra'(r)|\leq K|r|^{\frac{m-3}{2}}\quad\text{if}\ r>0,
\end{equ}
as well as
\begin{equation}\label{eq:as Psi}
K\fra(r)\geq I_{|r|\geq 1},\quad
  K|[\fra](r)-[\fra](\zeta)|\geq\left\{
  \begin{array}{@{}ll@{}}
    |r-\zeta|, & \text{if}\ |r|\vee|\zeta|\geq 1, \\
    |r-\zeta|^{\frac{m+1}{2}}, & \text{if}\ |r|\vee|\zeta|< 1.
  \end{array}\right.
\end{equation} 

\item \label{as:ic}
The initial condition $\xi$ is an $\mathcal{F}_0$-measurable $L_{m+1}(\bT^d)$-valued random variable such that $\E \| \xi \|_{L_{m+1}(\bT^d)}^{m+1}< \infty$.

\end{enumerate}
\end{assumption}

\begin{assumption}            \label{as:sigma}
For $i \in \{1,...,d\}$ we consider functions $G^i : \bT^d \times \bR \to \bR$, and $\sigma^i= (\sigma^{ik})_{k=1}^\infty: \bT^d \times \bR \to l_2$ such that for all $l \in \{1,...,d\}$, $q \in \{1,2\}$, and all multi-indices $\gamma \in ({\bN \cup\{0\}})^d$ with $q+|\gamma| \leq 3$,  the derivatives $\D_rG^i,\D_{x_l} G^i, \D_{rx_l}G^i$, $\D_r^q \D^\gamma_x \sigma^i$ exist and are continuous on $\bT^d \times \bR$. Moreover, there exist $ \bar \kappa \in ((m \wedge 2)^{-1}, 1] $, $\beta \in ( (2 \bar \kappa )^{-1}, 1]$, $\tilde{\beta} \in (0,1)$, and a constant $N_0 \in \bR $ such that for all $i,l \in \{1,...,d\}$,  $r \in \bR$ we have: 
\begin{equs}
\label{eq:bounded-sigma_r}
&\sup_r \| \sigma^i_r(\cdot, r)\|_{W^2_\infty(\bT^d; l_2)}+ [ \sigma^j_{x_j}(\cdot, 0) ]_{C^{\bar \kappa}( \bT^d, l_2)}\leq N_0,
\\             \label{eq:regularity-sigma-2}
&  \sup_x \left( [ \sigma_r(x, \cdot)]_{C^\beta(\bR;l_2)}+\| \sigma^i_{rx_l}(x, \cdot)\|_{W^1_\infty(\bR;l_2)}\right) \leq N_0,
\\                                \label{eq:sigma-x-sigma-xr}
& \|\D_r (\sigma^{jk}_{x_j} \sigma^{ik}_{rx_l}) \|_{L_\infty} \leq N_0,
\\                         \label{eq:Holder-b}
&\sup_{x} \|G^i_r(x, \cdot)\|_{C^{\beta}(\bR)}+ \sup_{x} \| \D_r(\sigma^{ik}_r \sigma^{jk}_{x_j})\|_{C^{\beta}(\bR)}\leq N_0, 
\\
\label{eq:Holder-bx-x}
&[G^i_{x_l}(\cdot, r)]_{C^{\tilde{\beta}}(\bT^d)}+[\D_{x_l}(\sigma^{ik}_r(\cdot, r)\sigma^{jk}_{x_j}(\cdot,r))]_{C^{\tilde{\beta}}(\bT^d)}  \leq N_0(1+|r|),
\\
&  \|\D_{x_l}\D_r(\sigma^{ik}_r \sigma^{jk}_{x_j})\|_{L_\infty} + \|G^i_{x_lr}\|_{L_\infty} \leq N_0.                        \label{eq:Lip-b}
\end{equs}

\end{assumption}

\begin{remark}
By Assumption \ref{as:sigma}, it follows that there exists a constant $N_1$ such that, for all $r \in \bR$
\begin{equs}
\label{eq:F-linear-growth}
&\sup_{x}|G^i(x,r)|+\sup_{x}|(\sigma^{ik}_r(x, r)\sigma^{jk}_{x_j}(x,r)| \leq N_1(1+|r|),
\\
\label{eq:b-linear-growth}
&\sup_x|\D_{x_l}(\sigma^{ik}_r(x, r)\sigma^{jk}_{x_j}(x,r))| +\sup_x|G^i_{x_l}(x, r)| \leq N_1(1+|r|),
\\                    
                     \label{eq:sigma-growth}
& \sup_ x | \sigma^j_{x_j}(x, r)|_{l_2}\leq N_1(1+|r|),
\\    \label{eq:regularity-sigma-5}
& [ \sigma^j_{x_j}(\cdot, r) ]_{C^{\bar \kappa}( \bT^d, l_2)} \leq N_1(1+|r|).
\end{equs}
\end{remark}

We now motivate the concept of entropy solutions. Suppose that we approximate equation \eqref{eq:main-equation} with a viscous equation, that is, in place of $\Phi(u)$ we have $\Phi(u)+\eps u$ for $\eps>0$. 
Let us choose a non-negative $\phi \in C^\infty_c([0,T) \times \bT^d)$ and a convex $\eta \in C^2(\bR)$. If $u(=u^\eps)$ solves the viscous version of \eqref{eq:main-equation}, by It\^o's formula we have (formally)
\begin{equs}
d \int_{\bT^d} \phi \eta(u) \, dx =&\int_{\bR^d} \left(  \phi_t \eta(u)   - \eta'(u)\Phi'(u) u_{x_i}  \phi_{x_i} -  \phi_{x_i} \eta'(u) a^{ij}(u) u_{x_j} - \phi_{x_i} \eta'(u) b^i(u) \right) \, dx dt 
\\
+& \int_{\bT^d} \phi\eta'(u)( f^i_r(u)u_{x_i}+  f^i_{x_i}(u) )\, dxdt  
\\
+&\int_{\bT^d} \eps \eta(u) \Delta \phi - \eps \phi \eta''(u) | \nabla u|^2 \, dx dt 
\\
 -&  \int_{\bT^d}\phi \eta''(u) \left(| \nabla [\fra](u)|^2 + a^{ij}(u) u_{x_i} u_{x_j} +   u_{x_i} b^i(u)\right) \, dxdt 
 \\
 + &\int_{\bT^d} \frac{1}{2}\phi \eta''(u) \left( 2a^{ij}(u)  u_{x_i} u_{x_j} +  2b^i(u)  u_{x_i}  + \sum_k | \sigma^{ik}_{x_i}(u) |^2 \right)\,  dxdt
 \\          \label{eq:primitive-entropy}
 +& \int_{\bT^d}\phi \eta'(u) \D_{x_i} \sigma^{ik}(u) \, dx d\beta^k(t).
\end{equs}
By integration by parts and the cancellations we have 
\begin{equs}
d \int_{\bT^d} \eta(u) \phi \, dx=&\int_{\bT^d}\left(   \eta(u)\phi_t   + [\fra^2\eta'] (u) \Delta \phi +[ a^{ij}\eta'](u)  \phi_{x_ix_j} \right)  \,dx dt 
\\
 +&\int_{\bT^d}\left(  [(a^{ij}_{x_j}-f^i_r )\eta' ](u) -\eta'(u) b^i(u)\right)  \phi_{x_i}   \,dx dt 
\\
+&\int_{\bT^d}\left( \eta'(u)f^i_{x_i}(u)-[f^i_{rx_i} \eta' ](u) \right)  \phi   \,dx dt 
\\
+&\int_{\bT^d} \eps \eta(u) \Delta \phi - \eps \phi \eta''(u) | \nabla u|^2 \, dx dt 
\\
 +&\int_{\bT^d}\left(  \frac{1}{2} \eta''(u)\sum_k | \sigma^{ik}_{x_i}(u) |^2\phi-  \eta''(u) | \nabla [ \fra ](u)|^2\phi  \right) \,dx dt 
\\                  
 +&\int_{\bT^d}\left(\eta'(u)\phi \sigma^{ik}_{x_i}(u)-  [\sigma^{ik}_{rx_i}\eta'](u) \phi - [\sigma^{ik}_r \eta'](u)  \phi_{x_i} \right)  \,dx d\beta^k(t).       
 \\           \label{eq:befor-epsilon-to-zero}
\end{equs}
Now we want to pass to the limit $\eps \downarrow 0$. Assuming for the moment that $u^\eps$ converges to some $u$ as $\eps \downarrow 0$ we may expect that
\begin{equs}
\int_0^T \int_{\bT^d} \eps \eta(u^\eps) \Delta \phi\, dx dt \to 0.
\end{equs}
In contrast, this may not be valid for the term 
\begin{equs}
I_\eps:=-\int_0^T \int_{\bT^d}\eps \phi \eta''(u) | \nabla u|^2 \, dx dt,
\end{equs}
since, in general, $\|\nabla u^\eps\|_{L_2(Q_T)}^2 \sim \eps^{-1}$. However, since $I_\eps \leq 0$, one may drop the term $I_\eps$ from the right hand side of \eqref{eq:befor-epsilon-to-zero}, replace the equality with an inequality, and then pass to the limit $\eps \downarrow 0$. This motivates the following definition.

\begin{definition}        \label{def:entropy-solution}
An entropy solution of \eqref{eq:main-equation} is 
a predictable stochastic process $u : \Omega_T \to L_{m+1}(\bT^d)$ such that
\begin{enumerate}[(i)]

\item  \label{item:in-Lm} $u \in  L_{m+1}(\Omega_T; L_{m+1}(\bT^d))$ 

\item \label{item:chain_ruleW2} For all $f \in C_b(\bR)$ we have $[\fra f](u) \in L_2(\Omega_T;W^1_2(\bT^d)$) and 
\begin{equation*}      
\D_{x_i}[ \fra f](u)= f(u) \D_{x_i} [\fra](u).
\end{equation*}

\item \label{item:entropies}For all convex $\eta\in C^2(\bR)$ with $\eta''$ compactly supported and all  $\phi\geq 0$ of the form $\phi= \varphi \varrho$ with $\varphi \in C^\infty_c([0,T))$, $\varrho \in C^\infty(\bT^d)$,   we have 
almost surely
\begin{equs}                    
-&\int_0^T \int_{\bT^d}  \eta(u)\phi_t\, dx dt 
\\
 \leq & \int_{\bT^d} \eta(\xi) \phi(0) \, dx 
+ \int_0^T \int_{\bT^d}  \left([\fra^2\eta'] (u)  \Delta \phi + [ a^{ij}\eta'](u) \phi_{x_ix_j} \right)  \, dx  dt 
\\
 +&\int_0^T\int_{\bT^d}\left(  [(a^{ij}_{x_j}-f^i_r) \eta' ](u) -\eta'(u) b^i(u)\right)  \phi_{x_i}   \,dx dt 
\\
+&\int_0^T\int_{\bT^d}\left( \eta'(u)f^i_{x_i}(u)-[f^i_{rx_i} \eta' ](u) \right)  \phi   \,dx dt 
\\
  + & \int_0^T \int_{\bT^d} \left(  \frac{1}{2} \eta''(u)\sum_k | \sigma^{ik}_{x_i}(u) |^2\phi-  \eta''(u) | \nabla [\fra](u)|^2\phi  \right) \, dx  dt 
\\        
  + &\int_0^T \int_{\bT^d}\left(\eta'(u)\phi \sigma^{ik}_{x_i}(u)-  [\sigma^{ik}_{rx_i}\eta' ](u) \phi - [\sigma^{ik}_r\eta' ](u)  \phi_{x_i} \right)  \,dx d\beta^k(t). 
 \\
\label{eq:entropy-inequality}
\end{equs}
\end{enumerate}
\end{definition}

\begin{remark}
  In \cite{FG18} a notion of pathwise kinetic solutions to \eqref{eq:equation-stratonovich} has been introduced. It is expected, although not immediate to prove, that in the regime where both approaches apply, pathwise kinetic solutions and entropy solutions in the sense of Definition \ref{def:entropy-solution} coincide. The difficulty in validating this lies in the identification of the stochastic integral in \cite{FG18}. In fact, in \cite{FG18} no meaning is given to the stochastic integral itself, but solutions are obtained as limits of smooth approximations of the noise. As a consequence, the identification of the two concepts would require the proof of a Wong-Zakai approximation result on the approximative level \cite[equation (5.1)]{FG18}.
\end{remark}

\begin{theorem}         \label{thm:main-theorem}
Let $\Phi$, $\xi$ satisfy  Assumptions \ref{as:A} and $\sigma, G$ satisfy Assumption \ref{as:sigma}. Then, there exists a unique entropy solution of equation  \eqref{eq:main-equation} with initial condition $\xi$. Moreover, if $\tu$ is the unique entropy solution of equation \eqref{eq:main-equation} with initial condition $\txi$, then 
\begin{equs}           \label{eq:main contraction}
\esssup_{t \leq T} \E \| u(t)-\tu(t)\|_{L_1(\bT^d)} \leq N \E \| \xi-\txi \|_{L_1(\bT^d)},
\end{equs}
where $N$ is a constant depending only on $N_0$, $N_1$,  $d$ and $T$.
\end{theorem}

\section{Auxiliary results}
In this section we state and we prove some tools that will be used for the proofs of the main theorem. We begin with two remarks.
\begin{remark}                  \label{rem:intergation-limits}
For any functions $f : \bR \times \bT^d \to \bR$, $u: \bT^d \to \bR$, $\phi : \bT^d \to \bR$ (that are regular enough for the following expressions to make sense) and any $a \in \bR$ we have 
\begin{equs}
&\int_{\bT^d} \D_{x_i} \phi(x) \int_0^{u(x)} f(r,x) \,ds dx- \int_{\bT^d} \phi(x) \int_0^{u(x)} \D_{x_i} f(r,x) \, ds dx 
\\
=&\int_{\bT^d} \D_{x_i} \phi(x) \int_a^{u(x)} f(r,x) \,ds dx- \int_{\bT^d} \phi(x) \int_a^{u(x)} \D_{x_i} f(r,x) \, ds dx.
\end{equs}
\end{remark}

\begin{remark}
For any $f \in L_1(0,T)$ and $\theta \in (0, T)$ we have 

\begin{equ}  \label{eq:whole-theta}          
\int_\theta^T\int_{t-\theta}^t|f(s)|\,ds\,dt\leq\theta\int_{0}^T|f(s)|\,ds.
\end{equ}
\end{remark}

\begin{lemma}              \label{lem:initial-condition}
Let $u$ be an entropy solution  \eqref{eq:main-equation}. Then we have that
$$
\lim_{h \to 0}\frac{1}{h} \E \int_0^h \int_{\bT^d} |u(t,x)- \xi(x)|^2 \, dx dt =0.
$$
\end{lemma}
\begin{proof}
For $\varrho_{\varepsilon} := \rho_\eps^{\otimes d}$, 
we have 
\begin{align}               \label{eq:spliting}
\nonumber
\frac{1}{h}\E \int_0^h \int_{x}  |u(t,x)-\xi(x)|^2  \,  dt 
&\leq 2 \E  \int_{x,y} |\xi(y)-\xi(x)|^2 \varrho_\varepsilon(x-y) 
\\
&+ \frac{2}{h}\E \int_0^h \int_{x,y} |u(t,x)-\xi(y)|^2 \varrho_\varepsilon(x-y)\,  dt .
\end{align}
 We first estimate the second term on the right hand side for $h \in[0,T]$. 
 Take a decreasing, non-negative function $\gamma\in C^\infty([0,T])$, such that 
\begin{equ}
\gamma(0)=2,\quad  \gamma\leq 2I_{[0,2h]},\quad\partial_t\gamma\leq -\frac{1}{h}I_{[0,h]}.
\end{equ}
Take furthermore for each $\delta>0$, $\eta_\delta\in C^2(\bR)$ defined by
\begin{equ}
\eta_\delta(0)=\eta_\delta^\prime(0)=0,\quad \eta_\delta''(r)=2I_{[0,\delta^{-1})}(|r|)+(-|r|+\delta^{-1}+2)I_{[\delta^{-1},\delta^{-1}+2)}(|r|),
\end{equ}
and notice that $\eta_\delta(r) \to r^2$ as $\delta \to 0$. 
Let $y\in\T^d$ and $a \in \bR$. Then, using the entropy inequality \eqref{eq:entropy-inequality} with $\phi(t,x)= \gamma(t) \varrho_\varepsilon(x-y)$, $\eta (r)= \eta_\delta(r-a)$, we obtain 
\begin{equs}            
&-\int_{t,x} \eta_\delta(u-a) \D_t \gamma(t) \varrho_\varepsilon(x-y)
\\
&\leq 2\int_{x}\eta_\delta(\xi-a)\varrho_\varepsilon(x-y) 
\\ 
& +N\int_{t,x} (1+|u|^{m+1}+|a|^{m+1}) \left( \sum_{ij}|\D_{x_ix_j}\varrho_\eps(x-y)| +\sum_{i}|\D_{x_i}\varrho_\eps(x-y)|+\varrho_\eps(x-y)\right) \gamma(t) 
\\  
&+\frac{1}{2}\int_{t,x} \eta_\delta^{\prime\prime}(u-a)\sum_k | \sigma^{ik}_{x_i}(x,u) |^2 \varrho_\varepsilon(x-y)\gamma(t)
\\   
&+ \int_0^T \int_x\left(\eta^{\prime}_\delta(u-a)\phi  \sigma^{ik}_{x_i}(u)-  [\sigma^{ik}_{rx_i}\eta^{\prime}_\delta(\cdot-a)](u) \phi - [\sigma^{ik}_r\eta^{\prime}_\delta(\cdot-a)](u)  \phi_{x_i} \right)   \, d\beta^k(t),
\end{equs}
where for the second term on the right hand side we have used \eqref{eq:as fra}, \eqref{eq:Holder-b}, \eqref{eq:b-linear-growth}, \eqref{eq:Lip-b},  \eqref{eq:bounded-sigma_r}, and \eqref{eq:sigma-growth}.
Notice that all the terms are continuous in $a \in \bR$.
Upon substituting $a=\xi(y)$ taking expectations, integrating over $y \in \bT^d$, and using the bounds on $\gamma$, one gets
\begin{equs}              
&\frac{1}{h}\int_0^h \E \int_{x,y} \eta_\delta(u(t,x)-\xi(y))  \varrho_\varepsilon(x-y)\,dt
\\
& 
\leq 2\E \int_{x,y}\eta_\delta(\xi(x)-\xi(y))\varrho_\varepsilon(x-y) 
\\ 
& +\frac{N}{\varepsilon^2}\E \int_0^{2h}\int_{x} (1+|u(t,x)|^{m+1}+|\xi(x)|^{m+1}) \,dt
\\
&+\E \int_0^{2h} \int_{x,y} \eta_\delta^{\prime\prime}(u(t,x)-\xi(y))\sum_k |  \sigma^{ik}_{x_i}(u(t,x)) |^2 \varrho_\varepsilon(x-y)\, dt.
\end{equs}
In the  limit $\delta\to0$ this yields
\begin{equs}          
\frac{1}{h}\E \int_0^h\int_{x,y} |u(t,x)-\xi(y)|^2  \varrho_\varepsilon(x-y)\,dt 
&\leq 2\E \int_{x,y}|\xi(x)-\xi(y)|^2\varrho_\varepsilon(x-y) \, dx
\\
& +\frac{N}{\varepsilon^2}\E \int_0^{2h}\int_{x} (1+|u(t,x)|^{m+1}+|\xi(x)|^{m+1})\, dt
\\
&+2\E \int_0^{2h} \int_{x,y}\sum_k | \sigma^{ik}_{x_i}(x,u(t,x)) |^2 \varrho_\varepsilon(x-y)\, dt,
\end{equs}
which implies that 
\begin{equs}
\limsup_{h \to 0}\frac{1}{h}\E &\int_0^h\int_{x,y} |u(t,x)-\xi(y)|^2  \varrho_\varepsilon(x-y)\,dt 
\leq 2\E \int_{x,y}|\xi(x)-\xi(y)|^2\varrho_\varepsilon(x-y).
\end{equs}
Consequently, by \eqref{eq:spliting} we get
\begin{align*}
&\limsup_{h \to 0}\frac{1}{h}\E \int_0^h\int_{x} |u(t,x)-\xi(x)|^2  \,dt 
\leq 3\E \int_{x,y}|\xi(x)-\xi(y)|^2\varrho_\varepsilon(x-y),
\end{align*}
from which the claim follows, since right hand side goes to $0$ as $\varepsilon \to 0$ due to the continuity of translations in $L_2(\bT^d)$.
\end{proof}
The proof of the following lemma can be found in \cite[Lemma 3.1]{DareiotisGerencserGess}. 

\begin{lemma}\label{lem:frac reg}
Let Assumption \ref{as:A} hold, let $u\in L_1(\Omega\times Q_T)$ and for some $\eps\in (0,1)$, let $\varrho:\bR^d\mapsto\bR$ be a non-negative function integrating to one and supported on a ball of radius $\eps$.
Then one has the bound
\begin{equ}\label{eq:frac reg}
\E\int_{t,x,y}|u(t,x)-u(t,y)|\varrho(x-y)\leq N \eps^{\frac{2}{m+1}}\big(1+\E\|\nabla[\fra](u)\|_{L_1(Q_T)}\big),
\end{equ}
where $N$ depends on $d, K$ and $T$.
\end{lemma}

We now introduce the definition of the ($\star$)-property, an analog of of which was first introduced in \cite{FengNualart} in the context of stochastic conservation laws. It is somewhat technical  but important in order to obtain the uniqueness of entropy solutions. To be more precise, as a first step,  we will estimate the difference of two entropy solutions provided that one of them has the $(\star)$-property. In the construction of entropy solutions it will be verified that, given that the initial condition is sufficiently integrable in $\omega$, the constructed solutions indeed satisfy the $(\star)$-property (see Corollary 3.9 and Lemma 5.3 below). 

Let $h \in C^\infty(\bR)$ with $h' \in C_c^\infty(\bR)$, $\varrho \in C^\infty(\bT^d \times \bT^d)$, $\varphi \in C^\infty_c((0,T))$,
$\tu\in L_{m+1}(\Omega_T; L_{m+1}( \T^d))$,
and let $\sigma$ satisfy Assumption \ref{as:sigma}.
For $\theta>0$, we introduce  
\begin{equs}
 \phi_\theta(t,x,s,y):=\varrho(x,y) \rho_\theta(t-s) \varphi\left(\frac{t+s}{2}\right).
\end{equs}
We further define
\begin{equs}
F_\theta(t,x,a):= &\int_0^T \int_{y} h(\tu-a) \sigma^{ik}_{y_i}(y,\tu) \phi_\theta(t,x,s,y) \, d\beta^k(s)
\\
- &\int_0^T \int_{y} [\sigma^{ik}_{rx_i}h(\cdot-a)](y,\tu)  \phi_\theta(t,x,s,y)d\beta^k(s)
\\
-&\int_0^T \int_{y}[\sigma^{ik}_rh(\cdot-a)](y,\tu) \D_{y_i} \phi_\theta(t,x,s,y) \, d\beta^k(s)
\end{equs}
and 
\begin{equs}
\mathcal{B}(u,\tu, \theta)=&  -\E \int_{t,x,s,y}   \D_{y_ix_j}\phi_\theta \int_\tu^u \int_r^\tu h'(\tilde{r}-r) \sigma^{ik}_r(y,\tilde{r}) 
  \sigma^{jk}_r(x, r) \, d\tilde{r} d r
 \\
 &- \E \int_{t,x,s,y}   \D_{y_i}\phi_\theta \int_\tu^u \int_r^\tu h'(\tr-r) \sigma^{ik}_r(y,\tr)  \sigma^{jk}_{rx_j}(x,r) \,  d\tr d r 
 \\
 &+\E \int_{t,x,s,y}   \D_{y_i}\phi_\theta \int_u^\tu h'(\tr-u) \sigma^{ik}_r(y,\tr)  \sigma^{jk}_{x_j} (x,u)  \, d\tr  
 \\
 &-\E \int_{t,x,s,y}    \D_{x_j}\phi_\theta  \int_\tu^u \int_r^\tu h'(\tr-r) \sigma^{ik}_{ry_i}(y,\tr) \sigma^{jk}_r(x,r) \, d\tr dr 
 \\
 &-\E \int_{t,x,s,y}   \phi_\theta \int_\tu^u  \int_r^\tu h'(\tr-r) \sigma^{ik}_{ry_i}(y,\tr)    \sigma^{jk}_{rx_j}(x,r) \, d\tr dr
 \\
 &+ \E \int_{t,x,s,y}   \phi_\theta \int_u^\tu h'(\tr-u) \sigma^{ik}_{ry_i}(y,\tr)  \sigma^{jk}_{x_j}(x,u) \, d\tr
 \\
  &+\E \int_{t,x,s,y}   \D_{x_j} \phi_\theta \int_\tu^u h'(\tu-r)  \sigma^{ik}_{y_i} (y,\tu) \sigma^{jk}_r(x,r) \, dr
  \\
   &+\E \int_{t,x,s,y}   \phi_\theta \int_\tu^u h'(\tu-r) \sigma^{ik}_{y_i} (y,\tu) \sigma^{jk}_{rx_j}(x,r) \, dr
  \\               \label{eq:def-error}
   &-\E \int_{t,x,s,y}   \phi_\theta  h'(\tu-u) \sigma^{ik}_{y_i} (y,\tu)   \sigma^{jk}_{x_j}(x,u),
   \end{equs}
   where $u=u(t,x)$ and $\tu=\tu(s,y)$. 
\begin{remark}
The function $F_\theta$ is smooth in $(t,x,a)$ (see, e.g., \cite[Exercise 3.15, page 78]{Kun}).
\end{remark}
Set $\mu=\mu(m)=\frac{3m+5}{4(m+1)}$, which is chosen so that one has $\frac{m+3}{2(m+1)}<\mu<1$.
\begin{definition}\label{def:star}
A function $u\in L_{m+1}(\Omega_T\times \T^d)$ is said to have the $(\star)$-property if for all $h,\varrho,\varphi,\tu$ as above,
and for all sufficiently small $\theta>0$, we have that $F_\theta(\cdot, \cdot, u) \in L_1(\Omega_T \times \bT^d)$ and
\begin{equ}
\E \int_{t,x} F_\theta(t,x, u(t,x)) \leq  N\theta^{1-\mu}+ \mathcal{B}(u, \tu, \theta)   \label{eq:strong-entropy}
\end{equ}
hold with some constant $N$ independent of $\theta$.
\end{definition}
\begin{remark}
 
Notice that since $\varphi$ is supported in $(0,T)$ and $\rho_\theta(t-
\cdot)$ is supported in $[t-\theta, t]$, we have for all sufficiently small $\theta$
\begin{equs}
 F_\theta(t,x,a)
= &I_{t >\theta} \int_{t-\theta}^t \int_{y} h(\tu-a) \sigma^{ik}_{y_i}(y,\tu) \phi_\theta(t,x,s,y) \, d\beta^k(s)
\\
- &I_{t >\theta} \int_{t-\theta}^t  \int_{y} [\sigma^{ik}_{rx_i}h(\cdot-a)](y,\tu)  \phi_\theta(t,x,s,y)d\beta^k(s)
\\
-&I_{t >\theta} \int_{t-\theta}^t  \int_{y}[\sigma^{ik}_rh(\cdot-a)](y,\tu) \D_{y_i} \phi_\theta(t,x,s,y) \, d\beta^k(s).
\label{eq:F rewrite}
\end{equs}
\end{remark}
\begin{lemma}\label{lem:F}
For any  $\lambda\in ( \frac{m+3}{2(m+1)}, 1)$, $k\in\mathbb{N}$ we have for all sufficiently small $\theta \in (0,1)$
\begin{equation}            \label{eq:F-estimate-sharper}
\E \| \D_ a F_\theta \|_{L_\infty([0,T]; W^{k}_{m+1}(\bT^d \times \bR))}^{m+1} \leq N \theta^{-\lambda (m+1)}\mathcal{N}_m(\tu),
\end{equation} 
where 
\begin{equs}
\mathcal{N}_m(\tu):= \E \int_0^T (1+ \| \tu(t)\|_{L_{\frac{m+1}{2}}(\bT^d)}^{m+1}+ \|\tu(t) \|_{L_2(\bT^d)}^{m+1}) \, dt 
\end{equs}
and $N$ is a constant depending only on $N_0, N_1, k,d, T ,\lambda$, $m$,  and the functions $h, \varrho, \varphi$, but not on $\theta$. In particular, 
\begin{equs}
\E \| \D_ a F_\theta \|_{L_\infty([0,T]; W^{k}_{m+1}(\bT^d \times \bR))}^{m+1} \leq N \theta^{-\lambda (m+1)}(1+ \E \| \tu\|_{L_{m+1}(Q_T)}^{m+1}) \label{eq:F estimate}.
\end{equs}
\end{lemma}

\begin{proof}
To ease the notation we suppress the $y\in\T^d$ argument in $\tsigma$ and the $s,y\in Q_T$ arguments in $\tu$.
For any  $q \in \mathbb{N}^d$, $l \in \mathbb{N}$, $j \in \{0,1\}$, we have 
by the Burkholder-Davis-Gundy inequality
\begin{equs}
 \E  |&\D_t^j\D^{l+1}_a \D^q_xF_\theta (t,x,a ) |^{m+1} \\
 \le & \E  I_{t>\theta}\left[\int_{t-\theta}^t \sum_k\left(\int_{y} \D_a^{l+1} h(\tu-a)  \sigma^{ik}_{y_i}(\tu)   \D^q_x \D^j_t \phi_\theta \right)^2  \, ds \right]^{(m+1)/2}
 \\
 + & \E I_{t>\theta} \left[ \int_{t-\theta}^t \sum_k \left( \int_{y} \D_a^{l+1} [\sigma^{ik}_{rx_i}h(\cdot-a)](y,\tu) \D^q_x \D^j_t \phi_\theta \right)^2  ds \right]^{(m+1)/2}
 \\
 + &\E  I_{t>\theta} \left[ \int_{t-\theta}^t \sum_k \left(\int_{y} \D_a^{l+1}[\sigma^{ik}_rh(\cdot-a)](y,\tu) \D^q_x \D^j_t\D_{y_i} \phi_\theta \right)^2 \, ds \right]^{(m+1)/2}
\\
=& C_1+C_2+C_3.   \label{eq:C1C2C3}
\end{equs}
We deal first with $C_3$. By H\"older's inequality and \eqref{eq:bounded-sigma_r}, we have 
\begin{equs}
&\E  I_{t>\theta} \left[ \int_{t-\theta}^t \sum_k \left(\int_{y} \D_a^{l+1}[\sigma^{ik}_rh(\cdot-a)](y,\tu) \D^q_x \D^j_t\D_{y_i} \phi_\theta \right)^2 \, ds \right]^{(m+1)/2}
\\
 \le&  \E  I_{t>\theta} \left[ \int_{t-\theta}^t  \left(\int_{y} \int_{-|\tu|}^{|\tu |}  |\D_a^{l+1}h(r-a)|^2 \, dr \right) \left(\int_{y} \int_{-|\tu|}^{|\tu |} \sum_k \left| \sigma^{ik}_{r}(y,r) \right|^2 dr\right)   \theta^{-2(j+1)}  \, ds  \right]^{(m+1)/2}
\\
 \le &  \E  I_{t>\theta} \left[ \int_{t-\theta}^t  \left( \int_{y} \int_{-|\tu|}^{|\tu |}  |\D_a^{l+1}h(r-a)|^2 \, dr\right)  \| \tu\|_{L_1(\bT^d)} \ \theta^{-2(j+1)}  \, ds \right]^{(m+1)/2}.
\end{equs}
By H\"older's inequality we get 
\begin{equs}
C_3  \le  \theta^{\frac{m-1}{2}}\theta^{-(m+1)(1+j)} \E  I_{t>\theta}  \int_{t-\theta}^t  \int_{y} \left[ \int_{-|\tu|}^{|\tu |}  |\D_a^{l+1}h(r-a)|^2 \, dr \right]^{(m+1)/2}\| \tu\|_{L_1(\bT^d)}^{(m+1)/2}  \, ds 
\\
\le \theta^{\frac{m-1}{2}}\theta^{-(m+1)(1+j)} \E  I_{t>\theta}  \int_{t-\theta}^t  \| \tu\|_{L_1(\bT^d)}^{(m+1)/2} \int_{y} |\tu|^{(m-1)/2} \int_{-|\tu|}^{|\tu |}  |\D_a^{l+1}h(r-a)|^{(m+1)} \, dr   \, ds.
\\
\label{eq:double-role} 
\end{equs}
By integrating over $a  \in \bR$, using the fact that $h' \in C_c^\infty(\bR)$,  integrating over $[0, T] \times \bT^d$ and using the estimate \eqref{eq:whole-theta}
we obtain 
\begin{equation}      \label{eq:estC2}
\int_{t,x,a} C_3 \le \theta^{\frac{m-1}{2}}\theta^{-(m+1)(1+j)+1} \E \int_0^T \| \tu(t) \|_{L_{\frac{m+1}{2}}(\bT^d)}^{m+1}\, dt .
\end{equation}
In the same manner, one obtains
\begin{equation}            \label{eq:estC3}
\int_{t,x,a} C_2  \le \theta^{\frac{m-1}{2}}\theta^{-(m+1)(1+j)+1}\E \int_0^T \| \tu(t) \|_{L_{\frac{m+1}{2}}(\bT^d)}^{m+1}\, dt .
\end{equation}
Similarly, by \eqref{eq:sigma-growth}, H\"older's inequality, and \eqref{eq:whole-theta}, we obtain
\begin{equation}             \label{eq:estC1}
\int_{t,x,a} C_1 \leq \theta^{\frac{m-1}{2}}\theta^{-(m+1)(1+j)+1}\E\int_0^T (1+ \| \tu(t)\|_{L_{2}(\bT^d)}^{m+1})\, dt .
\end{equation}
Consequently, by \eqref{eq:estC2}-\eqref{eq:estC1} and \eqref{eq:C1C2C3}, we obtain 
\begin{equation}         \label{eq:choose-j}
\int_{t,x,a} \E  |\D_t^j\D^{l+1}_a \D^q_xF_\theta (t,x,a ) |^{m+1} \le  \theta^{-(m+1)(1+j)+1} \mathcal{N}_m(\tu).
\end{equation}
Choosing $j=0$ and  summing  over all $|q|+l \leq k$,  we obtain
\begin{equation}               \label{eq:LpWk}
\E \|\D_a F_\theta \|_{L_{m+1}([0,T]; W^{k}_{m+1}(\bT^d \times \bR))}^{m+1}\le \theta^{-\frac{m+1}{2}}\mathcal{N}_m(\tu).
\end{equation} 
Similarly, choosing $j=1$ in \eqref{eq:choose-j} and summing  over all $|q|+l \leq k$ gives
\begin{equation}                   \label{eq:W1Wk}
\E \|\D_a F_\theta \|_{W^1_{m+1}([0,T] ; W^{k}_{m+1}(\bT^d \times \bR))}^{m+1} \le  \theta^{-3\frac{(m+1)}{2}}\mathcal{N}_m(\tu).
\end{equation}
By interpolating between \eqref{eq:LpWk} and \eqref{eq:W1Wk}  we have for $\delta \in [0,1]$
\begin{equation*}
\E \|\D_a F_\theta \|_{W^\delta_{m+1}([0,T] ; W^{k}_{m+1}(\bT^d \times \bR))}^{m+1} \le  \theta^{-(m+1)(1+2\delta)/2} \mathcal{N}_m(\tu).
\end{equation*}
For arbitrary $\delta \in (1/(m+1),1/2)$, we set  $\lambda=(1+2\delta)/2$, 
and the claim follows by Sobolev embedding.
\end{proof}
\begin{corollary}\label{cor} 
(i) Let $u_n$ be a sequence bounded in $L_{m+1}(\Omega_T\times \bT^d)$, satisfying the $(\star)$-property  uniformly in $n$, that is,  with constant $N$ in \eqref{eq:strong-entropy} independent of $n$. 
Suppose that $u_n$ converges for almost all $\omega,t,x$ to a function $u$. Then $u$ has the $(\star)$-property.

(ii) Let $u\in L_2(\Omega\times Q_T)$. Then one has for all $\theta>0$
\begin{equ}\label{eq:0lambda limit}
\E \int_{t,x} F_\theta(t,x, u(t,x)) = \lim_{\lambda \to 0}  \E \int_{t,x,a} F_\theta(t,x, a) \rho_\lambda(u(t,x)-a) \, .
\end{equ}
\end{corollary}

\begin{proof}
(i) We have that $\lim_{n \to \infty} F_\theta(t,x, u_n(t,x))=F_\theta(t,x, u(t,x))$ for almost all $(\omega,t,x)$. Moreover, 
\begin{equ}\label{eq:simple}
|F_\theta(t,x, u_n(t,x))|\leq \|\D_aF_\theta\|_{L_\infty(Q_T\times\bR)} |u_n(t,x)|+|F(t,x,0)|. 
\end{equ}
By Lemma \ref{lem:F}, and the fact that  $\E\int_{t,x}|F_\theta(t,x,0)|<\infty$,
we see that the right hand side above is uniformly integrable in $(\omega,t,x)$.
Hence, one can take limits on
the left-hand side of \eqref{eq:strong-entropy} to get
$$
\lim_{n\to\infty}\E \int_{t,x} F_\theta(t,x, u_n(t,x))=\E \int_{t,x} F_\theta(t,x, u(t,x)).
$$
By similar (in fact, easier) arguments one can see the convergence of the second term on the right-hand side of \eqref{eq:strong-entropy}, and since   the constant $N$ was assumed to be independent of $n \in \bN$, we get the claim.

(ii) Writing
$$
\big| F_\theta(t,x,u(t,x))-\int_aF_\theta(t,x, a) \rho_\lambda(u(t,x)-a) \,  \big|
\leq \lambda\| \D_a F_\theta\|_{L_\infty(Q_T\times\bR)},
$$
the claim simply follows from Lemma \ref{lem:F}.
\end{proof}
\section{Stability under the $(\star)$-property} \label{Stability}

\begin{theorem}  \label{thm:uniqueness}
Let $(\Phi, \xi)$,  $(\tilde{\Phi}, \tilde{\xi})$   satisfy Assumption \ref{as:A}, and $\sigma, G$ satisfy Assumption \ref{as:sigma}.
Let $u$, $\tu$ be two entropy solutions of $\Pi(\Phi,\xi)$, $\Pi(\tilde{\Phi}, \txi)$ respectively,
and assume that $u$ has the $(\star)$-property.
Then,

\begin{enumerate}[(i)]
\item if furthermore $\Phi=\tilde{\Phi}$, then 
\begin{equ}\label{eq:L_1 contraction}
\esssup_{t\in[0,T]}\E\|u(t)-\tu(t)\|_{L_1(\bT^d)} \leq N \E\|\xi-\txi \|_{L_1(\bT^d)},
\end{equ}
where $N$ is a constant depending only on $N_0, N_1, d$ and $T$, 
\item  \label{it:super-inequality}for all $\eps,\delta\in(0,1]$, $\lambda\in[0,1]$ and $\alpha\in(0,1\wedge(m/2))$,  we have 
\begin{equs}
\E \|u-\tilde{u}\|_{L_1(Q_T)}   
& \leq N \E \|\xi-\txi\|_{L_1(\bT^d)}
\\
&+N\eps^{\frac{2}{m+1}}\big(1+\E\|\nabla[\fra](u)\|_{L_1(Q_T)}\big)
+N\sup_{|h|\leq\eps} \E\|\txi(\cdot)-\txi(\cdot+h)\|_{L_1(\T^d)}
\\
&+N\eps^{-2}\E\big(\|I_{|u|\geq R_\lambda}(1+|u|)\|_{L_m(Q_T)}^m+
\|I_{|\tu|\geq R_\lambda}(1+|\tu|)\|_{L_m(Q_T)}^m\big)
\\
&+N C(\delta, \eps, \lambda) \E(1+\|u\|^{m+1}_{L_{m+1}(Q_T)}+\|\tu\|_{L_{m+1}(Q_T)}^{m+1}),
\label{eq:super inequality}
\end{equs}
where
\begin{equs}\label{eq:R lambda}
R_\lambda& :=\sup\{R\in[0,\infty]:\,|\fra(r)-\tfra(r)|\leq\lambda, \,\,\forall |r|<R\},
\\
C(\delta, \eps, \lambda)&:= \big(\delta^\beta + \delta^{2\beta} \eps^{-2}+\delta^\beta \eps^{-1}+\eps^{2\bar\kappa}\delta^{-1}+\eps^{-2}\delta^{2\alpha}+\eps^{-2}\lambda^2+\eps^{\tilde{\beta}}+\eps^{\bar \kappa}),
\end{equs}
and $N$ is a constant depending only on $N_0, N_1, m,K,d,T$, and $\alpha$.
\end{enumerate}
\end{theorem}

We collect first some technical results that will be needed for the proof of the above theorem. Let us first introduce some notation that will be used throughout this section.  

Denote $\varrho_\eps=\rho_\eps^{\otimes d}$, and fix a $\varphi \in C^\infty_c((0,T))$ such that 
$\|\varphi\|_{L_\infty([0,T])}\vee\|\partial_t\varphi\|_{L_1([0,T])}\leq 1$.
Introduce, for $\theta,\eps>0$,
\begin{equ}
\phi_{\theta, \varepsilon}(t, x,s,y)= \rho_{\theta}(t-s)\varrho_\varepsilon\left(x-y\right) \varphi \left(\tfrac{t+s}{2}\right),
\,\,
\phi_{\eps}(t,x,y)=\varrho_\eps(x-y)\varphi(t).
\end{equ}
Furthermore,  for each $\delta>0$, let $\eta_\delta\in C^2(\bR)$ be defined by
\begin{equ}
\eta_\delta(0)=\eta_\delta^\prime(0)=0,\quad\eta_\delta^{\prime\prime}(r)=\rho_\delta(|r|).
\end{equ}
Note that
\begin{equ}\label{eq:eta prop}
\big|\eta_\delta(r)-|r|\big|\leq\delta,\quad\supp\eta_\delta^{\prime\prime}\subset[-\delta,\delta],
\quad\int_\bR|\eta_\delta^{\prime\prime} (r-\zeta)|\,d\zeta\leq2,
\quad|\eta^{\prime\prime}_\delta|\leq 2\delta^{-1}.
\end{equ}
For $g : \bT^d \times \bR \to \bR$ we introduce the notation
\begin{equation}
[g,\delta](x,r,a):= [g\eta^{\prime}_\delta( \cdot-a)](x,r). 
\end{equation}
Finally, with the short hand notation  $u=u(t,x)$ and $\tu=\tu(t,y)$ let us define 
the quantities 
 \begin{equs}
  \mathcal{A}^{(\eps, \delta) } (u, \tu):= & \E \int_{t,x,y}   \left( [ a^{ij}, \delta](x,u,\tu) \D_{x_ix_j} \phi_{ \varepsilon}  + \left( [a^{ij}_{x_j},\delta ](x,u,\tu) -\eta^{\prime}_\delta(u-\tu) b^i(x,u)\right) \D_{x_i} \phi_{\varepsilon}\right)  
\\
+ & \E \int_{t,x,y}  \left( [ a^{ij}, \delta](y,\tu,u) \D_{y_iy_j} \phi_{\varepsilon}  + \left( [a^{ij}_{x_j},\delta ](y,\tu,u) -\eta^{\prime}_\delta(\tu-u) b^i(y,\tu)\right) \D_{y_i} \phi_{ \varepsilon}\right), 
  \end{equs}
also, 
\begin{equs}
\mathcal{B}^{(\eps, \delta)}_ 1(u,\tu)=&  -\E \int_{t,x,y}   \D_{y_ix_j} \phi_\eps \int_\tu^u \int_r^\tu \eta''_\delta(\tilde{r}-r) \sigma^{ik}_r(y,\tilde{r}) 
  \sigma^{jk}_r(x, r) \, d\tilde{r} d r
 \\
\mathcal{B}^{(\eps, \delta)}_2(u,\tu)= &- \E \int_{t,x,y}   \D_{y_i}\phi_\eps \int_\tu^u \int_r^\tu \eta''_\delta(\tr-r) \sigma^{ik}_r(y,\tr)  \sigma^{jk}_{rx_j}(x,r) \,  d\tr d r 
 \\
\mathcal{B}^{(\eps, \delta)}_3(u,\tu) =&\E \int_{t,x,y}   \D_{y_i}\phi_\eps \int_u^\tu \eta''_\delta(\tr-u) \sigma^{ik}_r(y,\tr)  \sigma^{jk}_{x_j} (x,u)  \, d\tr  
 \\
\mathcal{B}^{(\eps, \delta)}_4(u,\tu)= &-\E \int_{t,x,y}    \D_{x_j}\phi_\eps  \int_\tu^u \int_r^\tu \eta''_\delta(\tr-r) \sigma^{ik}_{ry_i}(y,\tr) \sigma^{jk}_r(x,r) \, d\tr dr 
 \\
\mathcal{B}^{(\eps, \delta)}_5(u,\tu)= &-\E \int_{t,x,y}   \phi_\eps \int_\tu^u  \int_r^\tu \eta''_\delta(\tr-r) \sigma^{ik}_{ry_i}(y,\tr)    \sigma^{jk}_{rx_j}(x,r) \, d\tr dr
 \\
\mathcal{B}^{(\eps, \delta)}_6(u,\tu)= & \E \int_{t,x,y}   \phi_\eps \int_u^\tu \eta''_\delta(\tr-u) \sigma^{ik}_{ry_i}(y,\tr)  \sigma^{jk}_{x_j}(x,u) \, d\tr
 \\
\mathcal{B}^{(\eps, \delta)}_7(u,\tu) = &\E \int_{t,x,y}   \D_{x_j} \phi_\eps \int_\tu^u \eta''_\delta(\tu-r)  \sigma^{ik}_{y_i} (y,\tu) \sigma^{jk}_r(x,r) \, dr
  \\
 \mathcal{B}^{(\eps, \delta)}_8(u,\tu) = &\E \int_{t,x,y}   \phi_\eps \int_\tu^u \eta''_\delta(\tu-r) \sigma^{ik}_{y_i} (y,\tu) \sigma^{jk}_{rx_j}(x,r) \, dr
  \\
\mathcal{B}^{(\eps, \delta)}_9(u,\tu)   =&-\E \int_{t,x,y}  \phi_\eps  \eta''_\delta(\tu-u) \sigma^{ik}_{y_i} (y,\tu)   \sigma^{jk}_{x_j}(x,u)
\end{equs}
and
\begin{equs}
 \mathcal{B}^{(\eps, \delta)}(u,\tu) : = \sum_{l=1}^9 \mathcal{B}^{(\eps, \delta)}_l(u, \tu),
   \end{equs}
and finally,
\begin{equs}
\mathcal{C}^{(\eps, \delta)}(u, \tu):= &\E \int_{t,x,y}\left(\eta^{\prime}_\delta(u-\tu) f^i_{x_i}(x,u)\phi_{\varepsilon}-[f^i_{rx_i}, \delta ](x,u,\tu)\phi_{\varepsilon} -[f^i_r ,\delta](x,u, \tu)\D_{x_i} \phi_{ \varepsilon}\right)  
\\
+ &\E \int_{t,x,y}\left(\eta^{\prime}_\delta(\tu-u) f^i_{x_i}(y,\tu)\phi_{\varepsilon}-[f^i_{rx_i}, \delta ](y,\tu,u)\phi_{ \varepsilon} -[f^i_r ,\delta](y,\tu,u)\D_{y_i} \phi_{ \varepsilon}\right).
\end{equs}

  With this notation we have the following lemmata. 
\begin{lemma}\label{lem:A+B}
There exists a constant $N=N(N_0, N_1, d, T)$ such that for all $u, \tu \in L_1(Q_T)$ and all $\eps, \delta \in (0,1)$
\begin{equs}
\mathcal{A}^{(\eps, \delta)}(u, \tu) 
+ \sum _{l=1}^8 \mathcal{B}^{(\eps, \delta)}_l(u, \tu)&  \leq N C_0(\eps, \delta) \left(  \E \|u\|_{L_1(Q_T)}+  \E \|\tu\|_{L_1(Q_T)}\right)
\\
+& N   \E \int_{t,x,y} \left(\varepsilon^ 2\sum_{ij}| \D_{x_iy_j}\phi_\eps|  
+\varepsilon \sum_{i}|\D_{x_i}\phi_\eps| + \phi_\eps\right) |u-\tu|,
\end{equs}
where 
\begin{equs}
C_0(\eps, \delta) = \delta^{2\beta}\eps^{-2}+\delta^\beta \eps^{-1}+\delta^\beta+\eps+\eps^{\bar \kappa}.
\end{equs}
\end{lemma}

\begin{proof}
By  Remark \ref{rem:intergation-limits} (with $a= \tu(t,y)$), the relation $\D_{x_i x_j} \phi_\eps= - \D_{x_i y_j}\phi_\eps$, and the identity 
\begin{equs}                            \label{eq:idendity-eta}
\eta'_\delta(r-\tu) = \int_\tu^r \eta''_\delta(r- \tr)  \, d\tr,
\end{equs}
we have
\begin{equs}                             
&\E \int_{t,x,y}   \left( [ a^{ij}, \delta](x,u,\tu) \D_{x_ix_j} \phi_{ \varepsilon}  + \left( [a^{ij}_{x_j},\delta ](x,u,\tu) -\eta^{\prime}_\delta(u-\tu) b^i(x,u)\right) \D_{x_i} \phi_{\varepsilon}\right) 
\\
=&-\E \int_{t,x,y} \D_{x_iy_j} \phi_\eps \int_\tu^u\int_\tu^r  \eta^{\prime \prime}_\delta(r-\tr)a^{ij}(x,r) d\tr dr
\\                                 \label{eq:u,tu}
&-\E \int_{t,x,y} \left( \D_{x_i} \phi_\eps \eta^\prime_\delta(u-\tu)b^i(x,u)+ \D_{x_i} \phi_\eps \int_\tu^u \eta^\prime_\delta(r-\tu) a^{ij}_{x_j}(x,r) \, dr \right).
\end{equs}
By symmetry we have that 
\begin{equs}                
& \E \int_{t,x,y}  \left( [ a^{ij}, \delta](y,\tu,u) \D_{y_iy_j} \phi_{\varepsilon}  + \left( [a^{ij}_{x_j},\delta ](y,\tu,u) -\eta^{\prime}_\delta(\tu-u) b^i(y,\tu)\right) \D_{y_i} \phi_{ \varepsilon}\right)  
\\
=&-\E \int_{t,x,y} \D_{x_iy_j} \phi_\eps \int_u^{\tu}\int_u^{\tr}  \eta^{\prime \prime}_\delta(\tr-r)a^{ij}(y,\tr) dr d \tr
\\                                       \label{eq:tu,u}
&-\E \int_{t,x,y} \left( \D_{x_i} \phi_\eps \eta^\prime_\delta(\tu-u)b^i(y,\tu)+ \D_{x_i} \phi_\eps \int_u^\tu \eta^\prime_\delta(\tr-u) a^{ij}_{y_j}(y,\tr) \, d\tr \right).
\end{equs}
Notice that 
\begin{equs}
&-\E \int_{t,x,y} \D_{x_iy_j} \phi_\eps \int_\tu^u\int_\tu^r  \eta^{\prime \prime}_\delta(r-\tr)a^{ij}(x,r) d\tr dr
\\
=&-\E \int_{\tu \leq u} \D_{x_iy_j} \phi_\eps \int_\tu^u\int_\tu^u I_{\tr \leq r} \eta^{\prime \prime}_\delta(r-\tr)a^{ij}(x,r) d\tr dr
\\                      \label{eq:expan1}
 &-\E \int_{\tu \geq u} \D_{x_iy_j} \phi_\eps \int_\tu^u\int_\tu^u I_{\tr \geq r} \eta^{\prime \prime}_\delta(r-\tr)a^{ij}(x,r) d\tr dr.
\end{equs}
Similarly
\begin{equs}
&-\E \int_{t,x,y} \D_{x_iy_j} \phi_\eps \int_u^{\tu}\int_u^{\tr}  \eta^{\prime \prime}_\delta(\tr-r)a^{ij}(y,\tr) dr d \tr
\\
=&-\E \int_{\tu \geq u} \D_{x_iy_j} \phi_\eps \int_\tu^u\int_\tu^u I_{r \leq \tr} \eta^{\prime \prime}_\delta(\tr-r)a^{ij}(y,\tr) dr d\tr
\\                               \label{eq:expan2}
 &-\E \int_{\tu \leq u} \D_{x_iy_j} \phi_\eps \int_\tu^u\int_\tu^u I_{r \geq \tr} \eta^{\prime \prime}_\delta(\tr-r)a^{ij}(y,\tr) dr d\tr.
\end{equs}
By adding \eqref{eq:u,tu} and \eqref{eq:tu,u} and using \eqref{eq:expan1}, \eqref{eq:expan2} we obtain
\begin{equs}
\mathcal{A}^{(\eps, \delta)}(u, \tu)= \mathcal{A}^{(\eps, \delta)}_1(u, \tu) +\mathcal{A}^{(\eps, \delta)}_2(u, \tu),
\end{equs}
where
\begin{equs}
\mathcal{A}^{(\eps, \delta)}_1(u, \tu):=&-\E \int_{\tu \leq u} \D_{x_iy_j} \phi_\eps \int_\tu^u\int_\tu^u I_{\tr \leq r} \eta''_\delta(r-\tr)a^{ij}(x,r) d\tr dr
\\
 &-\E \int_{\tu \geq u} \D_{x_iy_j} \phi_\eps \int_\tu^u\int_\tu^u I_{\tr \geq r} \eta''_\delta(r-\tr)a^{ij}(x,r) d\tr dr
\\
&-\E \int_{\tu \geq u} \D_{x_iy_j} \phi_\eps \int_\tu^u\int_\tu^u I_{r \leq \tr} \eta^{\prime \prime}_\delta(\tr-r)a^{ij}(y,\tr) dr d\tr
\\                               \label{eq:expan3}
 &-\E \int_{\tu \leq u} \D_{x_iy_j} \phi_\eps \int_\tu^u\int_\tu^u I_{r \geq \tr} \eta^{\prime \prime}_\delta(\tr-r)a^{ij}(y,\tr) dr d\tr
\end{equs}
and 
\begin{equs}
\mathcal{A}^{(\eps, \delta)}_2(u, \tu):=&-\E \int_{t,x,y} \left( \D_{x_i} \phi_\eps \eta'_\delta(u-\tu)b^i(x,u)+ \D_{x_i} \phi_\eps \int_\tu^u \eta'_\delta(r-\tu) a^{ij}_{x_j}(x,r) \, dr \right) 
\\
&
-\E \int_{t,x,y} \left(\D_{y_i} \phi_\eps\eta'_\delta(\tu-u)b^i(y,\tu)+ \D_{y_i}\phi_\eps\int_u^\tu \eta'_\delta(\tr-u) a^{ij}_{x_j}(y,\tr) \, d \tr \right) .
\end{equs}
We further set 
\begin{equs}
\mathcal{A}^{(\eps, \delta)}_{2,1}(u, \tu)&= -\E \int_{t,x,y}\D_{x_i} \phi_\eps\eta'_\delta(u-\tu)b^i(x,u)-\E \int_{t,x,y}\D_{y_i} \phi_\eps\eta'_\delta(\tu-u)b^i(y,\tu)
\\
&=:\mathcal{A}^{(\eps, \delta)}_{2,1, 1}(u, \tu)+\mathcal{A}^{(\eps, \delta)}_{2,1, 2}(u, \tu),    \label{eq:defB1}
\\            
\mathcal{A}^{(\eps, \delta)}_{2,2}(u, \tu)&=- \E \int_{t,x,y} \left( \D_{x_i} \phi_\eps\int_\tu^u \eta'_\delta(r-\tu) a^{ij}_{x_j}(x,r) \, dr+ \D_{y_i} \phi_\eps\int_u^\tu \eta'_\delta(\tr-u) a^{ij}_{x_j}(y,\tr) \, d\tr \right) .
\\
\label{eq:defB1B2}
\end{equs}

We next estimate $\mathcal{A}^{(\eps, \delta)}_1(u, \tu)+\mathcal{B}^{(\eps, \delta)}_1(u, \tu)$.  Notice that 
\begin{equs}
\mathcal{B}^{(\eps, \delta)}_1(u, \tu)=&\E \int_{t,x,y}   \D_{y_ix_j} \phi_\eps \int_\tu^u \int_\tu^r \eta''_\delta(r-\tilde{r}) 
  \sigma^{jk}_r(x, r) \sigma^{ik}_r(y,\tilde{r})  \, d\tilde{r} d r
 \\
  =&\E \int_{ \tu \leq u} \D_{y_ix_j}\phi_\eps \int_\tu^u 
 \int_\tu ^u I_{\tr \leq r} \eta_\delta''(r-\tr)  \sigma^{jk}_r(x,r) \sigma^{ik}_r(y,\tr) \, d\tr   dr
 \\                         \label{eq:E1}
 + & \E \int_{ \tu \geq u} \D_{y_ix_j}\phi_\eps  \int_\tu^u            
 \int_\tu ^u I_{\tr \geq r} \eta_\delta''(r-\tr)   \sigma^{jk}_r(x,r)\sigma^{ik}_r(y,\tr) \, d \tr d r.
\end{equs}
By the definition of $a^{ij}$ we have that 
\begin{equs}
&a^{ij}(x,r)+a^{ij}(y, \tr)-\sigma_r^{ik}(x,r)\sigma_r^{jk}(y, \tr)
\\
= &  
\frac{1}{2}\sigma^{ik}_r(x, r)(\sigma^{jk} _r(x,r) - \sigma^{jk}_r(y, \tr)) - \frac{1}{2}\sigma^{jk}_r(y, \tr) ( \sigma^{ik}_r(x,r)-\sigma^{ik}_r(y,\tr)).
\end{equs}
Using the fact that $\D_{x_i y_j} \phi_\eps = \D_{x_jy_i}\phi_\eps$ we see that 

\begin{equs}
&\D_{x_i y_j} \phi_\eps ( a^{ij}(x,r)+a^{ij}(y, \tr)-\sigma_r^{ik}(x,r)\sigma_r^{jk}(y, \tr))
\\
=& \frac{1}{2} \D_{x_i y_j} \phi_\eps (\sigma^{ik}_r(x, r)- \sigma^{ik}_r(y, \tr))(\sigma^{jk} _r(x,r) - \sigma^{jk}_r(y, \tr))
\\                \label{eq:eps-delta}
\le &  \sum_{ij} |\D_{x_i y_j}\phi_\eps|  ( \eps + \delta^\beta)^2  \le  \sum_{ij} |\D_{x_i y_j}\phi_\eps|  ( \eps^2 + \delta^{2\beta}),  
\end{equs}
where we have used \eqref{eq:bounded-sigma_r} and \eqref{eq:regularity-sigma-2}. Consequently, by \eqref{eq:expan3}, \eqref{eq:E1}, and \eqref{eq:eps-delta} combined with the fact that 
\begin{equs}                  \label{eq:estimate-eta''}
\left| \int_\tu ^u \int_\tu ^u \eta_\delta^{\prime \prime} (r-\tr)  \, d\tr dr \right| \leq 2 |\tu -u|,
\end{equs}
we obtain 
\begin{equs}
&\mathcal{A}^{(\eps, \delta)}_1(u, \tu)+\mathcal{B}^{(\eps, \delta)}_1(u, \tu)
\\
 \le&   \ \E \int_{t,x,y} \varepsilon^ 2\sum_{ij} | \D_{x_iy_j}\phi_\eps|  |u(t,x)-\tilde{u}(t,y)|+  \E \int_{t,x,y} \delta^ {2\beta}\sum_{ij} | \D_{x_iy_j}\phi_\eps|  |u(t,x)-\tilde{u}(t,y)|
 \\                                 
 \le&   \ \E \int_{t,x,y} \varepsilon^ 2 \sum_{ij}| \D_{x_iy_j}\phi_\eps|  |u(t,x)-\tilde{u}(t,y)|+   \delta^ {2\beta}\eps^{-2}\E(\|u\|_{L_1(Q_T)}+\|\tu\|_{L_1(Q_T)}),
 \\
 \label{eq:2}
\end{equs}
where we have used Assumption \ref{as:sigma}. We proceed with an estimate for $\mathcal{A}^{(\eps, \delta)}_{2,2}(u, \tu)+\mathcal{B}^{(\eps, \delta)}_2(u, \tu)+\mathcal{B}^{(\eps, \delta)}_4(u, \tu)$. Using the fact that $\D_{x_i} \phi_\eps= -\D_{y_i} \phi_\eps$ we get
\begin{equs}
\mathcal{A}^{(\eps, \delta)}_{2,2}(u, \tu) =-& \E \int_{\tu \leq u }  \D_{x_i} \phi_\eps \int_\tu^u \int_\tu^u I_{\tr \leq r}\eta_\delta''(r-\tr)\left( a^{ij}_{x_j}(x,r)- a^{ij}_{x_j}(y,\tr)\right)\, dr d\tr
\\              \label{eq:aij1}
-& \E \int_{\tu \geq u }  \D_{x_i} \phi_\eps \int_u^\tu \int_u^\tu I_{\tr \geq r}\eta_\delta''(r-\tr)\left( a^{ij}_{x_j}(x,r)- a^{ij}_{x_j}(y,\tr)\right)\, dr d\tr.
\end{equs}
By \eqref{eq:bounded-sigma_r} and \eqref{eq:regularity-sigma-2} we have
\begin{equs}                 \label{eq:aij2}
&| a^{ij}_{x_j}(x,r)- a^{ij}_{x_j}(y,\tr)|
\le &  \  |r-\tr|^\beta+ |x-y|.
\end{equs} 
Again, using the fact that $\D_{x_i} \phi_\eps= -\D_{y_i} \phi_\eps$ and relabelling $i \leftrightarrow j$ in $\mathcal{B}^{(\eps, \delta)}_4(u, \tu)$, gives
\begin{equs}
&\mathcal{B}^{(\eps, \delta)}_2(u, \tu)+\mathcal{B}^{(\eps, \delta)}_4(u, \tu)
\\
=& \E \int_{t,x,y}   \D_{x_i}\phi_\eps \int_\tu^u \int_r^\tu \eta^{\prime \prime}_\delta(\tr-r) \sigma^{ik}_r(y,\tr)  \sigma^{jk}_{rx_j}(x,r) \,  d\tr d r 
 \\
 -& \E \int_{t,x,y}    \D_{x_i}\phi_\eps  \int_\tu^u \int_r^\tu \eta^{\prime \prime}_\delta(\tr-r) \sigma^{jk}_{ry_j}(y,\tr) \sigma^{ik}_r(x,r) \, d\tr dr 
 \\
=&  \E \int_{\tu \leq u}\int_\tu^u \int_\tu^u I_{\tr \leq r}\eta_\delta^{\prime \prime}(r-\tr)\D_{x_i} \phi_\eps \left( \sigma^{ik}_r(x,r)\sigma^{jk}_{ry_j}(y, \tr) - \sigma^{ik}_r(y,\tr)\sigma^{jk}_{rx_j}(x, r) \right) \, d \tr dr 
\\             
+&\E \int_{\tu \geq u}\int_u^\tu \int_u^\tu I_{\tr \geq r}\eta_\delta^{\prime \prime}(r-\tr)\D_{x_i} \phi_\eps \left( \sigma^{ik}_r(x,r)\sigma^{jk}_{ry_j}(y, \tr) - \sigma^{ik}_r(y,\tr)\sigma^{jk}_{rx_j}(x, r) \right) d \tr dr.
\\
  \label{eq:sigma-ij1}
\end{equs}
By  \eqref{eq:bounded-sigma_r} and \eqref{eq:regularity-sigma-2} again we have
\begin{equs}          \label{eq:sigma-ij2}
 &\left|\sigma^{ik}_r(x,r)\sigma^{jk}_{ry_j}(y, \tr) - \sigma^{ik}_r(y,\tr)\sigma^{jk}_{rx_j}(x, r) \right| 
\le \  |r-\tr|^\beta+|x-y|.
\end{equs}
By adding \eqref{eq:aij1} and \eqref{eq:sigma-ij1} and using \eqref{eq:aij2}, \eqref{eq:sigma-ij2}, and \eqref{eq:estimate-eta''}, we obtain
\begin{equs}                 
&\mathcal{A}^{(\eps, \delta)}_{2,2}(u, \tu)+\mathcal{B}^{(\eps, \delta)}_2(u, \tu)+\mathcal{B}^{(\eps, \delta)}_4(u, \tu)
\\ 
\le &  \ \E \int_{t,x,y} \delta^\beta \sum_{i} |\D_{x_i}\phi_\eps| |u(t,x)-\tilde{u}(t,y)|  + \E \int_{t,x,y} \varepsilon \sum_{i} |\D_{x_i}\phi_\eps| |u(t,x)-\tilde{u}(t,y)| \ 
\\                    
\label{eq:3}
\le &  \ \delta^\beta \eps^{-1}\E (\|u\|_{L_1(Q_T)}+\|\tu\|_{L_1(Q_T)} )  + \E \int_{t,x,y} \varepsilon \sum_{i}|\D_{x_i}\phi_\eps| |u(t,x)-\tilde{u}(t,y)|.
\end{equs}
We proceed with the estimation of $\mathcal{A}^{(\eps, \delta)}_{2,1}(u, \tu)+\mathcal{B}^{(\eps, \delta)}_3(u, \tu)+\mathcal{B}^{(\eps, \delta)}_7(u, \tu)$. Recall that $\mathcal{A}^{(\eps, \delta)}_{2,1}(u, \tu)=\mathcal{A}^{(\eps, \delta)}_{2,1,1}(u, \tu)+\mathcal{A}^{(\eps, \delta)}_{2,1,2}(u, \tu)$, see \eqref{eq:defB1}. Using the fact that $\D_{y_i} \phi_\eps = -\D_{x_i} \phi_\eps$ and the definition of $b^i$, we see that
\begin{equs}
&  \ \ \ \mathcal{B}^{(\eps, \delta)}_3(u, \tu)+\mathcal{A}^{(\eps, \delta)}_{2,1,1}(u, \tu) 
\\
&=\E \int_{t,x,y}   \D_{y_i}\phi_\eps \int_u^\tu \eta^{
\prime \prime}_\delta(\tr-u) \sigma^{ik}_r(y,\tr)  \sigma^{jk}_{x_j} (x,u)  \, d\tr  
 -\E \int_{t,x,y}\D_{x_i} \phi_\eps\eta^{\prime}_\delta(u-\tu)b^i(x,u)
\\
 &= \E \int_{t,x,y}  \D_{x_i} \phi_\eps \int_u^\tu \eta_\delta^{\prime \prime}(r-u) \sigma^{jk}_{x_j}(x,u)\left(\sigma^{ik}_r(x,u) -\sigma^{ik}_r(y,r)\right)  \, dr.
\end{equs}
Using this, \eqref{eq:regularity-sigma-2}, and 
\begin{equ}               \label{eq:est-eta''2}
\int_{\mathbb{R}}\eta''(\tr-u) \, d \tr \leq 2, 
\end{equ}
 we see that 
\begin{equs}
 &\ \ \ \  \mathcal{B}^{(\eps, \delta)}_3(u, \tu)+\mathcal{A}^{(\eps, \delta)}_{2,1,1}(u, \tu)  
 \\
& \le \delta^\beta \varepsilon^{-1}\E (1+\|u\|_{L_1(Q_T)}) + \E \int_{t,x,y}  \D_{x_i} \phi_\eps \int_u^\tu \eta_\delta^{\prime \prime}(r-u) \sigma^{jk}_{x_j}(x,u)\left(\sigma^{ik}_r(x,u) -\sigma^{ik}_r(y,u)\right)  \, dr 
\\
&= \delta^\beta \varepsilon^{-1} \E (1+\|u\|_{L_1(Q_T)})+ \E \int_{t,x,y}  \D_{x_i} \phi_\eps  \eta_\delta^{\prime}(\tu-u) \sigma^{jk}_{x_j}(x,u)\left(\sigma^{ik}_r(x,u) -\sigma^{ik}_r(y,u)\right)  
\\
&= \delta^\beta \varepsilon^{-1}\E (1+\|u\|_{L_1(Q_T)}) + \E \int_{t,x,y}  \D_{x_i} \phi_\eps \eta_\delta^{\prime}(\tu-u) \sigma^{jk}_{x_j}(x,u) (x_l-y_l)\int_0^1 \sigma^{ik}_{rx_l}(y+\theta(x-y),u) \, d\theta.
\end{equs}
Similarly,
\begin{equs}
&\mathcal{B}^{(\eps, \delta)}_7(u, \tu)+\mathcal{A}^{(\eps, \delta)}_{2,1,2}(u, \tu) 
\\
\le & \delta^\beta \varepsilon^{-1}\E (1+\|\tu\|_{L_1(Q_T)}) -\E \int_{t,x,y}  \D_{y_i} \phi_\eps  \eta_\delta^{\prime}(\tu-u) \sigma^{jk}_{x_j}(y,\tu) (y_l-x_l)\int_0^1 \sigma^{ik}_{rx_l}(x+\theta(y-x),\tu) \, d\theta.
\end{equs}
Using the relation $\D_{x_i}\phi_\eps= -\D_{y_i} \phi_\eps$, we obtain
\begin{equs}
&\mathcal{A}^{(\eps, \delta)}_{2,1}(u, \tu)+\mathcal{B}^{(\eps, \delta)}_3(u, \tu)+\mathcal{B}^{(\eps, \delta)}_7(u, \tu)
\\
\le & \delta^\beta \varepsilon^{-1}\E (1+\|u\|_{L_1(Q_T)}+\|\tu\|_{L_1(Q_T)})+ \E \int_{t,x,y}   \D_{x_i} \phi_\eps  \eta_\delta^{\prime}(\tu-u) (x_l-y_l) 
\\
& \times \left(  \sigma^{jk}_{x_j}(x,u) \int_0^1 \sigma^{ik}_{rx_l}(y+\theta(x-y),u) \, d\theta   -
\sigma^{jk}_{x_j}(y,\tu) \int_0^1 \sigma^{ik}_{rx_l}(x+\theta(y-x),\tu) \, d\theta \right) 
\\
\le&  \delta^\beta \varepsilon^{-1}\E (1+\|u\|_{L_1(Q_T)}+\|\tu\|_{L_1(Q_T)})+  \E \int_{t,x,y}   |\D_{x_i} \phi_\eps|  |x_l-y_l|
\\
& \times \left|  \sigma^{jk}_{x_j}(x,u) \int_0^1 \sigma^{ik}_{rx_l}(y+\theta(x-y),u) \, d\theta   -
\sigma^{jk}_{x_j}(y,\tu) \int_0^1 \sigma^{ik}_{rx_l}(x+\theta(y-x),\tu) \, d\theta \right|.
\\
\label{eq:E3E7B1-1}
\end{equs}
By \eqref{eq:sigma-growth} and \eqref{eq:bounded-sigma_r} we have 
\begin{equs}
&\left|  \sigma^{jk}_{x_j}(x,u) \int_0^1 \sigma^{ik}_{rx_l}(y+\theta(x-y),u) \, d\theta   -
\sigma^{jk}_{x_j}(y,\tu) \int_0^1 \sigma^{ik}_{rx_l}(x+\theta(y-x),\tu) \, d\theta \right|
\\
 \le & 
 \eps (1+|u|+|\tu|) +  \left|  \sigma^{jk}_{x_j}(x,u)  \sigma^{ik}_{rx_l}(x,u)    -
\sigma^{jk}_{x_j}(y,\tu)  \sigma^{ik}_{rx_l}(y,\tu)  \right|.  \label{eq:E3E7B1-2}
\end{equs}
By \eqref{eq:sigma-x-sigma-xr} we have 
\begin{equs}
 &\left|  \sigma^{jk}_{x_j}(x,u)  \sigma^{ik}_{rx_l}(x,u)    -
\sigma^{jk}_{x_j}(y,\tu)  \sigma^{ik}_{rx_l}(y,\tu)  \right|
\\
\le &\left|  \sigma^{jk}_{x_j}(x,u)  \sigma^{ik}_{rx_l}(x,u)    -
\sigma^{jk}_{x_j}(y,u)  \sigma^{ik}_{rx_l}(y,u)  \right| +|u-\tu|
\\
\le & \ (\varepsilon^{\bar \kappa}+ \eps )(1+|\tu|) +|u-\tu|,               \label{eq:E3E7B1-3}
\end{equs}
where for the last inequality we have used \eqref{eq:bounded-sigma_r}, \eqref{eq:sigma-growth}, and \eqref{eq:regularity-sigma-5}. Combining \eqref{eq:E3E7B1-1}, \eqref{eq:E3E7B1-2}, and \eqref{eq:E3E7B1-3}, we obtain
\begin{equs}
& \mathcal{A}^{(\eps, \delta)}_{2,1}(u, \tu)+\mathcal{B}^{(\eps, \delta)}_3(u, \tu)+\mathcal{B}^{(\eps, \delta)}_7(u, \tu)
\\
 \le &  (\delta^\beta \varepsilon^{-1}+ \eps+\eps^{\bar \kappa})\E (1+\|u\|_{L_1(Q_T)}+\|\tu\|_{L_1(Q_T)})
                           + \E \int_{t,x,y}  \eps \sum_{i} |\D_{x_i} \phi_\eps|   \left| u - \tu \right| .  
                           \\
                             \label{eq:4}
\end{equs}

We proceed with the estimation of the remaining terms. By \eqref{eq:bounded-sigma_r} and \eqref{eq:est-eta''2} we have 
\begin{equs}
\mathcal{B}^{(\eps, \delta)}_5(u,\tu)=&-\E \int_{t,x,y}  \phi_\eps \int_\tu^u \int_r^\tu \eta_\delta''(r-\tr) \sigma^{ik}_{ry_i}(y,\tr) \, dr   \sigma^{jk}_{rx_j}(x,r) \, d \zeta
\\            \label{eq:5}
\le & \  \E \int_{t,x,y}  \phi_\eps |u-\tu| .
\end{equs}
Also, 
\begin{equs}
\mathcal{B}^{(\eps, \delta)}_6(u, \tu) = &  \E \int_{t,x,y}   \phi_\eps  \int_u^\tu \eta_\delta''(r-u) \sigma^{ik}_{rx_i}(y,r) \, dr \sigma^{jk}_{x_j}(x,u)
\\
\le & \delta^ \beta\E (1+\|u\|_{L_1(Q_T)}) + \E \int _{t,x,y}  \phi_\eps  \eta_\delta'(\tu-u) \sigma^{ik}_{rx_i}(y,u)  \sigma^{jk}_{x_j}(x,u).
\end{equs}
Similarly, 
\begin{equ}
\mathcal{B}^{(\eps, \delta)}(u, \tu) \le  \delta^\beta \E (1+\|\tu\|_{L_1(Q_T)}) + \E \int_{t,x,y}    \phi_\eps \eta_\delta'(u-\tu) \sigma^{jk}_{x_j}(y,\tu)  \sigma^{ik}_{rx_i}(x,\tu).
\end{equ}
Hence, 
\begin{equs}
\mathcal{B}^{(\eps, \delta)}_6(u, \tu)+\mathcal{B}^{(\eps, \delta)}_8(u, \tu) & \le \delta^\beta \E (1+\|u\|_{L_1(Q_T)}+\|\tu\|_{L_1(Q_T)})
\\                              
& + \E \int_{t,x,y}  \phi_\eps | \sigma^{ik}_{rx_i}(y,u)  \sigma^{jk}_{x_j}(x,u)-\sigma^{jk}_{x_j}(y,\tu)  \sigma^{ik}_{rx_i}(x,\tu)|
\\
&\le (\delta^\beta+\eps+\eps^{\bar \kappa}) \E (1+\|u\|_{L_1(Q_T)}+\|\tu\|_{L_1(Q_T)})
\\                               \label{eq:6}
& + \E \int_{t,x,y}  \phi_\eps | u-\tu|.
\end{equs}
The claim follows by adding \eqref{eq:1}, \eqref{eq:2}, \eqref{eq:3}, \eqref{eq:4}, \eqref{eq:5}, and \eqref{eq:6}.
\end{proof}

\begin{lemma}\label{lem:C}
There exists a constant $N=N(N_0, N_1, d, T)$ such that for all $u, \tu \in L_1(Q_T)$ and all $\eps, \delta \in (0,1)$
\begin{equs}
\mathcal{C}^{(\eps, \delta)}(u, \tu)& \leq N (\eps^{\tilde{\beta}}+ \delta^\beta \eps^{-1})\E(1+\|u\|_{L_1(Q_T)}+\|\tu\|_{L_1(Q_T)})
\\
&+N  \E \int_{t,x,y}\left(\eps \sum_{i=1}^d|\D_{x_i}\phi_\eps| |u-\tu| +\phi_\eps |u-\tu|\right).
\end{equs}
\end{lemma}

\begin{proof}
By Remark \ref{rem:intergation-limits}, \eqref{eq:idendity-eta}, and the relation $ \D_{x_i} \phi_\eps= -\D_{y_i}\phi_\eps$, we get 
\begin{equs}
 &\E \int_{t,x,y}\left(-[f^i_{rx_i}, \delta ](x,u,\tu)\phi_{\varepsilon} -[f^i_r ,\delta](x,u, \tu)\D_{x_i} \phi_{ \varepsilon}\right)  
\\
+ &\E \int_{t,x,y}\left(-[f^i_{rx_i}, \delta ](y,\tu,u)\phi_{ \varepsilon} -[f^i_r ,\delta](y,\tu,u)\D_{y_i} \phi_{ \varepsilon}\right)
\\
=& \E \int_{\tu \leq u}\D_{x_i}\phi_{\eps} \int_\tu^u \int_\tu^u I_{\tr \leq r} \eta^{\prime \prime}_\delta(\tr-r)\left( f^i_r(y,\tr)-f^i_r(x,r) \right) \, d\tr dr 
\\
+ &\E \int_{\tu \geq u}\D_{x_i}\phi_\eps \int_\tu^u \int_\tu^u I_{\tr \geq r} \eta^{\prime \prime}_\delta(\tr-r)\left( f^i_r(y,\tr)-f^i_r(x,r) \right) \, d\tr dr 
\\
- &\E \int_{t,x,y}\phi_\eps\left( \int_\tu^u \eta^{\prime}_\delta(\tu-r) f^i_{rx_i}(x,r) \, dr +\int_\tu^u \eta^{\prime}_\delta(u-\tr) f^i_{rx_i}(y,\tr) \, d\tr  \right) 
\\
\le \ & \E \int_{t,x,y}\left((\eps+\delta^\beta)\sum_{i=1}^d|\D_{x_i}\phi_\eps| |u-\tu| +\phi_\eps |u-\tu|\right) ,
\end{equs}
where for the last inequality we have used \eqref{eq:Holder-b} and  \eqref{eq:Lip-b}. Moreover, we have 
\begin{equs}
&\E \int_{t,x,y}\eta^{\prime}_\delta(u-\tu) \left( f^i_{x_i}(x,u)- f^i_{x_i}(y,\tu)\right) \phi_{\varepsilon} 
\\
\le &\  \eps^{\tilde{\beta}}(1+\E \|u\|_{L_1(Q_T)})+ \E \int_{t,x,y}|u-\tu| \phi_{\varepsilon},
\end{equs}
where we have used \eqref{eq:Lip-b} and \eqref{eq:Holder-bx-x}. Consequently, 
\begin{equs}
\mathcal{C}^{(\eps, \delta)}(u, \tu) &\le (\eps^{\tilde{\beta}}+ \delta^\beta \eps^{-1})\E(1+\|u\|_{L_1(Q_T)}+\|\tu\|_{L_1(Q_T)})
\\
&+ \E \int_{t,x,y}\left(\eps \sum_{i=1}^d|\D_{x_i}\phi_\eps| |u-\tu| +\phi_\eps |u-\tu|\right),
\end{equs}
which finishes the proof.
\end{proof}

We are now ready to proceed with the proof of Theorem \ref{thm:uniqueness}. 

\begin{proof}[Proof of Theorem \ref{thm:uniqueness}]
The majority of the proof is identical for (i) and (ii), so their separation is postponed to the very end.

We apply the entropy inequality \eqref{eq:entropy-inequality} for $u=u(t,x)$ with $\eta_\delta(\cdot-a)$
in place of $\eta$ and $\phi_{\theta,\eps}(\cdot,\cdot,s,y)$ in place of $\phi$,
for some $s\in[0,T]$, $y\in\T^d$, $a\in\bR$.
Assuming that $\theta$ is sufficiently small, one has $\phi_{\theta, \varepsilon}(0, x,s,y)=0$,
and thus we get
\begin{equs}              
-&\int_{t,x}\eta_\delta(u-a) \D_t\phi_{\theta, \varepsilon} 
 \leq 
\int_{t,x}[\fra^2, \delta](u,a) \Delta_x \phi_{\theta, \varepsilon}  
\\
&+\int_{t,x}   \left( [ a^{ij}, \delta](x,u,a) \D_{x_ix_j} \phi_{\theta, \varepsilon}  + \left( [a^{ij}_{x_j},\delta ](x,u,a) -\eta'_\delta(u-a) b^i(x,u)\right) \D_{x_i} \phi_{\theta, \varepsilon}\right)  
\\
&+ \int_{t,x}\left(\eta'(u-a) f^i_{x_i}(x,u)\phi_{\theta, \varepsilon}-[f^i_{rx_i}, \delta ](x,u,a)\phi_{\theta, \varepsilon} -[f^i_r, \delta ](x,u,a)\D_{x_i}\phi_{\theta, \varepsilon}\right)  
\\
&+\int_{t,x} \left(\frac{1}{2}\int_{t,x} \eta_\delta^{\prime \prime}(u-a)\sum_k |  \sigma^{ik}_{x_i}(x,u) |^2 \phi_{\theta, \varepsilon} -\eta_\delta^{\prime\prime}(u-a) | \nabla_x [\fra](u)|^2\phi_{\theta, \varepsilon}  
\right) 
\\   
& +\int_0^T \int_x\left(\eta^\prime(u-a)\phi  \sigma^{ik}_{x_i}(x,u) -  [\sigma^{ik}_{rx_i},\delta](x,u,a) \phi_{\theta, \varepsilon}- [\sigma^{ik}_r,\delta](x,u,a) \D_{x_i} \phi_{\theta, \varepsilon}\right)   \, d\beta^k(t),
\\
\label{eq:ineq1}
\end{equs}
Notice that all the expressions in \eqref{eq:ineq1} are continuous in $(a,s,y)$. 
We now substitute $a= \tilde{u}(s,y)$, integrate over $(s,y)$, and take expectations.
For the last term in \eqref{eq:ineq1} this is justified by \eqref{eq:simple}. All of the other terms are continuous in $a$ and can be bounded by $N(|a|^m+X)$ with some constant $N$ and some integrable random variable $X$ (recall \eqref{eq:as fra}), so that  substituting $a= \tilde{u}(s,y)$ and integrating out $s$, $y$, and $\omega$, results in finite quantities.

After writing the analogous inequality with the roles of $u,t,x$ and $\tu,s,y$ reversed,  using the symmetry of $\eta_\delta$, and adding both inequalities, one arrives at
\begin{equs}              
&-\E \int_{t,x,s,y}\eta_\delta(u-\tu) (\D_t\phi_{\theta, \varepsilon}+\D_s\phi_{\theta, \varepsilon} ) 
\\
& \leq 
\E \int_{t,x,s,y} ([\fra^2,\delta](u,\tu) \Delta_x \phi_{\theta, \varepsilon} + [\tfra^2,\delta](\tilde{u},u) \Delta_y \phi_{\theta, \varepsilon})
\\
&+ \E \int_{t,x,s,y}   \left( [ a^{ij}, \delta](x,u,\tu) \D_{x_ix_j} \phi_{\theta, \varepsilon}  + \left( [a^{ij}_{x_j},\delta ](x,u,\tu) -\eta^{\prime}_\delta(u-\tu) b^i(x,u)\right) \D_{x_i} \phi_{\theta, \varepsilon}\right)  
\\
&+ \E \int_{t,x,s,y}  \left( [ a^{ij}, \delta](y,\tu,u) \D_{y_iy_j} \phi_{\theta, \varepsilon}  + \left( [a^{ij}_{x_j},\delta ](y,\tu,u) -\eta^{\prime}_\delta(\tu-u) b^i(y,\tu)\right) \D_{y_i} \phi_{\theta, \varepsilon}\right)  
\\
&+ \E \int_{t,x,s,y}\left(\eta^{\prime}_\delta(u-\tu) f^i_{x_i}(x,u)\phi_{\theta, \varepsilon}-[f^i_{rx_i}, \delta ](x,u,\tu)\phi_{\theta, \varepsilon} -[f^i_r ,\delta](x,u, \tu)\D_{x_i} \phi_{\theta, \varepsilon}\right)  
\\
&+ \E \int_{t,x,s,y}\left(\eta^{\prime}_\delta(\tu-u) f^i_{x_i}(y,\tu)\phi_{\theta, \varepsilon}-[f^i_{rx_i}, \delta ](y,\tu,u)\phi_{\theta, \varepsilon} -[f^i_r ,\delta](y,\tu,u)\D_{y_i} \phi_{\theta, \varepsilon}\right)  
\\
&+\E \int_{t,x,s,y} \left(\frac{1}{2} \eta_\delta^{\prime \prime}(u-\tu)\sum_k | \sigma^{ik}_{x_i}(x,u) |^2 \phi_{\theta, \varepsilon} -\eta_\delta^{\prime\prime}(u-\tu) | \nabla_x [\fra](u)|^2\phi_{\theta, \varepsilon}  
\right) 
\\
&+\E \int_{t,x,s,y} \left(\frac{1}{2} \eta_\delta^{\prime \prime}(u-\tu)\sum_k |  \sigma^{ik}_{y_i}(y,\tu) |^2 \phi_{\theta, \varepsilon} -\eta_\delta^{\prime\prime}(u-\tu) | \nabla_y [\tfra](\tu)|^2\phi_{\theta, \varepsilon}  
\right) 
\\   \label{eq:F1F2}
& +\E \int_{s,y} F^1_\theta(s,y)  +\E \int_{t,x} F^2_\theta(t,x),
\end{equs}
where $u=u(t,x)$, $\tu =\tu(s,y)$, $\phi_{\theta, \eps}= \phi_{\theta, \eps}(t,x,s,y)$, and 
\begin{equs}
F^1_\theta(s,y):= \left[ \int_0^T \int_x\left(\eta^\prime_\delta(u-a)\phi_{
\theta, \eps}  \sigma^{ik}_{x_i}(x,u)-  [\sigma^{ik}_{rx_i},\delta](x,u,a) \phi_{\theta, \varepsilon} \right. \right.
\\
-\left. \left. [\sigma^{ik}_r,\delta](x,u,a) \D_{x_i} \phi_{\theta, \varepsilon}\right)   \, d\beta^k(t)\right]_{a=\tu(s,y)},
\\
F^2_\theta(t,x):=\left[ \int_0^T \int_y\left(\eta^\prime_\delta(\tu-a)\phi_{\theta, \eps}  \sigma^{ik}_{y_i}(y,\tu)- [\sigma^{ik}_{rx_i},\delta](y,\tu,a) \phi_{\theta, \varepsilon}\right. \right.
\\
- \left. \left.[\sigma^{ik}_r,\delta](y,\tu,a) \D_{x_i} \phi_{\theta, \varepsilon}\right)   \, d\beta^k(s)\right]_{a=u(t,x)}.
\end{equs} 
For the term containing $F^1_\theta$  at the right hand side of \eqref{eq:F1F2} we have the following: $\D_{x_i} \phi_{\theta, \varepsilon}$ is supported on $[s,s+\theta]$, hence the integration in $t$ is over $[s, (s+\theta)\wedge T]$. Then we plug in a quantity with is $\mathcal{F}_s$-measurable. Therefore, this term vanishes in expectation (a rigorous justification follows from a limiting procedure similar to \eqref{eq:0lambda limit}).
 We now pass to the $\theta\to0$ limit. For this, we use \cite[Proposition 3.5, see also p.15]{DareiotisGerencserGess} and the ($\star$)-property with $h= \eta'$ and $\varrho= \varrho_\eps$ to get 
\begin{equs}              
-\E \int_{t,x,y}\eta_\delta(u-\tu) \D_t\phi_\varepsilon 
 & \leq 
\E \int_{t,x,s,y} ([\fra^2,\delta](u,\tu) \Delta_x \phi_{\theta, \varepsilon} + [\tfra^2,\delta](\tilde{u},u) \Delta_y \phi_{\varepsilon})
\\
&+ \mathcal{A}^{(\eps, \delta)}(u, \tu )+ \mathcal{C}^{(\eps, \delta)}(u, \tu )
\\
&+\E \int_{t,x,y} \left(\frac{1}{2} \eta_\delta^{\prime \prime}(u-\tu)\sum_k | \sigma^{ik}_{x_i}(x,u) |^2 \phi_{ \varepsilon} -\eta_\delta^{\prime\prime}(u-\tu) | \nabla_x [\fra](u)|^2\phi_{\varepsilon}  
\right) 
\\
&+\E \int_{t,x,y} \left(\frac{1}{2} \eta_\delta^{\prime \prime}(u-\tu)\sum_k |  \sigma^{ik}_{y_i}(y,\tu) |^2 \phi_{ \varepsilon} -\eta_\delta^{\prime\prime}(u-\tu) | \nabla_y [\tfra](\tu)|^2\phi_{ \varepsilon}  
\right) 
\\          \label{eq:added-entropies}
&+ \mathcal{B}^{(\eps, \delta)}(u, \tu),
\end{equs}

   Notice that   that by \eqref{eq:regularity-sigma-2} and \eqref{eq:regularity-sigma-5} we have that for all $x,y \in \bT^d$ and $r, \tr \in \bR$ 
\begin{equs}
|\sigma^i_{x_i}(x,r)- \sigma^i_{x_i}(y,\tr)|_{l_2} \leq N |r-\tr|+N(1+|r|)|x-y|^{\bar \kappa},
 \end{equs}
 where $N$ depends only on $N_0, N_1$, and $d$.  
  Under this condition and under Assumption \ref{as:A} \eqref{as:A first} it is shown in \cite[Theorem 4.1, p.13-15, see (4.8) and (4.18) therein]{DareiotisGerencserGess} that for all $ \alpha \in (0, 1 \wedge (m/2))$ we have 
\begin{equs}
& \mathcal{B}^{(\eps, \delta)}_9(u,\tu) +\E \int_{t,x,s,y} ([\fra^2,\delta](u,\tu) \Delta_x \phi_{\theta, \varepsilon} + [\tfra^2,\delta](\tilde{u},u) \Delta_y \phi_{\varepsilon})
\\
+&\E  \int_{t,x,y} \left(\frac{1}{2} \eta_\delta^{\prime \prime}(u-\tu)\sum_k | \sigma^{ik}_{x_i}(x,u) |^2 \phi_\varepsilon -\eta_\delta^{\prime\prime}(u-\tu) | \nabla_x [\fra](u)|^2\phi_\varepsilon 
\right) 
\\
+&\E  \int_{t,x,y} \left(\frac{1}{2} \eta_\delta^{\prime \prime}(u-\tu)\sum_k |  \sigma^{ik}_{y_i}(y,\tu) |^2 \phi_\varepsilon -\eta_\delta^{\prime\prime}(u-\tu) | \nabla_y[\tfra](\tu)|^2\phi_\varepsilon 
\right)
\\
\le& \big(\delta+\eps^{2\bar\kappa}\delta^{-1}+\eps^{-2}\delta^{2\alpha}+\eps^{-2}\lambda^2)
\E(1+\|u\|^{m+1}_{L_{m+1}(Q_T)}+\|\tu\|_{L_{m+1}(Q_T)}^{m+1})
\\         \label{eq:1}
+&
\eps^{-2}\big(\E\|I_{|u|\geq R_\lambda}(1+|u|)\|_{L_m(Q_T)}^m+
\E\|I_{|\tu|\geq R_\lambda}(1+|\tu|)\|_{L_m(Q_T)}^m\big).
\end{equs}

Hence, by the above inequality, Lemma \ref{lem:A+B}, and Lemma \ref{lem:C},  we obtain for all $\eps, \delta \in (0,1)$
\begin{equs}
-&\E \int_{t,x,y}\eta_\delta(u-\tu) \D_t\phi_\varepsilon 
\\
 \le &   \  C(\eps, \delta, \lambda)\E(1+\|u\|^{m+1}_{L_{m+1}(Q_T)}+\|\tu\|_{L_{m+1}(Q_T)}^{m+1})
\\         
+ &\eps^{-2}\big(\E\|I_{|u|\geq R_\lambda}(1+|u|)\|_{L_m(Q_T)}^m+
\E\|I_{|\tu|\geq R_\lambda}(1+|\tu|)\|_{L_m(Q_T)}^m\big)
\\
+&   \ \E \int_{t,x,y} \varepsilon^ 2\sum_{ij}| \D_{x_iy_j}\phi_\eps|  |u-\tilde{u}|
+ \E \int_{t,x,y} \varepsilon \sum_{i}|\D_{x_i}\phi_\eps| |u-\tilde{u}|+  \E \int_{t,x,y} \phi_\eps |u-\tu|,
\end{equs}
with 
$$
 C(\eps, \delta, \lambda):=\big(\delta^\beta + \delta^{2\beta} \eps^{-2}+\delta^\beta \eps^{-1}+\eps^{2\bar\kappa}\delta^{-1}+\eps^{-2}\delta^{2\alpha}+\eps^{-2}\lambda^2+\eps^{\tilde{\beta}}+\eps^{\bar \kappa}),
$$
which by virtue of
\begin{equ}
\left| \E \int_{t,x,y} \eta_\delta(u-\tu) \D_t \phi_\eps -  \E \int_{t,x,y} |u-\tu| \D_t \phi_\eps\right| \le \delta,
\end{equ}
gives 
\begin{equs}
-&\E \int_{t,x,y}|u-\tu| \varrho_\epsilon \D_t\varphi
\\
& \le C(\eps, \delta, \lambda)\E(1+\|u\|^{m+1}_{L_{m+1}(Q_T)}+\|\tu\|_{L_{m+1}(Q_T)}^{m+1})
\\         
&+ \eps^{-2}\big(\E\|I_{|u|\geq R_\lambda}(1+|u|)\|_{L_m(Q_T)}^m+
\E\|I_{|\tu|\geq R_\lambda}(1+|\tu|)\|_{L_m(Q_T)}^m\big)
\\
&+   \ \E \int_{t,x,y} \varepsilon^ 2\sum_{ij}| \D_{x_iy_j}\varrho_\eps |  \varphi |u-\tilde{u}|
+ \E \int_{t,x,y} \varepsilon \sum_{i}| \D_{x_i}\varrho_\eps |  \varphi  |u-\tilde{u}|
\\       \label{eq:before eps limit}
&+  \E \int_{t,x,y} \varrho_\eps \varphi |u-\tu|.
\end{equs}
Let  $s, t \in (0,T)$, with $s < t$, be Lebesgue points of the function
$$
t \mapsto \E \int_{x,y}|u(t,x)-\tilde{u}(t,y)|\varrho_\varepsilon(x-y),
$$
and fix some $\gamma >0$ such that $\gamma< t-s$ and $t+ \gamma<T$. 
We now make use of the freedom of choosing $\varphi$:
choose in \eqref{eq:before eps limit} $\varphi =\varphi_n \in C^\infty_c((0,T))$ 
obeying the bound $\|\varphi_n\|_{L_\infty([0,T])}\vee\|\partial_t\varphi_n\|_{L_1([0,T])}\leq 1$,
such that 
$$
\lim_{n \to \infty} \|\varphi_n-\zeta\|_{{W^1_2}((0,T))} = 0,  
$$
where 
$\zeta : [0,T] \to \bR$ is such that $\zeta(0)=0$ and $\zeta'=\gamma^{-1}I_{s,s+\gamma}-\gamma^{-1}I_{t,t+\gamma}$.
After letting $n \to \infty$ we obtain 
\begin{equs}
\frac{1}{\gamma}&\E\int_t^{t+ \gamma} \int_{x,y}|u(r,x)-\tilde{u}(r,y)|\varrho_\varepsilon(x-y) \, dr 
\\
&- \frac{1}{\gamma}\E\int_s^{s+ \gamma} \int_{x,y}|u(r,x)-\tilde{u}(r,y)|\varrho_\varepsilon(x-y) \,  dr \\
& \le C(\eps, \delta, \lambda)\E(1+\|u\|^{m+1}_{L_{m+1}(Q_T)}+\|\tu\|_{L_{m+1}(Q_T)}^{m+1})
\\         
&+ \eps^{-2}\big(\E\|I_{|u|\geq R_\lambda}(1+|u|)\|_{L_m(Q_T)}^m+
\E\|I_{|\tu|\geq R_\lambda}(1+|\tu|)\|_{L_m(Q_T)}^m\big)
\\
&+   \ \E \int_0 ^{t+\gamma} \int_{x,y} \varepsilon^ 2 \sum_{ij}| \D_{x_iy_j}\varrho_\epsilon(x-y) |  |u(s,x)-\tilde{u}(s,y)| \, ds
\\            
& + \E  \int_0 ^{t+\gamma} \int_{x,y} \varepsilon \sum_{i}| \D_{x_i}\varrho_\epsilon(x-y)  |   |u(s,x)-\tilde{u}(s,y)|+ | \varrho_\epsilon(x-y)  |   |u(s,x)-\tilde{u}(s,y)| \, ds,
\\
\label{eq:with-gamma}
\end{equs}
which, after letting $\gamma \downarrow 0$,  gives
\begin{equs}
\E & \int_{x,y}|u(t,x)-\tilde{u}(t,y)|\varrho_\varepsilon(x-y) 
- \E \int_{x,y}|u(s,x)-\tilde{u}(s,y)|\varrho_\varepsilon(x-y)
 \le M,
\end{equs}
where $M$ is the right hand side of \eqref{eq:with-gamma} with $\gamma=0$
Notice that the above inequality holds for almost all $s \leq t$. After averaging over $s \in (0,\gamma)$ for some $\gamma>0$ we obtain
\begin{equs}
\E & \int_{x,y}|u(t,x)-\tilde{u}(t,y)|\varrho_\varepsilon(x-y) 
\\
&\leq M+ \frac{1}{\gamma}\E\int_0^{\gamma} \int_{x,y}|u(s,x)-\tilde{u}(s,y)|\varrho_\varepsilon(x-y)\,ds.
\end{equs}
Letting $\gamma \to 0$, we obtain by virtue of Lemma \ref{lem:initial-condition},
\begin{equs}
\E & \int_{x,y}|u(t,x)-\tilde{u}(t,y)|\varrho_\varepsilon(x-y) 
\leq M+ \E \int_{x,y}|\xi(x)-\txi(y)|\varrho_\varepsilon(x-y).\label{eq:fork}
\end{equs}

 We now prove (ii).  We integrate \eqref{eq:fork} over $t \in (0,s)$  for some $s \leq T$ and we get 
 \begin{equs}
\E & \int_0^s \int_{x,y}|u(t,x)-\tilde{u}(t,y)|\varrho_\varepsilon(x-y) \, dt 
\\             
&\le  T\E \int_{x}|\xi(x)-\txi(x)|+
T \sup_{|h|\leq\eps}\E\|\txi(\cdot)-\txi(\cdot+h)\|_{L_1(\T^d)}
\\
&+T C(\eps, \delta, \lambda)\E(1+\|u\|^{m+1}_{L_{m+1}(Q_T)}+\|\tu\|_{L_{m+1}(Q_T)}^{m+1})
\\         
&+ T \eps^{-2}\big(\E\|I_{|u|\geq R_\lambda}(1+|u|)\|_{L_m(Q_T)}^m+
\E\|I_{|\tu|\geq R_\lambda}(1+|\tu|)\|_{L_m(Q_T)}^m\big)
\\
&+   \ \E \int_0 ^s \int_0 ^{t} \int_{x,y} \varepsilon^ 2 \sum_{ij}| \D_{x_iy_j}\varrho_\epsilon(x-y) |  |u(\zeta,x)-\tilde{u}(\zeta,y)| \, d\zeta dt 
\\            
& + \E   \int_0 ^s \int_0 ^{t} \int_{x,y} \varepsilon \sum_{i}| \D_{x_i}\varrho_\epsilon(x-y)  |   |u(\zeta,x)-\tilde{u}(\zeta,y)|+ | \varrho_\epsilon(x-y)  |   |u(\zeta,x)-\tilde{u}(\zeta,y)| \, d\zeta dt.
\\
\label{eq:before-gronwal}
\end{equs}

Then, notice that for an approximation of the identity $\varrho_\eps$ we have 
\begin{equs}
\Big|\E & \int_{t,x}|u(t,x)-\tilde{u}(t,x)| -\E \int_{t,x,y}|u(t,x)-\tilde{u}(t,y)|\varrho_\varepsilon(x-y) \Big|
\\
&\leq
\E \int_{t,x,y}|u (t,x)-u(t,y)|\varrho_\varepsilon(x-y).
\end{equs}
Moreover, notice that $\eps |\D_{x_i} \varrho_\eps|$ and $\eps^2 |\D_{x_ix_j} \varrho_\eps|$ are also approximations of the identity (up to a constant). From these observations, we obtain by virtue of \eqref{eq:before-gronwal} and Lemma \ref{lem:frac reg}
\begin{equs}
\E & \int_0^s \int_{x}|u(t,x)-\tilde{u}(t,x)| \, dt 
\\             
&\le  \E \int_{x}|\xi(x)-\txi(x)|+
\sup_{|h|\leq\eps}\E\|\txi(\cdot)-\txi(\cdot+h)\|_{L_1(\T^d)}
\\
&+ C(\eps, \delta, \lambda)\E(1+\|u\|^{m+1}_{L_{m+1}(Q_T)}+\|\tu\|_{L_{m+1}(Q_T)}^{m+1})
\\         
&+  \eps^{-2}\big(\E\|I_{|u|\geq R_\lambda}(1+|u|)\|_{L_m(Q_T)}^m+
\E\|I_{|\tu|\geq R_\lambda}(1+|\tu|)\|_{L_m(Q_T)}^m\big)
\\
& +\eps^{\frac{2}{m+1}}\big(1+\E\|\nabla[\fra](u)\|_{L_1(Q_T)}\big)
 \\
 &+   \ \E \int_0 ^s \int_0 ^{t} \int_{x}   |u(\zeta,x)-\tilde{u}(\zeta,x)| \, d\zeta dt.
\end{equs}
Gronwall's lemma leads to (ii). In order to prove  (i), we choose in \eqref{eq:fork} $\lambda=0$ and  $R_\lambda=\infty$ (recall the definition of $M$) to obtain
  \begin{equs}
\E & \int_{x,y}|u(t,x)-\tilde{u}(t,y)|\varrho_\varepsilon(x-y) 
\\
 & \le\E \int_{x,y}|\xi(x)-\txi(y)|\varrho_\varepsilon(x-y)
 \\
&+ C(\eps, \delta)\E(1+\|u\|^{m+1}_{L_{m+1}(Q_T)}+\|\tu\|_{L_{m+1}(Q_T)}^{m+1})
\\         
&+   \ \E \int_0 ^t \int_{x,y} \varepsilon^ 2 \sum_{ij}| \D_{x_iy_j}\varrho_\epsilon(x-y) |  |u(s,x)-\tilde{u}(s,y)| \, ds
\\          
& + \E  \int_0 ^t \int_{x,y} \varepsilon \sum_{i} | \D_{x_i}\varrho_\epsilon(x-y)  |   |u(s,x)-\tilde{u}(s,y)|+  \varrho_\epsilon(x-y)    |u(s,x)-\tilde{u}(s,y)| \, ds,
\\
 \label{eq:ready-to-go}    
\end{equs}  
  with

\begin{equ}
C(\eps, \delta)=\big(\delta^\beta + \delta^{2\beta} \eps^{-2}+\delta^\beta \eps^{-1}+\eps^{2\bar\kappa}\delta^{-1}+\eps^{-2}\delta^{2\alpha}+\eps^{\tilde{\beta}}+\eps^{\bar \kappa}).
\end{equ}

We now choose
$\nu \in ((m\wedge2)^{-1},\bar\kappa)$ such that $2 \beta \nu >1$ (recall that $\beta \in ( 2 \bar \kappa)^{-1},1]$) and  $\alpha<1\wedge(m/2)$ such that
$-2+(2\alpha)(2\nu)>0$. Setting $\delta= \eps^{2\nu}$ then yields $C(\eps, \delta) \to 0$ as $\eps \to 0$. Consequently, by letting $\eps \to 0$ in \eqref{eq:ready-to-go} and using the continuity of translations in $L_1$ we obtain
\begin{equs}
\E& \|u(t)-\tilde{u}(t)\|_{L_1(\bT^d)}
\le \E \|\xi-\tilde{\xi}\|_{L_1(\bT^d)} + \int_0^t \E \|u(s)-\tilde{u}(s)\|_{L_1(\bT^d)} \, ds.
\end{equs}
The above relation holds for almost all $t \in [0,T]$. Hence, \eqref{eq:L_1 contraction} follows by Gronwall's lemma.

\end{proof}

\section{Approximations}

In Section \ref{Stability} we showed that if we have two entropy solutions of equation \eqref{eq:main-equation} with the same initial condition, then they coincide  provided that one of them satisfies the $(\star)$-property. Hence, in order to conclude the existence and uniqueness of entropy solutions, it suffices to show the existence of an entropy solution possessing the $(\star)$-property. 
To do so,  we use a vanishing viscosity approximation. In order to prove the strong (probabilistically) existence of solutions for the approximating equations, we use a technique from \cite{Istvan}, where a characterization of the convergence in probability  is used to show that weak existence combined with strong uniqueness implies strong existence. This has been used in the past in the context of SPDEs (see \cite{HOF2,GH18} and the references therein).
For the proof of the following Proposition see \cite[Proposition 5.1]{DareiotisGerencserGess}.

\begin{proposition}            \label{prop:Phi-n}
Let $\Phi$ satisfy Assumption \ref{as:A} (\ref{as:A first}) with a constant $K\geq 1$.
Then, for all $n$  there exists an increasing function $\Phi_n\in C^\infty(\bR)$ with bounded derivatives,
satisfying Assumption \ref{as:A} (\ref{as:A first}) with constant $3K$,
such that $\fra_n(r)\geq 2/n$, and 
\begin{equ}\label{eq:approx A}
\sup_{|r|\leq n}|\fra(r)-\fra_n(r)|\leq 4/n.
\end{equ}
\end{proposition}

Let $\Phi_n$ be  as above and set 
\begin{equs}           \label{def:xi-n}
\xi_n:=(-n) \vee(\xi \wedge n).
\end{equs}

\begin{definition}
 An $L_2$-solution of equation $\Pi(\Phi_n, \xi_n)$
is a  continuous $L_2(\T^d)$-valued process $u_n$, such that 
$u_n\in L_2(\Omega_T, W^1_2(\T^d))$, $\nabla \Phi_n(u_n)\in L_2(\Omega_T, L_2(\T^d))$, and the equality
\begin{equs}
(u_n(t,\cdot),\phi)=(\xi_n,\phi)
&-\int_0^t(\nabla \Phi_n(u_n(s,\cdot)),\nabla \phi)+(a^{ij}(u)\D_{x_j}u+ b^i(u)+f^i(u), \D_{x_i} \phi) \,ds
\\
&-\int_0^t(\sigma^k(\cdot,u_n(s,\cdot)),\nabla \phi) \,d\beta^k_s\,
\end{equs}
holds for all $\phi\in C^\infty(\T^d)$, almost surely for all $t\in[0,T]$.
\end{definition}
If $u_n$ is an $L_2$-solution of $\Pi(\Phi_n, \xi_n)$, then the following estimates hold (see Lemma \ref{lem:Ap} in the Appendix)
\begin{equs}         
\label{eq:change1}     
\E \sup_{t \leq T} \| u_n\|_{L_2(\bT^d)}^p + \E \|\nabla [\fra_n](u_n) \|_{L_2(Q_T)}^p &\leq N ( 1+ \E \|\xi_n\|_{L_2(\bT^d)}^p),
\\         
\label{eq:change2}           
\E \sup_{t \leq T} \|u_n\|_{L_{m+1}(\bT^d)}^{m+1}+ \E \| \nabla \Phi_n(u_n)\|_{L_2(Q_T)}^2 &\leq N(1+\E \|\xi_n\|_{L_{m+1}(\bT^d)}^{m+1}),
\end{equs}
where the constant $N$ depends only on $N_0, N_1, K, T, d,p$ and $m$ (but not on $n\in \bN$). Notice that $|\xi_n|$ is bounded by $n$, which implies that the right hand side of the above inequalities is finite. Moreover, by construction of $\xi_n$ one concludes that for all $p\geq 2$
\begin{equs}\label{eq:uniform}
\E \sup_{t \leq T} \| u_n\|_{L_2(\bT^d)}^p + \E \|\nabla [\fra_n](u_n) \|_{L_2(Q_T)}^p &\leq N ( 1+ \E \|\xi\|_{L_2(\bT^d)}^p),
\\
\label{eq:uniform-m+1}
\E \sup_{t \leq T} \|u_n\|_{L_{m+1}(\bT^d)}^{m+1}+ \E \| \nabla \Phi_n(u_n)\|_{L_2(Q_T)}^2 &\leq N(1+\E \|\xi\|_{L_{m+1}(\bT^d)}^{m+1}).
\end{equs}
with $N$ depending only on $N_0, N_1,K, T, d,p$ and $m$.
Finally, since $\fra_n\geq 2/n>0$, we have $|\nabla u_n|\leq N(n)|\nabla[\fra_n](u_n) |$, and so by \eqref{eq:uniform}, we have the ($n$-dependent) bound
\begin{equ}\label{eq:gradient}
\E\|\nabla u_n\|_{L_2(Q_T)}^p<\infty.
\end{equ}

\begin{lemma}                  \label{lem:strong entropy}
For each $n \in \bN$, let $u_n$ be an $L_2$-solution of $\Pi(\Phi_n, \xi_n)$. Then, $u_n$ has the $(\star)$-property. If in addition $\|\xi\|_{L_2(\bT^d)}$ has moments of order 4, then the constant $N$ in \eqref{eq:strong-entropy} is independent of $n$. 
\end{lemma}
\begin{proof}
Fix $\theta>0$ small enough so that \eqref{eq:F rewrite} holds. To ease notation we drop the lower index in $F_\theta$.
We proceed by two approximations:  first, as in Corollary \ref{cor} (ii),  the substitution of $u_n(t,x)$ into
$F(t,x,\cdot)$ is smoothed, and second,  $u_n$ is regularised. 

For a function $f \in L_2(\bT^d)$ let  $f^{(\gamma)}:=(\rho_\gamma)^{\otimes d}\ast f$ denote   its mollification. 
Then,  $u_n^{(\gamma)}$ satisfies (pointwise) the equation
\begin{equs}
du_n^{(\gamma)}&=\Delta (\Phi_n(u_n))^{(\gamma)} + \D_{x_i}(a^{i,j}(u_n) \D_{x_j} u_n+ b^i(u_n) +f^i(u_n))^{(\gamma)}\,dt
\\    \label{eq:viscous smoothed equation}
&+\D_{x_i} (\sigma^{ik}(u_n))^{(\gamma)}\,d\beta^k(t).
\end{equs}
We note that 
\begin{equs}
& \Big|\E\int_{t,x,a} F(t,x, a) \rho_\lambda(u_n(t,x)-a)- \E\int_{t,x,a} F(t,x, a) \rho_\lambda(u_n^{(\gamma)}(t,x)-a)\Big|
\\
&=\Big|\E\int_{t,x,a}\big(F(t,x,a)-F(t,x,a+u_n^{(\gamma)}(t,x)-u_n(t,x))\big)\rho_\lambda(u_n(t,x)-a)\Big|
\\
&\leq N  \left( \E \|u_n-u_n^{(\gamma)} \|_{L_1(Q_T)}^2\right)^{1/2}  \,  \left( \E\|\D_a F\|_{L_\infty(Q_T\times\bR)}^2
\right)^{1/2}  \to 0, \label{eq:0gamma limit}
\end{equs}
as $\gamma \to 0$.
By \eqref{eq:F rewrite} we have  $\E F(t,x,a)X=0$ for any $\cF_{t-\theta}$-measurable bounded random variable $X$. Hence, 
\begin{equs}
\E &F(t,x, a) \rho_\lambda(u_n^{(\gamma)}(t,x)-a)
\\
&=\E F(t,x, a) [\rho_\lambda(u_n^{(\gamma)}(t,x)-a)-\rho_\lambda(u_n^{(\gamma)}(t-\theta,x)-a)].
\end{equs}
 By \eqref{eq:viscous smoothed equation} and It\^o's formula one has
\begin{equs}
&\int_{t,x,a} F(t,x, a) \big(\rho_\lambda (u_n^{(\gamma)} (t,x)-a)-\rho_\lambda(u_n^{(\gamma)}(t-\theta,x)-a)\big)
\\
&=
\int_{t,x,a} F(t,x, a) \int_{t-\theta}^t \rho^{\prime}_\lambda  (u_n^{(\gamma)} (s,x)-a)
 \Delta (\Phi_n(u_n))^{(\gamma)} \, ds 
 \\
 &+\int_{t,x,a} F(t,x, a) \int_{t-\theta}^t \rho^{\prime}_\lambda  (u_n^{(\gamma)} (s,x)-a) \D_{x_i}(a^{ij}(x,u_n) \D_{x_j} u_n(s,x)+ b^i(x,u_n))^{(\gamma)}
\, ds 
\\
&+
\int_{t,x,a} F(t,x, a) \int_{t-\theta}^t \rho^{\prime}_\lambda  (u_n^{(\gamma)} (s,x)-a) \D_{x_i}(\sigma^{ik}(x,u_n))^{(\gamma)} \, d\beta^k(s) 
\\
&+
\int_{t,x,a} F(t,x, a) \frac{1}{2}\int_{t-\theta}^t \rho^{\prime\prime}_\lambda  (u_n^{(\gamma)} (s,x)-a) \sum_{k=1}^\infty|\D_{x_i}(\sigma^{ik}(x,u_n))^{(\gamma)}|^2 \, ds 
\\
 &+\int_{t,x,a} F(t,x, a) \int_{t-\theta}^t \rho^{\prime}_\lambda  (u_n^{(\gamma)} (s,x)-a) (\D_{x_i}f^i(x,u_n))^{(\gamma)}
\, ds 
\\&=: C^{(1)}_{\lambda,\gamma}+C^{(2)}_{\lambda,\gamma}+C^{(3)}_{\lambda,\gamma}+C^{(4)}_{\lambda,\gamma}+C^{(5)}_{\lambda,\gamma}.        \label{eq:after Ito}
\end{equs}
By \eqref{eq:F rewrite} and integration by parts (in $x$) we have 
\begin{equs}
-C^{(1)}_{\lambda,\gamma}&= \int_{t,x,a}I_{t> \theta}\int_{t-\theta}^t  \nabla_x F(t,x, a)\rho^{\prime}_\lambda  (u_n^{(\gamma)} (s,x)-a) \cdot \nabla(\Phi_n(u_n))^{(\gamma)}
\\
&\quad+ F(t,x, a)\rho^{\prime\prime}_\lambda  (u_n^{(\gamma)} (s,x)-a)  \nabla u_n^{(\gamma)}(s,x) 
\cdot \nabla(\Phi_n(u_n))^{(\gamma)}\, ds 
\\
& =: C^{(11)}_{\lambda,\gamma}+C^{(12)}_{\lambda,\gamma} .
\end{equs}
After integration by parts with respect to $a$, by the Cauchy-Schwarz inequality, inequalities \eqref{eq:whole-theta}, \eqref{eq:change2} and Lemma \ref{lem:F}, we have
\begin{equs}
\E|C^{(11)}_{\lambda,\gamma}|
&=\E\big|\int_{t,x,a}I_{t> \theta}\int_{t-\theta}^t\nabla_x\partial_a F(t,x,a)\rho_\lambda(u_n^{(\gamma)}(s,x)-a) \cdot\nabla(\Phi_n(u_n))^{(\gamma)}\,ds\big|
\\
&\leq N\theta \left( \E\|\nabla_x\partial_a F\|^2_{L_\infty(Q_T\times\bR)} \right) ^{1/2}
\left( \E\| \nabla \Phi_n(u_n)\|^2_{L_1(Q_T)} \right) ^{1/2}
\\
&\leq N(n)\theta^{1-\mu}.                    \label{eq:A1}
\end{equs}
Similarly, this time integrating by parts twice in $a$ we have for all sufficiently small $\theta \in (0,1)$ 
\begin{equs}
\E&|C^{(12)}_{\lambda,\gamma}|
\leq
N\theta^{1-\mu}
\left( \E
\| \nabla u_n^{(\gamma)} \cdot \nabla (\Phi_n(u_n))^{(\gamma)}\|_{L_1(Q_T)}^{\frac{m+1}{m}} \right)^{\frac{m}{m+1}}.
\end{equs}
To bound the right-hand side, note that by \eqref{eq:gradient},
$\nabla u_n^{(\gamma)}\to\nabla u_n$ in $L_p(\Omega;L_2(Q_T))$, for any $p$,
and by \eqref{eq:uniform-m+1},
$\nabla (\Phi_n(u_n))^{(\gamma)}\to\nabla \Phi_n(u_n)$
 in $L_2(\Omega;L_2(Q_T))$. Therefore, by \eqref{eq:change1}
\begin{equs}
\lim_{\gamma\to 0}\E
\| \nabla u_n^{(\gamma)}\cdot\nabla (\Phi_n(u_n))^{(\gamma)}\|_{L_1(Q_T)}^{\frac{m+1}{m}}
&=\E\| \nabla u_n \cdot \nabla \Phi_n(u_n)\|_{L_1(Q_T)}^{\frac{m+1}{m}}
\\& 
=\E\|\nabla [\fra_n](u_n) \|_{L_2(Q_T)}^{\frac{2(m+1)}{m}}\leq N(n). \label{eq:A2}
\end{equs}
Together with \eqref{eq:A1}, we therefore get
\begin{equ}\label{eq:A}
\limsup_{\gamma\to 0}\E|C^{(1)}_{\lambda,\gamma}|\leq N(n) \theta^{1-\mu}.
\end{equ}
We now estimate $C^{(2)}_{\lambda,\gamma}+C^{(4)}_{\lambda,\gamma}$. After integrating by parts in $x$  we have 
\begin{equs}
&C^{(2)}_{\lambda,\gamma}+C^{(4)}_{\lambda,\gamma} =
\\
&-\int_{t,x,a} \D_{x_i}F(t,x, a) \int_{t-\theta}^t \rho^{\prime}_\lambda  (u_n^{(\gamma)} (s,x)-a) (a^{ij}(x,u_n) \D_{x_j} u_n(s,x)+ b^i(x,u_n)+f^i(x,u_n) )^{(\gamma)}
\, ds 
\\
&+
\int_{t,x,a} F(t,x, a) \frac{1}{2}\int_{t-\theta}^t \rho^{\prime\prime}_\lambda  (u_n^{(\gamma)} (s,x)-a) \sum_{k=1}^\infty|(\sigma^{ik}_{x_i}(x,u_n))^{(\gamma)}|^2 \, ds     
\\
&-\int_{t,x,a} F(t,x, a) \int_{t-\theta}^t \rho^{\prime \prime}_\lambda  (u_n^{(\gamma)} (s,x)-a)\D_{x_i}(u_n)^{(\gamma)} (a^{ij}(x,u_n) \D_{x_j} u_n(s,x)+ b^i(x,u_n))^{(\gamma)}
\, ds 
\\
&+ \int_{t,x,a} F(t,x, a) \int_{t-\theta}^t \rho^{\prime \prime}_\lambda  (u_n^{(\gamma)} (s,x)-a)\frac{1}{2}\sum_{k=1}^\infty |( \sigma^{ik}_r(x,u) \D_{x_i}u )^{(\gamma)}|^2 
\, ds
\\
&+ \int_{t,x,a} F(t,x, a) \int_{t-\theta}^t \rho^{\prime \prime}_\lambda  (u_n^{(\gamma)} (s,x)-a)\sum_{k=1}^\infty ( \sigma^{ik}_r(x,u) \D_{x_i}u )^{(\gamma)} (\sigma^{jk}_{x_j}(x,u)) ^{(\gamma)} \, st 
\end{equs}
Hence,
\begin{equs}
&\limsup_{\gamma \to 0} \E  |C^{(2)}_{\lambda,\gamma}+C^{(4)}_{\lambda,\gamma}|
\\
& \leq \E \left|\int_{t,x,a}  \D_{x_i}F(t,x, a) \int_{t-\theta}^t \rho^{\prime}_\lambda  (u_n (s,x)-a) (a^{ij}(x,u_n) \D_{x_j} u_n(s,x)+ b^i(x,u_n) )
\, ds  \right|
\\
&+                     \label{eq:how-to-call-it}
\E \left|\int_{t,x,a} \D_{aa} F(t,x, a) \frac{1}{2}\int_{t-\theta}^t \rho_\lambda  (u_n (s,x)-a) \sum_{k=1}^\infty|\sigma^{ik}_{x_i}(x,u_n)|^2 \, ds \right|.
\end{equs}
Using the identity
\begin{equs}
& \rho^{\prime}_\lambda  (u_n-a)a^{ij}(x,u_n)\D_{x_j}u_n
 \\
 = &\D_{x_j} [a^{ij} \rho^{\prime}_\lambda(\cdot - a)] (x,u_n) -  [a^{ij}_{x_j} \rho^{\prime}_\lambda(\cdot - a)] (x,u_n),
\end{equs}
integration by parts (in $x$ and $a$), as well as the linear growth of $\sigma_{x_i}, b^i$ and the boundedness of $a^{ij}, a^{ij}_{x_j}$, one derives similarly to \eqref{eq:A1} the estimate

\begin{equation}\label{eq:C}
\limsup_{\gamma \to 0} \E  |C^{(2)}_{\lambda,\gamma}+C^{(4)}_{\lambda,\gamma}| \leq N \theta^{1- \mu}  (1+\E \|u_n\|_{L_2(Q_T)}^4)^{1/2} \leq N(n)\theta^{1- \mu}.
\end{equation}
We continue with an estimate for $C^{(5)}_{\lambda,\gamma}$. We have 
\begin{equs}
\limsup_{\gamma\to 0} \E|C^{(5)}_{\lambda,\gamma}|& \leq \E \left|\int_{t,x,a} F(t,x, a) \int_{t-\theta}^t \rho^{\prime}_\lambda  (u_n-a) (f^i_r(x,u_n)\D_{x_i}u+f^i_{x_i}(u_n)) \right| 
\\
&\leq \E \left|\int_{t,x,a} \D_a F(t,x, a) \int_{t-\theta}^t \rho_\lambda  (u_n-a) (f^i_r(x,u_n)\D_{x_i}u+f^i_{x_i}(u_n)) \right| 
\\
&\leq \E \left|\int_{t,x,a} \D_{x_i} \D_a F(t,x, a) \int_{t-\theta}^t  [f^i_r \rho_\lambda  (\cdot-a)]  \right| 
\\
&+ \E \left|\int_{t,x,a} \D_a F(t,x, a) \int_{t-\theta}^t   [(f^i_{x_i}-f^i_{rx_i}) \rho_\lambda  (\cdot-a)]  \right| 
\\                  \label{eq:estC5}
&\leq N \theta^{1- \mu}  (1+\E \|u_n\|_{L_2(Q_T)}^2)^{1/2} \leq N(n)\theta^{1- \mu}
\end{equs}

Next, we estimate $C^{(3)}_{\lambda,\gamma}$. By It\^o's isometry 
\begin{equs}
& \E C^{(3)}_{\lambda,\gamma}=\E\int_{a,t,x,y}\int_{t-\theta}^t
 \left(h(\tu-a) \sigma^{ik}_{y_i}(y,\tu) \phi_\theta \right. 
 \\
 & \left. -\left(  [\sigma^{ik}_{rx_i}h(\cdot-a)](y,\tu)  \phi_\theta+[\sigma^{ik}_{r}h(\cdot-a)](y,\tu) \D_{y_i} \phi_\theta \right)\right)
\rho^{\prime}_\lambda  (u_n^{(\gamma)} (s,x)-a) \D_{x_j}(\sigma^{jk}(x,u_n))^{(\gamma)} \, ds.
\end{equs}
Using Remark \ref{rem:intergation-limits} and letting $\gamma \to 0$ gives 
\begin{equs}
\lim_{\gamma \to 0}\E C^{(3)}_{\lambda,\gamma}=&-\E \int_{a,t,x,y} \int_a^\tu h(\tr-a) \sigma^{ik}_r(y,\tr ) \, d\tr \D_{y_i} \phi_\theta  \rho_\lambda '(u_n-a) \sigma^{jk}_r(x,u_n)\D_{x_j}u_n  
 \\
 &-\E \int_{a,t,x,y} \int_a^\tu h(\tr-a) \sigma^{ik}_r(y,\tr) \, d\tr \D_{y_i} \phi_\theta \rho_\lambda' (u_n-a)  \sigma^{jk}_{x_j} (x,u_n) 
 \\
 &-\E \int_{a,t,x,y} \int_a^\tu h(\tr-a) \sigma^{ik}_{ry_i}(y,\tr) \, d\tr \phi_\theta \rho_\lambda' (u_n-a)  \sigma^{jk}_r(x,u_n) \D_{x_j}u_n 
 \\
 &-\E \int_{a,t,x,y} \int_a^\tu h(\tr-a) \sigma^{ik}_{ry_i}(y,\tr) \, d\tr\phi_\theta \rho_\lambda' (u_n-a)  \sigma^{jk}_{x_j}(x,u_n)
 \\
  &+\E \int_{a,t,x,y} h(\tu-a)\phi_\theta \sigma^{ik}_{y_i} (y,\tu) \rho_\lambda' (u_n-a) \sigma^{jk}_r(x,u_n) \D_{x_j}u_n 
  \\
   &+\E \int_{a,t,x,y}  h(\tu-a) \phi_\theta \sigma^{ik}_{y_i} (y,\tu) \rho_\lambda' (u_n-a)  \sigma^{jk}_{x_j}(x,u_n)
  \\
 =&\sum_{i=1}^6 D_i.
\end{equs}
By integration by parts  we get 
\begin{equs}D_1+D_3=& +\E \int_{a,t,x,y} \int_a^\tu h'(\tr-a) \sigma^{ik}_r(y,\tr ) \, d\tr \D_{y_i} \phi_\theta  \rho_\lambda (u_n-a) \sigma^{jk}_r(x,u_n)\D_{x_j}u_n  
 \\
 &+\E \int_{a,t,x,y} \int_a^\tu h'(\tr-a) \sigma^{ik}_{ry_i}(y,\tr) \, d\tr\phi_\theta \rho_\lambda (u_n-a)  \sigma^{jk}_r(x,u_n) \D_{x_j}u_n 
 \\
 =&-\E \int_{a,t,x,y} \D_{x_jy_i}\phi_\theta \int_a^\tu h'(\tr-a) \sigma^{ik}_r(y,\tr ) \, d\tr   \int_{\tu}^{u_n}  \rho_\lambda (r-a) \sigma^{jk}_r(x,r) dr 
  \\
 & -\E \int_{a,t,x,y} \D_{y_i}\phi_\theta \int_a^\tu h'(\tr-a) \sigma^{ik}_r(y,\tr ) \, d\tr   \int_{\tu}^{u_n}  \rho_\lambda (r-a) \sigma^{jk}_{rx_j}(x,r) dr 
 \\
&- \E \int_{a,t,x,y} 
\D_{x_j} \phi_\theta \int_a^\tu h'(\tr-a) \sigma^{ik}_{ry_i}(y,\tr) \, d\tr\int_{\tu}^{u_n} \rho_\lambda (r-a)  \sigma^{jk}_r(x,r) \, dr 
 \\
&- \E \int_{a,t,x,y} 
 \phi_\theta \int_a^\tu h'(\tr-a) \sigma^{ik}_{ry_i}(y,\tr) \, d\tr\int_{\tu}^{u_n} \rho_\lambda (r-a)  \sigma^{jk}_{rx_j}(x,r) \, dr.  
\end{equs}
Similarly
\begin{equs}D_2+D_4=& \E \int_{a,t,x,y} \int_a^\tu h'(\tr-a) \sigma^{ik}_r(y,\tr) \, d\tr \D_{y_i} \phi_\theta \rho_\lambda (u_n-a)  \sigma^{jk}_{x_j} (x,u_n) 
 \\
 + &\E \int_{a,t,x,y} \int_a^\tu h'(\tr-a) \sigma^{ik}_{ry_i}(y,\tr) \, d\tr\phi_\theta \rho_\lambda (u_n-a)  \sigma^{jk}_{x_j}(x,u_n),
\end{equs}
and 
\begin{equs}
D_5 =& -\E \int_{a,t,x,y} h'(\tu-a)\phi_\theta \sigma^{ik}_{y_i} (y,\tu) \rho_\lambda  (u_n-a) \sigma^{jk}_r(x,u_n) \D_{x_j}u_n 
\\
=& \E \int_{a,t,x,y} \D_{x_j}\phi_\theta h'(\tu-a) \sigma^{ik}_{y_i} (y,\tu) \int_{\tu}^{u_n} \rho_\lambda  (r-a) \sigma^{jk}_r(x,r) \, dr 
\\
+ & \E \int_{a,t,x,y} \phi_\theta h'(\tu-a) \sigma^{ik}_{y_i} (y,\tu) \int_{\tu}^{u_n} \rho_\lambda  (r-a) \sigma^{jk}_{rx_j}(x,r) \, dr .
\end{equs}
Hence, one easily sees that
\begin{equs}          \label{eq:C3limit-lambda}
\lim_{\lambda \to 0} \lim_{\gamma \to 0}\E C^{(3)}_{\lambda,\gamma}= \mathcal{B}(u_n, \tu, \theta),
\end{equs}
where $\mathcal{B}$ is defined in \eqref{eq:def-error}. 
Putting all of \eqref{eq:0lambda limit}, \eqref{eq:0gamma limit}, \eqref{eq:after Ito}, \eqref{eq:A}, \eqref{eq:C}, \eqref{eq:estC5}, and \eqref{eq:C3limit-lambda} together, we conclude
\begin{equs}
\E \int_{t,x} F(t,x, u_n(t,x))
&\leq
\limsup_{\lambda\to0} \limsup_{\gamma\to 0}\E|C^{(1)}_{\lambda,\gamma}|+\limsup_{\lambda\to0}\limsup_{\gamma\to 0}\E\big(
|C^{(2)}_{\lambda,\gamma}+C^{(4)}_{\lambda,\gamma}|\big)
\\
&+\limsup_{\lambda\to0} \limsup_{\gamma\to 0}\E|C^{(5)}_{\lambda,\gamma}| 
+ \lim_{\lambda\to0}\lim_{\gamma\to 0} \E C^{(3)}_{\lambda,\gamma}
\\
&\leq N(n) \theta^{1-\mu}
 + \mathcal{B}(u_n, \tu,\theta),
\end{equs}
as claimed. Moreover, if $\E\|\xi\|^4_{L_2(\bT^d)}< \infty$, then by virtue of \eqref{eq:uniform} and \eqref{eq:uniform-m+1} it is clear that in  \eqref{eq:A1}, \eqref{eq:A2}, \eqref{eq:C}, \eqref{eq:estC5} we can choose $N$ independent of $n \in \bN$, which completes the proof.
\end{proof}

\begin{proposition}                \label{prop:viscoous-well-posedenss}
Suppose Assumptions \ref{as:sigma}-\ref{as:A} hold. Then, for each $n \in \mathbb{N}$, equation $\Pi(\Phi_n, \xi_n)$  has a unique $L_2$-solution $u_n$.
\end{proposition}

\begin{proof}
We fix $n \in \mathbb{N}$, and since $n$ is fixed, in order to ease the notation we drop the $n$-dependence and we relabel $\bar \Phi:= 
\Phi_n$, $\bar \xi := \xi_n$,  ($\Phi_n$ is given in Proposition \ref{prop:Phi-n} and $\xi_n$ is given in \eqref{def:xi-n}) and we are looking for a solution $u$.   Let $(e_k)_{k=1}^\infty \subset C^\infty(\bT^d)$ be an orthonormal basis of $L_2(\bT^d)$ consisting of eigenvectors of $(I-\Delta)$, and let $\Pi_l : W^{-1}_2 \to V_l:= \text{span}\{e_1,...,e_l\}$ be the projection operator, that is, for $v \in W^{-1}_2$
$$
\Pi_l v:= \sum_{i=1}^l {}_{W^{-1}_2} \langle v, e_i \rangle_{W^1_2} e_i.
$$
Then, the Galerkin approximation 
\begin{equs}          
\begin{aligned}                 \label{eq:Galerkin}
du_l&= \Pi_l \left( \Delta \bar{\Phi}(u_l) + \D_{x_i} \left(a^{ij}(u_l)\D_{x_j} u_l+ b^i(u_l)+f^i(u_l) \right)\right)  \, dt 
\\
&+ \Pi_l\D_{x_i} \sigma^{ik}(u_l)\,  d\beta^k(t) 
\\
u(0)&= \Pi_l \bar\xi,
\end{aligned}
\end{equs}
is an equation on $V_l$ with locally Lipschitz continuous coefficients having linear growth. Consequently, it admits a unique solution $u_l$, for which we have  
$$
u_l \in  L_2(\Omega_T ; W^1_2(\bT^d))\cap L_2(\Omega; C([0,T] ; L_2(\bT^d)).
$$
After applying It\^o's formula for the function $u \mapsto \|u\|_{L_2(\bT^d)}^2$, for $p \geq 2$,  after standard arguments (see for example the proof of Lemma \ref{lem:Ap} in the Appendix) one obtains
\begin{equation}           \label{eq:est-un-H1}
 \E \int_0^T \|u_l \|^2_{W^1_2(\bT^d)} \, dt  \leq N (1+\E \|\bar \xi\|_{L_2(\bT^d)}^2),
\end{equation}
and for all $p \geq 2$ 
\begin{equation}       \label{eq:bound-L2^p}
\E \sup_{t \leq T }\|u_l(t)\|_{L_2(\bT^d)}^p \leq N (1+\E \|\bar \xi\|_{L_2(\bT^d)}^p).
\end{equation}
In  these inequalities the constant $N$ is independent of $l \in \mathbb{N}$.
In $W^{-1}_2(\bT^d)$ we have almost surely, for all $t \in [0,T]$
\begin{align*}
u_l(t)& = \Pi_l \bar\xi + \int_0^t \Pi_l\left( \Delta  \bar{\Phi}(u_l) +\D_{x_i} \left(a^{ij}(u_l)\D_{x_j} u_l+ b^i(u_l)+f^i(u) \right)\right)  \, ds
\\
&+ \int_0^t \Pi_l\D_{x_i} \sigma^{ik}(u_l)\,  d \beta^k(s)
\\
&=J^1_l+J^2_l(t)+J^3_l(t).
\end{align*}
By Sobolev's embedding theorem  and \eqref{eq:est-un-H1} combined with the boundedness of $a^{ij}$ and the linear growth of $b^i$ and $f^i$ we get 
$$
\sup_l \E \|J^2_l \|^2_{W^{1/3}_4([0,T];W^{-1}_2(\bT^d))} \leq \sup_l \E \|J^2_l \|^2_{W^1_2([0,T];W^{-1}_2(\bT^d))}  < \infty.
$$
By \cite[Lemma 2.1]{FLANDOLI}, the linear growth of $\sigma$ and \eqref{eq:bound-L2^p} we have 
$$
\sup_l \E \|J^3_l\|^p_{W^\alpha_p([0,T];W^{-1}_2(\bT^d))} < \infty
$$
for all $\alpha \in (0, 1/2)$ and $p \geq 2$. 
By these two estimates and by \eqref{eq:est-un-H1} we obtain 
$$
\sup_l \E (\|u_l\|_{W^{1/3}_4([0,T];W^{-1}_2(\bT^d)) \cap L_2([0,T]; W^1_2(\bT^d))}) < \infty.
$$
By virtue of \cite[Theorem 2.1 and Theorem 2.2]{FLANDOLI} one can easily see that the embedding 
\begin{equs}
W^{1/3}_4([0,T];W^{-1}_2(\bT^d)) &\cap L_2([0,T]; W^1_2(\bT^d))
\\
& \hookrightarrow \mathcal{X}:= L_2([0,T];L_2(\bT^d)) \cap C([0,T];W^{-2}_2(\bT^d))
\end{equs}
is compact. It follows that for any sequences $(l_q)_{q \in \bN}$, $(\bar{l}_q)_{q \in \bN}$, the laws of $(u_{l_q}, u_{\bar{l}_q})$ are tight on $\mathcal{X} \times \mathcal{X}$.
Let us set 
$$
\beta(t)= \sum_{k=1}^\infty\frac{1}{\sqrt{2^k}}\beta^k(t)\mathfrak{e}_k,
$$
where $(\mathfrak{e}_k)_{k=1}^\infty$ is the standard orthonormal basis of $l_2$.
By Prokhorov's theorem, there exists a (non-relabelled) subsequence $(u_{l_q}, u_{\bar l_q})$ such that the laws of $
(u_{l_q}, u_{\bar l_q}, \beta)$ on $\mathcal{Z}:= \mathcal{X} \times \mathcal{X} \times C([0,T]; l_2)$ are weakly convergent. By Skorohod's representation theorem, there exist $\mathcal{Z}$-valued random variables $(\hat{u}, \check{u}, \tilde{\beta})$, $(\widehat{u_{l_q}}, \widecheck{u_{\bar l_q}}, \tilde{\beta}_q)$, $q\in \bN$, on a probability space $(\tilde{\Omega}, \tilde{\mathcal{F}}, \tilde{\bP})$  such that in $\mathcal{Z}$, $\tilde{\bP}$-almost surely
\begin{equation}          \label{eq:convergence-in-Z}
(\widehat{u_{l_q}}, \widecheck{u_{\bar l_q}}, \tilde{\beta}_q)\to (\hat{u}, \check{u}, \tilde{\beta}),
\end{equation}
as $l \to \infty$, and for each $q \in \bN$, as random variables in $\mathcal{Z}$
\begin{equation}            \label{eq:distribution}
(\widehat{u_{l_q}}, \widecheck{u_{\bar l_q}}, \tilde{\beta}_q)\overset{d}{=}(u_{l_q}, u_{\bar l_q}, \beta).
\end{equation}
 Moreover, upon passing to a non-relabelled subsequene, we may assume that \begin{equation}                     \label{eq:almost-everywhere}
 (\widehat{u_{l_q}}, \widecheck{u_{\bar l_q}}) \to (\hat{u}, \check{u}), \qquad \text{for almost all $(\tilde{\omega}, t,x)$}.
  \end{equation}
Let $(\tilde{\mathcal{F}}_t)_{t \in [0,T]}$ be the augmented filtration of
$\mathcal{G}_t:= \sigma( \hat{u}(s), \check{u}(s), \tilde \beta(s); s \leq t)$, and let $\tilde{\beta}^k(t):= \sqrt{2^k}(\tilde{\beta}(t), \mathfrak{e}_k)_{l_2}$. It is easy to see that $\tilde{\beta}^k$, $k \in \bN$, are independent, standard, real-valued $\tilde{\mathcal{F}}_t$-Wiener processes.
Indeed, they are $\tilde{\mathcal{F}}_t$-adapted by definition and they are independent since $\beta^k$ are. We only have to show that they are $\tilde{\mathcal{F}}_t$-Wiener processes. Let us fix $s < t$ and let $V$ be a bounded continuous function on $C([0,s]; W^{-2}_2(\bT^d)) \times C([0,s]; W^{-2}_2(\bT^d)) \times C([0, s] ; l_2)$. For each $l \in \bN$ we have 
\begin{equs}
&\tilde{\E} (\tilde{\beta}^k_q(t)-\tilde{\beta}^k_q(s))V(\widehat{u_{l_q}}|_{[0,s]}, \widecheck{u_{\bar l_q}}|_{[0,s]}, {\tilde{\beta}_q}|_{[0,s]})
\\
=& \E (\tilde{\beta}^k(t)-\tilde{\beta}^k(s))V( u_{l_q}|_{[0,s]}, u_{\bar l_q}|_{[0,s]}, \beta|_{[0,s]})=0,
\end{equs}
which by using uniform integrability and  passing to the limit $q \to \infty$ shows that $\tilde{\beta}^k(t)$ is a $\mathcal{G}_t$-martingale.  Similarly, $|\beta^k(t)|^2-t$  is a $\mathcal{G}_t$-martingale. By continuity of $\tilde{\beta}^k(t)$ and $|\beta^k(t)|^2-t$, and the fact that their supremum in time is integrable in $\omega$, one can easily see that they are also $\tilde{\mathcal{F}}_t$-martingales. Hence, by L\'evy's characterization theorem (see, e.g., \cite[p.157, Theorem 3.16]{Kar}) $\tilde{\beta}^k$ are $\tilde{\mathcal{F}}_t$-Wiener processes.

 We now show that $\hat{u}$ and $\check{u}$ both satisfy the equation 
 \begin{equs}
dv &= \Delta \bar{\Phi}(v)  + \D_{x_i} \left(a^{ij}(x,v)\D_{x_j} v+ b^i(x,v) +f^i(x,v)\right) \, dt 
\\
&+ \D_{x_i} \sigma^{ik}(x,v)\,  d\beta^k(t) 
\end{equs}
Notice that due to \eqref{eq:est-un-H1}, we have 
$$
\hat{u} \in L_2(\tilde{\Omega}_T ; W^1_2(\bT^d)).
$$
Let us set 
\begin{equs}
\nonumber
\hat{M}(t) &:= \hat{u}(t)- \hat{u}(0)-\int_0^t  \left(\Delta \bar{\Phi}(\hat u) +\D_{x_i} \left(a^{ij}(\hat u)\D_{x_j} \hat u+ b^i(\hat u) + f^i(\hat u)\right)\right)   \, ds
\\
\nonumber
\hat{M}_q(t) &:= \widehat{u_{l_q}}(t)- \widehat{u_{l_q}}(0)-\int_0^t  \Pi_{l_q}\left(\Delta \bar{\Phi}(\widehat{u_{l_q}}) +\D_{x_i} \left(a^{ij}(\widehat{u_{l_q}})\D_{x_j} \widehat{u_{l_q}}+ b^i(\widehat{u_{l_q}})+ f^i(\widehat{u_{l_q}}) \right)\right)  \, ds
\\
M_q(t) &:= u_{l_q}(t)- u_{l_q}(0)-\int_0^t   \Pi_{l_q}\left(\Delta \bar{\Phi}(u_{l_q}) +\D_{x_i} \left(a^{ij}(u_{l_q})\D_{x_j} u_{l_q}+ b^i(u_{l_q})+ f^i(u_{l_q}) \right)\right)  \, ds.
\end{equs}
We will show that for any $\phi \in W^{-2}_2(\bT^d)$ and $k \in \bN$, the processes 
\begin{equs}
\hat M^1(t)&:= (\hat M(t), \phi)_{W^{-2}_2(\bT^d)},
\\
 \hat M^2(t)&:= (\hat M(t), \phi)^2_{W^{-2}_2(\bT^d)}-\int_0^t \sum_{k=1}^\infty | (\D_{x_i}\sigma^{ik}(\hat{u}), \phi)_{W^{-2}_2(\bT^d)}|^2 \, ds,
\end{equs}
and 
$$
\hat M^{3,k}(t):=\tilde{\beta}^k(t)(\hat M(t), \phi)_{W^{-2}_2(\bT^d)}- \int_0^t (\D_{x_i}\sigma^{ik}(\hat{u}), \phi)_{W^{-2}_2(\bT^d)} \, ds
$$
are continuous $\tilde{\mathcal{F}}_t$-martingales. We first show that they are continuous $\mathcal{G}_t$-martingales. Assume for now that $\phi = (I-\Delta)^2\psi$, where $\psi \in V_{l_q}$. For, $i=1,2,3$, let us also define the processes $\hat M^i_q, M^i_q$ similarly to $\hat M^i$, but with $\hat M$, $\hat u$, $\D_{x_i} \sigma^{ki}(\cdot)$ replaced by $\hat M_q, \widehat{u_{l_q}}$, $\Pi_{l_q}\D_{x_i} \sigma^{ik}(\cdot)$ and $M_q, u_{l_q}$, $\Pi_{l_q}\D_{x_i} \sigma^{ik}(\cdot)$, respectively.   Let us fix $s < t$ and let $V$ be a bounded continuous function on $C([0,s]; W^{-2}_2(\bT^d)) \times C([0, s] ; l_2)$. We have that 
\begin{equs}
(M_q(t), \phi)_{W^{-2}_2(\bT^d)}&= \int_0^t (\Pi_{l_q}\D_{x_i}\sigma^{ik}(u_{l_q}), \phi)_{W^{-2}_2(\bT^d)} \, d \beta^k(s).
\end{equs}
It follows that $M^i_q$ are continuous $\mathcal{F}_t$-martingales. Hence, 
$$
\E V(u_{l_q}|_{[0,s]},u_{\bar l_q}|_{[0,s]}, \beta|_{[0,s]})(M_q^i(t)-M_q^i(s))=0,
$$
which combined with \eqref{eq:distribution} gives 
\begin{equation}                \label{eq:martingale-Ml}
\tilde{\E}V(\widehat{u_{l_q}}|_{[0,s]}, \widecheck{u_{\bar l_q}}|_{[0,s]},{\tilde{\beta}_q}|_{[0,s]}) (\hat M_q^i(t)-\hat M_q^i(s)) =0.
\end{equation}
Next, notice that 
\begin{equs}
 \tilde{\E} \int_0^T\left| \left( \Pi_{l_q}\Delta \bar{\Phi}(\widehat{u_{l_q}})- \Delta \bar{\Phi}(\hat u), \phi \right)_{W^{-2}_2(\bT^d)} \right| \,   dt 
=&  \tilde{\E} \int_0^T \left| \left(\bar{\Phi}(\widehat{u_{l_q}})- \bar{\Phi}(\hat{u}), \Delta \psi \right)_{L_2(\bT^d)} \, \right| \, dt  
\\
\le &  \tilde{\E} \int_0^T \| \hat{u}- \widehat{u_{l_q}}\|_{L_2(\bT^d)} \to 0,         \label{eq:conPhi}
\end{equs}
where the convergence follows from  \eqref{eq:convergence-in-Z} and the fact that 
$$
\int_0^T \|\widehat{u_{l_q}}-\hat{u}\|_{L_2(\bT^d)}\, dt 
$$  
are uniformly integrable on $\Omega$ (which in turn follows from \eqref{eq:est-un-H1}).
Notice also that for $v \in W^1_2(\bT^d)$ we have 
\begin{equs}
\left(\Pi_{l_q} \D_{x_j}(a^{ij}(v) \D_{x_i}v ) , \phi\right)_{W^{-2}_2(\bT^d)}&= -\left(a^{ij}(v) \D_{x_i}v  , \D_{x_j} \psi\right)_{L_2(\bT^d)}
\\
& = \left( [a^{ij}](v) , \D_{ij}\psi \right)_{L_2(\bT^d)}+\left(  [a^{ij}_{x_i}](v) , \D_{j}\psi \right)_{L_2(\bT^d)}.
\end{equs}
Since $[a^{ij}](u)(x,r) , [a^{ij}_{x_i}](x,r)$ are Lipschitz continuous in $r \in \bR$ uniformly in $x$ (by Assumption \ref{as:sigma}), we get 
\begin{equs}
 & \tilde{\E} \int_0^T\left| \Pi_{l_q}\left(\D_{x_j}(a^{ij}(\widehat{u_{l_q}}) \D_{x_i}\widehat{u_{l_q}})- \D_{x_j}(a^{ij}(\hat{u}) \D_{x_i}\hat{u} ),  \phi \right)_{W^{-2}_2(\bT^d)} \right| \,   dt  
 \\
  \le  & \tilde{\E} \int_0^T \| \hat{u}- \hat u_l\|_{L_2(\bT^d)} \to 0.            \label{eq:conaij}
\end{equs}
Similarly one shows that 
\begin{equs}
 & \tilde{\E} \int_0^T\left| \left( \Pi_{l_q} \D_{x_i} (b^i(\widehat{u_{l_q}})+f^i(\widehat{u_{l_q}}))- \D_{x_i}( b^{i}(\hat{u})+f^{i}(\hat{u})) ,  \phi \right)_{W^{-2}_2(\bT^d)} \right| \,   dt   \to 0.       \label{eq:conbi}
\end{equs}
Hence, by \eqref{eq:conPhi}, \eqref{eq:conaij}, \eqref{eq:conbi}, and \eqref{eq:convergence-in-Z} we see that for each $t \in [0,T]$
\begin{equation}         \label{eq:MltoM-in-probability}
(\hat M_q(t), \phi)_{W^{-2}_2(\bT^d)} \to (\hat M(t), \phi)_{W^{-2}_2(\bT^d)}
\end{equation} 
in probability. Then, one can easily verify that
$\hat{M}^i_q(t) \to \hat M^i(t)$ in probability. Moreover, for any $\phi \in W^{-2}_2(\bT^d)$ and any $p \geq 2$ we have, by \eqref{eq:distribution} and  \eqref{eq:bound-L2^p}
\begin{equs}
\nonumber
\sup_q \tilde{\E}| (\hat{M}_q(t), \phi)_{W^{-2}_2(\bT^d)}|^p&= \sup_q \E \left|   \int_0^t (\Pi_{l_q} \D_{x_i}\sigma^{ik}(u_{l_q}), \phi)_{W^{-2}_2(\bT^d)} \, \beta^k(s)\right|^p
\\
&\le \|\phi\|_{W^{-2}_2(\bT^d)}^p \E( 1+ \|\bar\xi\|_{L_2(\bT^d)}^p).
\end{equs}
From this, one easily deduces that for each $i=1,2,3$, and  $t\in [0,T]$, $M^i_q(t)$ are uniformly integrable. Hence, we can pass to the limit in \eqref{eq:martingale-Ml} to obtain
\begin{equation}       \label{eq:martingale-property}
\tilde{\E}V(\hat u|_{[0,s]},\check u|_{[0,s]} {\tilde{\beta}}|_{[0,s]})(\hat M^i(t)-\hat M^i(s)) =0.
\end{equation}
In addition, using the continuity of $\hat M^i(t)$ in $\phi$, uniform integrability, and the fact that $\cup_q (I+\Delta)^2 V_{l_q}$ is dense in $W^{-2}_2(\bT^d)$, it follows that  \eqref{eq:martingale-property} holds also for all $\phi \in W^{-2}_2(\bT^d)$. Hence, for all $\phi \in W^{-2}_2(\bT^d)$, $\hat{M}^i$ are continuous $\mathcal{G}_t$-martingales having all moments finite. In particular, by Doob's maximal inequality, they are uniformly integrable (in $t$), which combined with continuity (in $t$) implies that they are also $\tilde{\mathcal{F}}_t$-martingales.  By \cite[Proposition A.1]{HOF2} we obtain that  almost surely, for all $\phi \in W^{-2}_2(\bT^d)$,  $t \in [0,T]$
\begin{equs}
\nonumber
(\hat{u}(t), \phi)_{W^{-2}_2(\bT^d)}&= (\hat{u}(0), \phi)_{W^{-2}_2(\bT^d)}+ \int_0^t (\D_{x_i}\sigma^{ik}(\hat u), \phi)_{W^{-2}_2(\bT^d)} \, d \tilde{\beta}^k(s)
\\
&+\int_0^t (\Delta \bar{\Phi}(\hat{u})+ \D_{x_i}(a^{ij}(\hat{u}) \D_{x_j}\hat{u} +b^i(u)+f^i(u)), \phi)_{W^{-2}_2(\bT^d)} \, ds.
\end{equs}
Notice that $\hat{u}(0)\overset{d}{=} \bar \xi$, which implies that $\hat{u}(0) \in L_{m+1}(\bT^d)$ almost surely.
Choosing $\phi =(1+\Delta)^2 \psi$ for $\psi \in C^\infty(\bT^d)$, we obtain that for almost all $(\tilde{\omega}, t)$
\begin{equs}
\nonumber
(\hat{u}(t), \psi)_{L_2(\bT^d)}&= (\hat{u}(0), \psi)_{L_2(\bT^d)}-\int_0^t  \left(\D_{x_i}\bar{\Phi}(\hat{u})+a^{ij}(\hat u )\D_{x_j}u+ b^i(\hat u)+f^i(\hat{u}) , \D_{x_i} \psi\right)_{L_2(\bT^d)} \, ds
\\
&- \int_0^t ( \sigma^{ik}(\hat u) , \D_{x_i} \psi)_{L_2(\bT^d)} \, d \tilde{\beta}^k(s).
\end{equs}
If follows (see \cite{KR79}) that $\hat{u}$ is a continuous $L_2(\bT^d)$-valued  $\mathcal{F}_t$-adapted process. Hence, $\hat u$ is an $L_2$-solution of equation  $\Pi(\bar \Phi, \hat \xi )$ (on $(\tilde{\Omega}, (\tilde{F}_t)_t , \tilde{\bP})$ with driving noise $(\tilde{\beta}^k)_{k=1}^\infty$) where $\hat{\xi}:= \hat{u} (0)$. Again, by standard arguments, for all $p \geq 2$ one has the estimate
\begin{equs}
\E \sup_{t \leq T }\|\hat u(t)\|_{L_p(\bT^d)}^p+\E \int_0^T \int_{\bT^d}|\hat u|^{p-2} |\nabla \hat u|^2 \, dx dt  \leq N (1+\E \|\bar \xi\|_{L_2(\bT^d)}^p).
\end{equs}
Using this and It\^o's formula (see, e.g., \cite{K_Ito}) for the function 
$$
u \mapsto \int_x \eta(u)\varrho,
$$
 and It\^o's product rule, one can see that $\hat u$ is an entropy solution (on $(\tilde{\Omega}, (\tilde{F}_t)_t , \tilde{\bP})$ with driving noise $(\tilde{\beta}^k)_{k=1}^\infty$) with initial condition $\hat{\xi}:= \hat{u} (0)$.  
 In the exact same way $\check{u}$ is an $L_2$-solution and an entropy solution of $\Pi( \bar \Phi, \check \xi)$ (again, on $(\tilde{\Omega}, (\tilde{F}_t)_t , \tilde{\bP})$ with driving noise $(\tilde{\beta}^k)_{k=1}^\infty$) with $\check \xi:= \check u (0)$. Further, we have for $\delta>0$
\begin{equs}
\tilde \bP(\| \hat \xi - \check \xi \|_{W^{-2}_2(\bT^d)}> \delta) & \leq \delta^{-1} \tilde \E \|\hat \xi - \check \xi \|_{W^{-2}_2(\bT^d)}
\\
& \leq \liminf_{q \to \infty} \delta^{-1} \tilde \E \|\widehat{u_{l_q}} (0) - \widecheck{u_{\bar l_q}}(0)\|_{W^{-2}_2(\bT^d)} 
\\
&= \liminf_{q \to \infty} \delta^{-1}  \E \|\Pi_{l_q}\bar \xi - \Pi_{\bar l_q}\bar \xi\|_{W^{-2}_2(\bT^d)} =0.
\end{equs}
 Hence $\hat{u}$ and 
$\check{u}$ are both entropy solutions with the same initial condition. Moreover, by Lemma \ref{lem:strong entropy} they have the $(\star)$-property. Hence, by Theorem \ref{thm:uniqueness} we conclude that $\hat{u}= \check{u}$. By \cite[Lemma 1.1]{Istvan} we have that the initial sequence 
$(u_l)_{l=1}^\infty$ converges in probability in $\mathcal{X}$ to some $u \in \mathcal{X}$. Using this convergence and the uniform estimates on $u_l$, it  is then straight-forward to pass to the limit in 
\eqref{eq:Galerkin} and to see that the limit $u$ is indeed an $L_2$-solution.
\end{proof}

We are ready to proceed with the proof of Theorem \ref{thm:main-theorem}.

\section{Proof of the main theorem}

\begin{proof}[Proof of Theorem \ref{thm:main-theorem}]

\emph{Step 1:} As a first step we prove the existence of a solution having the $(\star)$-property under the auxiliary assumption that 
$\E \|\xi\|_{L_2(\bT^d)}^4< \infty$. Let $u_n$ be the solutions of $\Pi(\Phi_n, \xi_n)$ constructed in Proposition \ref{prop:viscoous-well-posedenss}. 
 Based on Theorem \ref{thm:uniqueness} \eqref{it:super-inequality},  we will show that $(u_n)_{n\in\N}$ is a Cauchy sequence in $L_1(\Omega_T;L_1(\T^d))$.
Let $\eps_0>0$, 
$\nu \in ((m\wedge2)^{-1},\bar\kappa)$ such that $2\beta \nu>1$ and $\alpha<1\wedge(m/2)$ such that
$-2+(2\alpha)(2\nu)>0$, $\eps \in (0,1)$, $\delta=\eps^{2\nu}$, $n\leq n'$, 
 and $\lambda=8/n$. Thanks to \eqref{eq:approx A}, we have that $R_\lambda\geq n$.
Recalling the uniform estimates \eqref{eq:uniform}, and the triangle inequality
\begin{equ}
\E\|\xi_{n'}(\cdot)-\xi_{n'}(\cdot+h)\|_{L_1(\T^d)}\leq
\E\|\xi(\cdot)-\xi(\cdot+h)\|_{L_1(\T^d)}+2\E\|\xi-\xi_{n'}\|_{L_1(\T^d)},
\end{equ}
the right-hand side of \eqref{eq:super inequality} (with $u=u_n, \tu=u_{n'}$) is bounded by
\begin{equs}
M(\eps)&+N\E\|\xi-\xi_{n'}\|_{L_1(\T^d)}+N \E\|\xi-\xi_{n}\|_{L_1(\T^d)}+ N\eps^{-2}n^{-2}
\\
&+N\eps^{-2}\E\big(\|I_{|u_n|\geq n}(1+|u_n|)\|_{L_m(Q_T)}^m+
\|I_{|u_{n'}|\geq n}(1+|u_{n'}|)\|_{L_m(Q_T)}^m\big),
\end{equs}
where $M(\eps)\to 0$ as $\eps\to 0$. Choose $\eps>0$  such that $M(\eps)\leq\eps_0$.
Then, we can choose $n_0$ sufficiently large so that for $n_0\leq n\leq n'$ 
we have 
\begin{equs}
N\E\|\xi-\xi_{n'}\|_{L_1(\T^d)}+N \E\|\xi-\xi_{n}\|_{L_1(\T^d)}+ N\eps^{-2}n^{-2}\leq \eps_0.
\end{equs}
The same is true for the term
\begin{equs}
N\eps^{-2}\E\big(\|I_{|u_n|\geq n}(1+|u_n|)\|_{L_m(Q_T)}^m+
\|I_{|u_{n'}|\geq n}(1+|u_{n'}|)\|_{L_m(Q_T)}^m\big),
\end{equs}
 thanks to the uniform integrability (in $(\omega,t,x)$) of $1+|u_n|^m$, which follows from \eqref{eq:uniform-m+1}. Hence, for $n_0\leq n\leq n'$, one has
\begin{equ}
\E \int_{t,x} |u_n(t,x)-u_{n'}(t,x)|\leq  3 \eps_0.
\end{equ}
Therefore, since $\eps_0>0$ was arbitrary, $(u_n)_{n\in\N}$ converges in $L_1(\Omega_T; L_1(\T^d))$ to a limit $u$.
 Moreover, by passing to a subsequence, we may also assume that 
 \begin{equation}               \label{eq:ae-convergence}
 \lim_{n \to \infty}u_n=u, \qquad \text{for almost all} \qquad (\omega, t, x)\in \Omega_T \times \bT^d. 
 \end{equation}
Consequently, by Lemma \ref{lem:strong entropy}, \eqref{eq:uniform-m+1},  and Corollary \ref{cor} (i), $u$ has the $(\star)$-property. 
In addition, it follows by \eqref{eq:uniform-m+1} that  for  any $q < m+1$, 
\begin{equation}            \label{eq:uniform-integrability}
(|u_n(t,x)|^q)_{n=1}^\infty \ \text{is uniformly integrable on $\Omega_T \times \bT^d$}.
\end{equation} 

We now show that $u$ is an entropy solution.  From now on, when we refer to the estimates \eqref{eq:uniform}, we only use them with $p=2$. 
 By the estimates in \eqref{eq:uniform-m+1}, it follows that $u$ satisfies Definition \ref{def:entropy-solution}, \eqref{item:in-Lm}.
 
 Let $f \in C_b(\bR)$. For each $n$, we clearly have $[\fra_nf](u_n) \in L_2(
\Omega_T; W^1_2(\bT^d))$ and $\D_{x_i}[\fra_nf](u_n)= f(u_n) \D_{x_i} [\fra_n](u_n)$. Also, we have $|[\fra_nf](r)|\leq \|f\|_{L_\infty} 3K |r|^{(m+1)/2}$  for all $r \in \bR$, which combined with  \eqref{eq:uniform} and \eqref{eq:uniform-m+1} gives that that
\begin{equation*}
\sup_n \E\int_t \|[\fra_nf](u_n)\|_{W^1_2(\bT^d)}^2 \, < \infty. 
\end{equation*} 
Hence, for a subsequence we have  $[\fra_n f](u_n) \wto v_f$, $[\fra_n](u_n)  \wto v$ for some $v_f, v \in L_2(\Omega_T; W^1_2(\bT^d))$. By \eqref{eq:approx A} and  \eqref{eq:ae-convergence},\eqref{eq:uniform-integrability} it is easy to see that 
$v_f=[\fra f](u)$, $ v= [\fra](u)$. Moreover, for any $\phi \in C^\infty(\bT^d)$, $B \in \mathcal{F}$, we have 
\begin{equs}
\E I_B \int_{t,x}  \D_{x_i}[\fra f](u) \phi \,  &= \lim_{n \to \infty} \E I_B \int_{t,x} \D_{x_i} [\fra_n f](u_n) \phi \, 
\\
&= \lim_{n \to \infty} \E I_B \int_{t,x} f(u_n)\D_{x_i}[\fra_n](u_n)  \phi \, 
\\
&= \E I_B \int_{t,x} f(u)\D_{x_i}[\fra](u) \phi \, ,
\end{equs}
where for the last equality we have used that $ \D_{x_i}[\fra_n](u_n)  \wto  \D_{x_i}[\fra](u)$ (weakly) and $f(u_n) \to f(u)$ (strongly) in $L_2(\Omega_T; L_2(\bT^d))$. Hence, \eqref{item:chain_ruleW2} from Definition \ref{def:entropy-solution} is also satisfied. We now show \eqref{item:entropies}. Let $\eta$ and $\phi$ be as in \eqref{item:entropies} and let $B \in \mathcal{F}$. By It\^o's formula (see, e.g., \cite{K_Ito}) for the function 
$$
u \mapsto \int_x \eta(u)\varrho,
$$
 and It\^o's product rule, we have 
\begin{equs}            
\nonumber          
-\E I_B \int_{t,x} \eta(u_n)\D_t\phi \, 
&\leq \E I_B   \left[   \int_x \eta(\xi_n) \phi(0) \,  \right.
\\
&\int_{\bT^d} \eta(\xi) \phi(0) \, dx 
+ \int_{t,x}  \left([\fra^2_n \eta'] (u_n)  \Delta \phi + [ a^{ij}\eta'](u_n) \phi_{x_ix_j} \right)  
\\
 +& \int_{t,x} \left(  [(a^{ij}_{x_j}-f^i_r) \eta' ](u_n) -\eta'(u_n) b^i(u_n)\right)  \phi_{x_i}   
\\
+& \int_{t,x} \left( \eta'(u_n)f^i_{x_i}(u_n)-[f^i_{rx_i} \eta' ](u_n) \right)  \phi  
\\
  + & \int_{t,x}  \left(  \frac{1}{2} \eta''(u_n)\sum_k | \sigma^{ik}_{x_i}(u_n) |^2\phi-  \eta''(u_n) | \nabla [\fra](u_n)|^2\phi  \right) 
\\        
  + & \left. \int_0^T \int_{x}\left(\eta'(u_n)\phi \sigma^{ik}_{x_i}(u_n)-  [\sigma^{ik}_{rx_i}\eta' ](u_n) \phi - [\sigma^{ik}_r\eta' ](u_n)  \phi_{x_i} \right)   d\beta^k(t) \right] .
 \\
 \label{eq:un-entropy-inequality} 
\end{equs}
Notice that $\D_{x_i} [ \sqrt{\eta''}\fra_n](u_n)= \sqrt{\eta''(u_n)} \D_{x_i} [\fra_n](u_n) $. As before we have (after passing to a subsequence if necessary)  $\D_{x_i}  [ \sqrt{\eta''}\fra_n](u_n) \wto \D_{x_i}  [ \sqrt{\eta''}\fra](u)$ in $L_2(\Omega_T ;L_2(\bT^d))$. In particular, this implies that $\D_{x_i}  [ \sqrt{\eta''}\fra_n](u_n) \wto \D_{x_i}  [ \sqrt{\eta''}\fra](u)$ in $L_2(\Omega_T\times \bT^d , \bar\mu)$, where $d \bar\mu:= I_B\phi  \  d\bP\otimes dx \otimes dt $ (recall that $\phi \geq 0$). This implies that 
\begin{equs}
\E I_B\int_{t,x}\phi \eta''(u) | \nabla [\fra](u)|^2\,  \leq \liminf_{n \to \infty} \E I_B \int_{t,x} \phi \eta''(u_n) | \nabla [\fra_n](u_n) |^2\, .
\end{equs}
On the basis of \eqref{eq:ae-convergence}, \eqref{eq:uniform-integrability} and the construction of $\xi_n$ and $\mathfrak{a}_n$ one can easily see that the  remaining terms in \eqref{eq:un-entropy-inequality} converge to the corresponding ones from \eqref{eq:entropy-inequality}.

Hence, taking $\liminf$ in \eqref{eq:un-entropy-inequality} along an appropriate subsequence, we see that $u$ satisfies Definition \ref{def:entropy-solution}, \eqref{item:entropies}.

 To summarise, we have shown that if in addition to the assumptions of Theorem \ref{thm:main-theorem} we have that $\E\| \xi\|_{L_2(\bT^d)}^4< \
\infty$, then there exists an entropy solution to \eqref{eq:main-equation} which has the $(\star)$-property (therefore, it is also unique by Theorem \ref{thm:uniqueness}). In addition, we can pass to the limit in \eqref{eq:uniform}-\eqref{eq:uniform-m+1} to obtain that 
\begin{equation}   \label{eq:uniform2}
\begin{aligned}
\E \sup_{t \leq T} \| u\|_{L_2(\bT^d)}^2 + \E \|\nabla [\fra]  (u) \|_{L_2(Q_T)}^2 &\leq N ( 1+ \E \|\xi\|_{L_2(\bT^d)}^2),
\\
\E \sup_{t \leq T} \|u\|_{L_{m+1}(\bT^d)}^{m+1}+ \E \| \nabla A(u)\|_{L_2(Q_T)}^2 &\leq N(1+\E \|\xi\|_{L_{m+1}(\bT^d)}^{m+1}),
\end{aligned}
\end{equation}
with a constant $N$ depending only on $N_0, N_1, d,K,T$ and $m$.

\emph{Step 2:} We now remove the extra condition on $\xi$.  For $n \in \bN$, let $\xi_n$ be defined again by $\xi_n= (n \wedge \xi)\vee(-n)$ and let $u_{(n)}$ be the unique solution of $\mathcal{E}(\Phi,\xi_{n})$. 
Notice that by step 1, $u_{(n)}$ has the $(\star)$-property. Hence, by Theorem \ref{thm:uniqueness} (i) we have that $(u_{(n)})$ is a Cauchy sequence in $L_1(\Omega_T;L_1(\bT^d))$ and therefore has a limit $u$. In addition, $u_{(n)}$ satisfy the estimates \eqref{eq:uniform2} uniformly in $n \in \bN$. With the arguments provided above it is now  routine to show that 
$u$ is an entropy solution. 

 We finally show \eqref{eq:main contraction} which also implies uniqueness. Let $\tilde{u}$ be an entropy solution of $\mathcal{E}(\Phi, \tilde{\xi})$. By Theorem \ref{thm:uniqueness} we have 
$$
\esssup_{t\in[0,T]}\E\int_x|u_{(n)}(t,x)-\tu(t,x)|\leq\E\int_x|\xi_n(x)-\txi(x)|,
$$
where $u_{(n)}$ are as above. We then let $n \to \infty$ 
to finish the proof.
\end{proof}

\section{Stochastic mean curvature flow}
In this section we demonstrate the proof of well-posedness for the one-dimensional stochastic mean curvature flow in graph form by minor modifications of the techniques developed in the previous sections. 

The stochastic mean curvature flow describes the evolution of a curve $M_t=\phi(t,M_0) \subset \bR^2$, $t \in [0,T]$ given by the flow $\phi : [0,T] \times M_0 \to \bR^2$ satisfying
\begin{equs}
d\phi(t,x,y)= \overset{\to}{H}_{M_t}(\phi(t,x,y)) \, dt + \sum_{k=1}^\infty\nu_{M_t}(\phi(t,x,y)) h^k(x,y) \circ d\beta^k(t),
\end{equs}
where $\overset{\to}{H}_{M_t}((x,y))$ is the mean curvature vector of $M_t$ at the point $(x,y) \in M_t$ and $\nu_{M_t}(x,y)$  denotes the normal vector of $M_t$ at $(x,y) \in M_t$. Assuming that $M_t$ is the level set of a function $f(t,\cdot) : \bR^2 \to \bR$, one derives the SPDE
\begin{equs}
df = |\nabla f| \text{div} \left(\frac{\nabla f}{|\nabla f| } \right) \, dt + \sum_{k=1}^\infty h^k |\nabla f | \circ d \beta^k(t).
\end{equs}
In the graph case, that is, when $f(x,y)=y-v(x)$ 
the above equation becomes 
\begin{equ}\label{eq:smc}
d v= \sqrt{1+|v_x|^2} \D_x \left( \frac{v_x}{\sqrt{1+|v_x|^2}} \right) \, dt +  \sum_{k=1}^\infty h^k(x,v)\sqrt{1+|v_x|^2}  \circ d\beta^k(t).
\end{equ}
In \cite{ESR12} the well-posedness of \eqref{eq:smc} is shown under the assumption that $h^1= \eps$, for some $\eps \leq \sqrt{2}$ and $h^k=0$ for $k\neq 1$.  Here, we assume that $h^k(x,y)=h^k(x)$. Hence, taking the derivative in $x$ in the above equation, we derive the following SPDE for $u=v_x$
\begin{equs}    \label{eq:main-equation'}
\begin{aligned}
du=&\partial_{xx} \text{arctan}(u) \, dt+\sum_{k=1}^\infty\partial_x(h^k(x)\sqrt{1+u^2})\circ d\beta^k(t). 
\end{aligned}
\end{equs}
For a function $\Phi: \bR \to \bR$, let  $\mathcal{E}(\Phi,\xi)$ denote the periodic problem 
\begin{equs}
du= \Delta \Phi(u) \, dt + \sum_{k=1}^\infty\D_x ( h^k(x) \sqrt{1+u^2})  \circ d \beta^k(t) \qquad \text{in } [0,T] \times \bT^d,
\end{equs}
with initial condition $\xi$. Therefore, we aim to solve $\mathcal{E}(\Phi, \xi)$ for $\Phi(u)= \text{arctan}(u)$. 
As mentioned above, the proofs of the statements in this section are almost identical to the corresponding  ones of the previous sections. For this reason, we will restrict to pointing out the differences.

For $n \in \mathbb{N}$, let $\mathfrak{b}_n$ be the unique real function on $\bR$  defined by the following properties 
\begin{enumerate}
\item $\mathfrak{b}_n$ is continuous and  odd 
\item $\mathfrak{b}_n(r)= -r (1+r^2)^{-3/2}$ for $r \in [0,n]$
\item $\mathfrak{b}_n$ is linear on $[n, c_n]$, vanishes on $[c_n, \infty)$, and 
\begin{equs}          \label{eq:condition-b-n}
\int_n^{c_n} \mathfrak{b}_n(r) \, dr = -\frac{1}{2\sqrt{1+n^2}}.
\end{equs}
\end{enumerate}
For $n \in \bN$ we set 
\begin{equs}
\mathfrak{a}_n(r):= 1+\int_0^r \mathfrak{b}_n(s)\, ds, \qquad \Phi_n(s):= \int_0^r \mathfrak{a}_n^2(s) \, ds, 
\end{equs}
\begin{equs}
\mathfrak{a}_\infty(r):= (1+|r|^2)^{-1/2}, \qquad \Phi_\infty(r):= \text{arctan}(r).
\end{equs}
We introduce
\begin{equs}
\mathcal{L}:=\Big\{ u : \Omega_T \to L_2(\bT)\Big|\esssup_{[0,T]} \E \|u(t)\|_{L_2(\bT^d)}^p< \infty, \text{for all} \, p >2 \Big\}.
\end{equs}

\begin{remark}\label{rem:coercivity}
By virtue of \eqref{eq:condition-b-n} we have that for all $n \in \bN \cup \{ \infty\}$, $r \in \bR$,  
\begin{equs}
\frac{1}{|\fra_n(r)|} \leq 2(1+|r|). 
\end{equs}
\end{remark}
\begin{assumption}                  \label{as:noise-coef'}
The function $h=(h^k)_{k=1}^\infty: \bT  \to l_2$ is in $C^3(\bT; l_2)$, and for a constant $N_0$
$$
\|h\|_{C^3(\bT; l_2)} \leq N_0.
$$
\end{assumption}

\begin{assumption}                     \label{as:ic'}
For all $p>2$,  $\E \| \xi\|_{L_2(\bT)}^p< \infty$. 
\end{assumption}

\begin{remark}
From now on we use the notation of Section \ref{sec:formulation} with $d=1$,
and 
$$
\sigma^k(x,r):=h^k(x)\sqrt{1+|r|^2}.
$$
Moreover, notice that $\sigma^k$ satisfies Assumption \ref{as:sigma} with $\bar{\kappa}=\beta=\tilde{\beta}=1$. 
\end{remark}

\begin{definition}        \label{def:entropy-solution'}
Let $n \in \bN \cup \{ \infty\}$.  An entropy solution of $\mathcal{E}(\Phi_n, \xi)$  is 
a stochastic process  $ u \in \mathcal{L}$ such that
\begin{enumerate}[(i)]

\item  For all $f \in C_b(\bR)$ we have $[\fra_n f](u) \in L_2(\Omega_T;W^1_2(\bT)$) and 
\begin{equation*}      
\D_{x} [\fra_n f](u)= f(u) \D_{x} [\fra_n](u).
\end{equation*}

\item For all convex $\eta\in C^2(\bR)$ with $\eta''$ compactly supported and all  $\phi\geq 0$ of the form $\phi= \varphi \varrho$ with $\varphi \in C^\infty_c([0,T))$, $\varrho \in C^\infty(\bT)$,   we have 
almost surely
\begin{equs}                    
-\int_0^T \int_{\bT}  \eta(u)\phi_t\, dx dt & \leq  \int_{\bT} \eta(\xi) \phi(0) \, dx 
\\
+ & \int_0^T \int_{\bT}  \left([\fra^2_n\eta'] (u)  \Delta \phi + [ a\eta'](u) \Delta \phi \right)  \, dx  dt 
\\
 +&\int_0^T\int_{\bT^d}\left(  [(a_{x}+
\tfrac{1}{2} b_r) \eta' ](u) -\eta'(u) b^i(u)\right)  \phi_{x_i}   \,dx dt 
\\
+&\int_0^T\int_{\bT^d}\left( -\eta'(u)\tfrac{1}{2} b_x(u)+[\tfrac{1}{2} b_{rx} \eta' ](u) \right)  \phi   \,dx dt 
\\
  + & \int_0^T \int_{\bT^d} \left(  \frac{1}{2} \eta''(u)\sum_k | \sigma^{k}_{x}(u) |^2\phi-  \eta''(u) | \nabla [\fra_n](u)|^2\phi  \right) \, dx  dt 
\\        
  + &\int_0^T \int_{\bT^d}\left(\eta'(u)\phi \sigma^{k}_{x}(u)-  [\sigma^{k}_{rx}\eta' ](u) \phi - [\sigma^{k}_r\eta' ](u)  \phi_{x} \right)  \,dx d\beta^k(t). 
\end{equs}
\end{enumerate}
\end{definition}
With the notation of Definition \ref{def:star} we define:
\begin{definition}
A function $u\in \mathcal{L}$ is said to have the $(\star \star )$-property if there exists a $\mu \in (0,1)$ such that for all $\tu \in \mathcal{L}$, $h,\varrho,\varphi$ as in the Definition \ref{def:star}, and for all sufficiently small $\theta>0$, we have that $F_\theta(\cdot, \cdot, u) \in L_1(\Omega_T \times \bT)$ and
\begin{equ}
\E \int_{t,x} F_\theta(t,x, u(t,x)) \leq  N\theta^{1-\mu}+ \mathcal{B}(u, \tu, \theta)   \label{eq:strong-entropy'}
\end{equ}
for some constant $N$ independent of $\theta$.
\end{definition}

Choosing $m=3$ in \eqref{eq:F-estimate-sharper} from Lemma \ref{lem:F} gives the following. 
\begin{lemma}\label{lem:F'}
For any   $\lambda\in (3/4, 1)$, $k\in\mathbb{N}$ we have
\begin{equation}\label{eq:F estimate'}
\E \| \D_ a F_\theta \|_{L_\infty([0,T]; W^{k}_4(\bT \times \bR))}^4 \leq N \theta^{-\lambda 4}(1+
\esssup_{[0,T]}
\E \| \tu(t)\|_{L_2(\bT)}^4),
\end{equation} 
where $N$ depends only on $N_0, k,d, T ,\lambda$,  and the functions $h, \varrho, \varphi, \tu$, but not on $\theta$.
\end{lemma}

Similarly to Corollary \ref{cor} one has:
\begin{corollary} \label{cor'}
(i) Let $u_n$ be a sequence bounded in $L_2(\Omega_T\times \bT)$, satisfying the $(\star \star)$-property  uniformly in $n$, that is,  with constant $N$ in \eqref{eq:strong-entropy'} independent of $n$. 
Suppose that $u_n$ converges for almost all $\omega,t,x$ to a function $u$. Then $u$ has the $(\star \star)$-property.

(ii) Let $u\in L_2(\Omega\times Q_T)$. Then one has for all $\theta>0$
\begin{equ}\label{eq:0lambda limit'}
\E \int_{t,x} F_\theta(t,x, u(t,x)) = \lim_{\lambda \to 0}  \E \int_{t,x,a} F_\theta(t,x, a) \rho_\lambda(u(t,x)-a) \, .
\end{equ}
\end{corollary}
\begin{theorem}             \label{thm:mean-curv-thm}
Suppose that Assumption \ref{as:noise-coef'} holds and let $\xi, \tilde{\xi}$ satisfy Assumption \ref{as:ic'}. For  $n, n'  \in \bN \cup \{ \infty\}$,  let $u, \tu$ be entropy solutions of $\mathcal{E}(\Phi_n, \xi)$,  $\mathcal{E}(\Phi_{n'}, \txi)$ respectively, and assume that $u$ has the $(\star \star )$-property. Then,

\begin{enumerate}[(i)]
 \item if furthermore $n={n'}$, then 
\begin{equ}\label{eq:L_1 contraction'}
\esssup_{t\in[0,T]}\E\|u(t)-\tu(t)\|_{L_1(\bT)} \leq N \E\|\xi-\txi \|_{L_1(\bT)},
\end{equ}
\label{it:contraction'}

\item  \label{it:super-inequality'}If $u \in L_2(\Omega_T; W^1_2(\bT))$, then for all $\eps,\delta\in(0,1]$, $\lambda\in[0,1]$,  we have 
\begin{equs}
\E \|u-\tilde{u}\|_{L_1(Q_T)}   
& \leq N \E \|\xi-\txi\|_{L_1(\bT)}
\\
&+N\eps \big(1+\E\|\D_x [\fra_n] (u)\|_{L_2(Q_T)}^2+ \E \|u\|_{L_2(Q_T)}^2 \big)
\\
&+N\sup_{|h|\leq\eps} \E\|\txi(\cdot)-\txi(\cdot+h)\|_{L_1(\T)}
\\
&+N\eps^{-2}\E\big(\|I_{|u|\geq R_\lambda}(1+|u|)\|_{L_1(Q_T)}+
\|I_{|\tu|\geq R_\lambda}(1+|\tu|)\|_{L_1(Q_T)}\big)
\\
&+N C(\delta, \eps, \lambda) \E(1+\|u\|^2_{L_{2}(Q_T)}+\|\tu\|^2_{L_2(Q_T)}),
\label{eq:super inequality'}
\end{equs}
where
\begin{equs}\label{eq:R lambda'}
R_\lambda& :=\sup\{R\in[0,\infty]:\,|\fra_n(r)-\fra_{n'}(r)|\leq\lambda, \,\,\forall |r|<R\},
\\
C(\delta, \eps, \lambda)&:= \big(\delta+ \delta^{2} \eps^{-2}+\delta \eps^{-1}+\eps^{2}\delta^{-1}+\eps^{-2}\lambda^2+\eps),
\end{equs}
and $N$ is a constant depending only on $N_0,d,$ and $T$.
\end{enumerate}
\end{theorem}
\begin{proof}
The proof is mostly a repetition of the proof of Theorem \ref{thm:uniqueness} (with $m=1$, $\bar{\kappa}=1$, and $\beta=1$) with very small modifications. Therefore, we only point out these modifications. One proceeds as in the proof of Theorem \ref{thm:uniqueness} up to \eqref{eq:1}. There, we claim that \eqref{eq:1} holds  with $\alpha=1$. 
This follows if one reproduces the proof of \cite[Theorem 4.1, (4.8) and (4.18) therein]{DareiotisGerencserGess} (with $m=1$) with only one difference: In order to estimate the term $D_1$  (see \cite[(4.13)]{DareiotisGerencserGess}), one uses that $\sup_n \sup_r |\fra_n'(r)|< \infty$ to obtain the estimate
\begin{equs}
|D_1| \le \delta^2|u-\tu|,
\end{equs}
in place of \cite[(4.16)]{DareiotisGerencserGess}. Proceeding then as in the proof of Theorem \ref{thm:uniqueness} one obtains \eqref{eq:fork} with $m=1$, $\bar{\kappa}=1$, and $\beta=1$. From there, \eqref{it:contraction'} follows exactly as in Theorem \ref{thm:uniqueness}. For \eqref{it:super-inequality'}, the only difference to the proof of Theorem \ref{thm:uniqueness} is that instead of Lemma \ref{lem:frac reg}, one uses the following
\begin{equs}
&\E \int_{t,x,y}|u (t,x)-u(t,y)|\varrho_\varepsilon(x-y)
\\
\leq  & \E \int_{t,x,y} \int_0^1 |x-y| |u_{x}(x+ \theta (y-x))| \, d \theta \varrho_\varepsilon(x-y)
\\
= & \E \int_{t,x,y} \int_0^1 |x-y| \frac{| (\D_{x} [\fra_n](u))(x+ \theta (y-x))|}{\fra_n(u)(x+ \theta (y-x))}\, d \theta \varrho_\varepsilon(x-y)
\\
\leq & \eps N \left( \E \| \D_x [\fra_n] (u) \|^2_{L_2(Q_T)}  + \E \| (\fra_n(u))^{-1} \|^2_{L_2(Q_T)}  \right) 
\\
\leq & \eps N \left(1+\E \| \D_x  [\fra_n](u) \|^2_{L_2(Q_T)}  + \E \|u\|^2_{L_2(Q_T)}  \right),
\end{equs}
with $N$ independent of $n$ (where we have used Remark \ref{rem:coercivity}).
\end{proof}

Similarly to \eqref{eq:uniform}-\eqref{eq:uniform-m+1}, we have that if $u_n$ are $L_2$-solutions to $\mathcal{E}(\xi, \Phi_n)$ for $n \in \bN$, then for all $p \geq 2$
\begin{equs}\label{eq:uniform'}
\E \sup_{t \leq T} \| u_n\|_{L_2(\bT)}^p + \E \|\D_x [\fra_n](u_n) \|_{L_2(Q_T)}^p &\leq N ( 1+ \E \|\xi\|_{L_2(\bT)}^p),
\\
\label{eq:uniform-m+1'}
\E \sup_{t \leq T} \|u_n\|_{L_{2}(\bT)}^{2}+ \E \|\D_x \Phi_n(u_n)\|_{L_2(Q_T)}^2 &\leq N(1+\E \|\xi\|_{L_{2}(\bT)}^{2}),
\end{equs}
where $N$ depends only on $N_0, T, d,$ and $p$. Using these estimates, Corollary \ref{cor'}, and Lemma \ref{lem:F'}, one proves the following analogue of Lemma \ref{lem:strong entropy} :
\begin{lemma}                  \label{lem:strong entropy'}
Let Assumptions \ref{as:noise-coef'}-\ref{as:ic'} hold, and  for each $n \in \bN$, let $u_n$ be an $L_2$-solution of $\mathcal{E}(\Phi_n, \xi)$. Then, $u_n$ has the $(\star \star)$-property and the constant $N$ in \eqref{eq:strong-entropy'} is independent of $n$. 
\end{lemma}
Moreover, similarly to Proposition \ref{prop:viscoous-well-posedenss} one proves the following. 

\begin{proposition}           \label{prop:existence}
Let Assumptions \ref{as:noise-coef'}-\ref{as:ic'} hold. Then, for each $n \in \mathbb{N}$, equation $\mathcal{E}(\Phi_n, \xi)$  has a unique $L_2$-solution $u_n$. 
\end{proposition}
Finally, using Proposition \ref{prop:existence}, Lemma \ref{lem:strong entropy'}, and Theorem \ref{thm:mean-curv-thm}, we obtain the following theorem in a similar manner as Theorem \ref{thm:main-theorem} is concluded from Proposition \ref{prop:viscoous-well-posedenss}, Lemma \ref{lem:strong entropy}, and Theorem \ref{thm:uniqueness}.
\begin{theorem}
Let Assumptions \ref{as:noise-coef'}-\ref{as:ic'} hold. 
Then, there exists a unique entropy solution of $\mathcal{E}(\Phi_\infty, \xi)$.  Moreover, if $\tu$ is the unique entropy solution of $\mathcal{E}(\Phi_\infty, \txi)$, then 
\begin{equs}           \label{eq:main contraction'}
\esssup_{t \leq T} \E \| u(t)-\tu(t)\|_{L_1(\bT)} \leq N \E \| \xi-\txi \|_{L_1(\bT)},
\end{equs}
where $N$ is a constant depending only on $N_0$ and $T$.
\end{theorem}
\begin{remark}
Notice that in Theorem \ref{thm:mean-curv-thm} \eqref{it:super-inequality'}, there is the extra assumption that $u \in L_2(\Omega_T; W^1_2(\bT))$ as compared to Theorem \ref{thm:uniqueness} \eqref{it:super-inequality}. However, this does not cause any complication since the approximating sequence $u_n$ of Proposition \ref{prop:existence} satisfies this condition. 
\end{remark}

\appendix
\section{}
\begin{lemma}                                    \label{lem:Ap}
Let Assumptions \ref{as:A} and \ref{as:sigma} hold. Let $\Phi_n$ and $\xi_n$ be as in Proposition \ref{prop:Phi-n} and \eqref{def:xi-n} respectively, let $u$ be an $L_2$-solution of $\Pi(\Phi_n, \xi_n)$, and let $p \in [2, \infty)$. Then there exists a constant $N$ depending only on $K, N_0, N_1,T,d,m$, and $p$ such that
\begin{equs}              \label{eq:appendix1}
\E \sup_{t \leq T} \| u\|_{L_2(\bT^d)}^p + \E \|\nabla [\fra_n]  (u) \|_{L_2(Q_T)}^p &\leq N ( 1+ \E \|\xi_n\|_{L_2(\bT^d)}^p),
\\                    \label{eq:appendix2}
\E \sup_{t \leq T} \|u\|_{L_{m+1}(Q_T)}^{m+1}+ \E \| \nabla \Phi_n(u)\|_{L_2(\bT_T)}^2 &\leq N(1+\E \|\xi_n\|_{L_{m+1}(\bT^d)}^{m+1}).
\end{equs}
  
\end{lemma}
\begin{proof} 
We start with \eqref{eq:appendix1}. By It\^o's formula we have 
\begin{equs}         
\|u(t)\|_{L_2(\bT^d)}^2= &\|\xi_n\|^2_{L_2(\bT^d)}  -2 \int_0^t ( \D_{x_i} \Phi_n(u)+a^{ij}(u)\D_{x_j} u+ b^i(u)+f^i(u), \D_{x_i} u)_{L_2(\bT^d)} \, ds 
\\  
-& 2\int_0^t (\sigma^{ik}(u),\D_{x_i} u)_{L_2(\bT^d)}   d \beta^k(s)+ \int_0^t  \sum_{k=1}^\infty \| \sigma^{ik}_r(u)\D_{x_i} u +\sigma^{ik}_{x_i}(u)\|^2_{L_2(\bT^d)} \, ds
\\
= &\|\xi_n\|^2_{L_2(\bT^d)}+ \int_0^t  \sum_{k=1}^\infty \|\sigma^{ik}_{x_i}(u)\|^2_{L_2(\bT^d)} -2 ( \D_{x_i} \Phi_n(u)+f^i(u), \D_{x_i}u)_{L_2(\bT^d)}\, ds 
\\   \label{Ito-formula}
-& 2\int_0^t (\sigma^{ik}(u),\D_{x_i} u)_{L_2(\bT^d)}  \, d \beta^k(s).
\end{equs}
Using that $\Phi_n$ is increasing and \eqref{eq:sigma-growth}, we get
\begin{equs}
\|u(t)\|_{L_2(\bT^d)}^2\leq N+\|\xi_n\|^2_{L_2(\bT^d)}&+ \int_0^t \left( N \| u\|^2_{L_2(\bT^d)}+(f^i(u), \D_{x_i}u)_{L_2(\bT^d)} \right)\, ds 
\\
&- 2\int_0^t (\sigma^{ik}(u),\D_{x_i} u)_{L_2(\bT^d)}  \, d \beta^k(s).
\end{equs}
Notice that 
\begin{equs}
|(f^i(u), \D_{x_i}u)_{L_2(\bT^d)}|& = \left|\int_{\bT^d} \D_{x_i}  [f^i](x,u)- [f^i_{x_i}](x,u) \, dx\right| 
\\
& = \left| \int_{\bT^d} [f^i_{x_i}](x,u) \, dx \right| \le 1+\|u\|^2_{L_2(\bT^d)},      \label{eq:est-f}
\end{equs}
where for the last inequality we used \eqref{eq:b-linear-growth}, and the fact that $[f^i] \in W^{1,1}(\bT^d)$
for almost all $(\omega,t)\in \Omega_T$ (which in turn follows from 
\eqref{eq:b-linear-growth} and \eqref{eq:F-linear-growth}). 
Raising to the power $p/2$, taking suprema up to time $t'$ and expectations, gives 
\begin{equs}
\E \sup_{t \leq t'} \|u(t)\|_{L_2(\bT^d)}^p\leq  &N\left[1+\E \|\xi_n\|^p_{L_2(\bT^d)}+  \int_0^{t'}   \E \sup_{t \leq s} \|u(t)\|_{L_2(\bT^d)}^p \, ds \right.
\\          \label{eq:pre-gronwall}
+& \left.\E \sup_{t \leq t'} \left| \int_0^t (\sigma^{ik}(u),\D_{x_i} u)_{L_2(\bT^d)}  \, d \beta^k(s)\right|^{p/2}\right].
\end{equs}
By the Burkholder-Davis-Gundy inequality we have 
\begin{equs}
 \E \sup_{t \leq t'} \left| \int_0^t (\sigma^{ik}(u),\D_{x_i} u)_{L_2(\bT^d)}  \, d \beta^k(s)\right|^{p/2}\leq N \E  \left( \int_0^{t'}\sum_k (\sigma^{ik}(u),\D_{x_i} u)^2_{L_2(\bT^d)}  \, ds\right)^{p/4}.
\end{equs}
As above
\begin{equs}
(\sigma^{ik}(u),\D_{x_i} u)_{L_2(\bT^d)}= \int_{\bT^d}\D_{x_i} [\sigma^{ik}](x,u)-[\sigma^{ik}_{x_i}](x,u) \, dx= -\int_{\bT^d}[\sigma^{ik}_{x_i}](x,u) \, dx.
\end{equs}
By Minkowski's inequality and \eqref{eq:sigma-growth} one has 
\begin{equs}
\sum_{k=1}^\infty \left(\int_{\bT^d} [\sigma^{ik}_{x_i}](x,u)\, dx \right)^2 \leq N (1+\|u\|_{L_2(\bT^d)}^4).
\end{equs}
Consequently, 
\begin{equs}
 & \E \sup_{t \leq t'} \left| \int_0^t (\sigma^{ik}(u),\D_{x_i} u)_{L_2(\bT^d)}  \, d \beta^k(s)\right|^{p/2}\\
 \leq & N + N\E\left( \int_0^{t'}\|u\|_{L_2(\bT^d)}^4 \right)^{p/4}
 \\     \label{eq:to-be-finite}
 \leq & N+  \varepsilon \E \sup_{t \leq t'} \|u(t)\|_{L_2(\bT^d)}^p +\eps^{-1} N \int_0^{t'} \E \sup_{t \leq s} \|u(t)\|_{L_2(\bT^d)}^p \, ds,
\end{equs}
which combined with \eqref{eq:pre-gronwall} gives, 
\begin{equs}               \label{eq:boundL-2-p}
 \E \sup_{t \leq T} \|u(t)\|_{L_2(\bT^d)}^p \leq N(1+\E \|\xi_n\|_{L_2(\bT^d)}^p), 
 \end{equs}
by virtue of Gronwall's lemma,  provided that the right hand side of \eqref{eq:to-be-finite} is finite. The latter can be achieved by means of a standard localization argument the details of which are left to the reader. Going back to \eqref{Ito-formula} after rearranging, raising to the power $p/2$, and taking expectations gives
\begin{equs}
\E \| \nabla [\fra_n](u) \|_{L_2(Q_T)}^p & \leq N \left[ \E \|\xi_n\|^p_{L_2(Q)}+ \E  \left| \int_0^T (\sigma^{ik}(u),\D_{x_i} u)_{L_2(\bT^d)}  \, d \beta^k(s)\right|^{p/2} \right. 
\\
&+ \left. \E \int_0^T \left( \sum_{k=1}^\infty \| \sigma^{ik}_{x_i}(u) \| ^2_{L_2(\bT^d)}+ |(f^i(u), \D_{x_i}u)_{L_2(\bT^d)}|\right)^{p/2} \, ds\right],
\end{equs}
which by \eqref{eq:sigma-growth}, \eqref{eq:est-f}, \eqref{eq:to-be-finite}, and \eqref{eq:boundL-2-p} gives 
\begin{equs}          \label{eq:energy-psi}
\E \| \nabla [\fra_n](u) \|_{L_2(Q_T)}^p & \leq N (1+\E \|\xi_n\|^p_{L_2(\bT^d)}).
\end{equs}
Hence, we have shown \eqref{eq:appendix1}. The estimate
\eqref{eq:appendix2}  is proved in a similar way. Namely, one first applies It\^o's formula for the function $u \mapsto \|u\|_{L_{m+1}(Q)}^{m+1}$ (see, e.g., \cite[Lemma 2]{DG15}) and by arguments similar to those used above, one derives the estimate
\begin{equs}            \label{eq:estimate-L-m+1}
\E \sup_{t \leq T} \|u(t)\|_{L_{m+1}(\bT^d)}^{m+1} \leq N(1+\E \|\xi_n\|_{L_{m+1} (\bT^d)}^{m+1} ).
\end{equs}
Writing It\^o's formula (see, e.g., \cite{K_Ito}) for the function 
$$
u \mapsto  \int_{\bT^d} \int_0^u \Phi_n(r) \, dr \, dx
$$
and using the properties of $\Phi_n$ and \eqref{eq:estimate-L-m+1}, the estimate 
\begin{equs}
 \E \| \nabla \Phi_n(u)\|_{L_2(Q_T)}^2 &\leq N(1+\E \|\xi_n\|_{L_{m+1}(\bT^d)}^{m+1}),
\end{equs}
follows in the same way as \eqref{eq:energy-psi} follows from \eqref{eq:boundL-2-p}.
This finishes the proof.
\end{proof}

\section*{Acknowledgement}

B. Gess acknowledges financial support by the DFG through the CRC 1283 ``Taming uncertainty and profiting from randomness and low regularity in analysis, stochastics and their applications."

\bibliography{PMEgrad}

\begin{thebibliography}{HKRSs18}
\expandafter\ifx\csname url\endcsname\relax
  \def\url#1{\texttt{#1}}\fi
\expandafter\ifx\csname urlprefix\endcsname\relax\def\urlprefix{URL }\fi
\expandafter\ifx\csname href\endcsname\relax
  \def\href#1#2{#2}\fi
\expandafter\ifx\csname burlalt\endcsname\relax
  \def\burlalt#1#2{\href{#2}{\texttt{#1}}}\fi

\bibitem[BDPR16]{BDPR16}
\textsc{V.~Barbu}, \textsc{G.~Da~Prato}, and \textsc{M.~R\"ockner}.
\newblock \emph{Stochastic porous media equations}, vol. 2163 of \emph{Lecture
  Notes in Mathematics}.
\newblock Springer, [Cham], 2016.

\bibitem[BR15]{BR15}
\textsc{V.~Barbu} and \textsc{M.~R\"ockner}.
\newblock An operatorial approach to stochastic partial differential equations
  driven by linear multiplicative noise.
\newblock \emph{J. Eur. Math. Soc. (JEMS)} \textbf{17}, no.~7, (2015),
  1789--1815.

\bibitem[BR18]{BR17}
\textsc{V.~Barbu} and \textsc{M.~R\"{o}ckner}.
\newblock Nonlinear {F}okker-{P}lanck equations driven by {G}aussian linear
  multiplicative noise.
\newblock \emph{J. Differential Equations} \textbf{265}, no.~10, (2018),
  4993--5030.
\newblock
  \burlalt{doi:10.1016/j.jde.2018.06.026}{http://dx.doi.org/10.1016/j.jde.2018.06.026}.

\bibitem[BVW15]{Witt}
\textsc{C.~Bauzet}, \textsc{G.~Vallet}, and \textsc{P.~Wittbold}.
\newblock A degenerate parabolic-hyperbolic {C}auchy problem with a stochastic
  force.
\newblock \emph{J. Hyperbolic Differ. Equ.} \textbf{12}, no.~3, (2015),
  501--533.
\newblock
  \burlalt{doi:10.1142/S0219891615500150}{http://dx.doi.org/10.1142/S0219891615500150}.

\bibitem[DG15]{DG15}
\textsc{K.~Dareiotis} and \textsc{M.~Gerencs\'er}.
\newblock On the boundedness of solutions of {SPDE}s.
\newblock \emph{Stochastic Partial Differential Equations: Analysis and
  Computations} \textbf{3}, no.~1, (2015), 84--102.
\newblock
  \burlalt{doi:10.1007/s40072-014-0043-5}{http://dx.doi.org/10.1007/s40072-014-0043-5}.

\bibitem[DG19]{DG17}
\textsc{K.~Dareiotis} and \textsc{B.~Gess}.
\newblock Supremum estimates for degenerate, quasilinear stochastic partial
  differential equations.
\newblock \emph{Ann. Inst. Henri Poincar\'{e} Probab. Stat.} \textbf{55},
  no.~3, (2019), 1765--1796.

\bibitem[DGG19]{DareiotisGerencserGess}
\textsc{K.~Dareiotis}, \textsc{M.~Gerencs\'{e}r}, and \textsc{B.~Gess}.
\newblock Entropy solutions for stochastic porous media equations.
\newblock \emph{J. Differential Equations} \textbf{266}, no.~6, (2019),
  3732--3763.
\newblock
  \burlalt{doi:10.1016/j.jde.2018.09.012}{http://dx.doi.org/10.1016/j.jde.2018.09.012}.

\bibitem[DHV16]{DHV16}
\textsc{A.~Debussche}, \textsc{M.~Hofmanova}, and \textsc{J.~Vovelle}.
\newblock Degenerate parabolic stochastic partial differential equations:
  quasilinear case.
\newblock \emph{Ann. Probab.} \textbf{44}, no.~3, (2016), 1916--1955.
\newblock
  \burlalt{doi:10.1214/15-AOP1013}{http://dx.doi.org/10.1214/15-AOP1013}.

\bibitem[DLN01]{DLN01}
\textsc{N.~Dirr}, \textsc{S.~Luckhaus}, and \textsc{M.~Novaga}.
\newblock A stochastic selection principle in case of fattening for curvature
  flow.
\newblock \emph{Calc. Var. Partial Differential Equations} \textbf{13}, no.~4,
  (2001), 405--425.

\bibitem[DSZ16]{DSZ16}
\textsc{N.~Dirr}, \textsc{M.~Stamatakis}, and \textsc{J.~Zimmer}.
\newblock Entropic and gradient flow formulations for nonlinear diffusion.
\newblock \emph{J. Math. Phys.} \textbf{57}, no.~8, (2016), 081505, 13.

\bibitem[ESvR12]{ESR12}
\textsc{A.~Es-Sarhir} and \textsc{M.-K. von Renesse}.
\newblock Ergodicity of stochastic curve shortening flow in the plane.
\newblock \emph{SIAM J. Math. Anal.} \textbf{44}, no.~1, (2012), 224--244.

\bibitem[FG95]{FLANDOLI}
\textsc{F.~Flandoli} and \textsc{D.~Gatarek}.
\newblock Martingale and stationary solutions for stochastic {N}avier-{S}tokes
  equations.
\newblock \emph{Probab. Theory Related Fields} \textbf{102}, no.~3, (1995),
  367--391.
\newblock
  \burlalt{doi:10.1007/BF01192467}{http://dx.doi.org/10.1007/BF01192467}.

\bibitem[FG18]{FG18-2}
\textsc{B.~Fehrman} and \textsc{B.~Gess}.
\newblock Path-by-path well-posedness of nonlinear diffusion equations with
  multiplicative noise.
\newblock \emph{arXiv preprint arXiv:1807.04230} (2018).

\bibitem[FG19]{FG18}
\textsc{B.~Fehrman} and \textsc{B.~Gess}.
\newblock Well-posedness of nonlinear diffusion equations with nonlinear,
  conservative noise.
\newblock \emph{Arch. Ration. Mech. Anal.} \textbf{233}, no.~1, (2019),
  249--322.
\newblock
  \burlalt{doi:10.1007/s00205-019-01357-w}{http://dx.doi.org/10.1007/s00205-019-01357-w}.

\bibitem[FN08]{FengNualart}
\textsc{J.~Feng} and \textsc{D.~Nualart}.
\newblock Stochastic scalar conservation laws.
\newblock \emph{J. Funct. Anal.} \textbf{255}, no.~2, (2008), 313--373.
\newblock
  \burlalt{doi:10.1016/j.jfa.2008.02.004}{http://dx.doi.org/10.1016/j.jfa.2008.02.004}.

\bibitem[Ges12]{G12}
\textsc{B.~Gess}.
\newblock Strong solutions for stochastic partial differential equations of
  gradient type.
\newblock \emph{J. Funct. Anal.} \textbf{263}, no.~8, (2012), 2355--2383.

\bibitem[GG19]{Gassiat1}
\textsc{P.~Gassiat} and \textsc{B.~Gess}.
\newblock Regularization by noise for stochastic {H}amilton-{J}acobi equations.
\newblock \emph{Probab. Theory Related Fields} \textbf{173}, no. 3-4, (2019),
  1063--1098.

\bibitem[GGLS20]{Gassiat2}
\textsc{P.~Gassiat}, \textsc{B.~Gess}, \textsc{P.-L. Lions}, and \textsc{P.~E.
  Souganidis}.
\newblock Speed of propagation for {H}amilton--{J}acobi equations with
  multiplicative rough time dependence and convex {H}amiltonians.
\newblock \emph{Probab. Theory Related Fields} \textbf{176}, no. 1-2, (2020),
  421--448.
\newblock
  \burlalt{doi:10.1007/s00440-019-00921-5}{http://dx.doi.org/10.1007/s00440-019-00921-5}.

\bibitem[GH18]{GH18}
\textsc{B.~Gess} and \textsc{M.~Hofmanov\'{a}}.
\newblock Well-posedness and regularity for quasilinear degenerate
  parabolic-hyperbolic {SPDE}.
\newblock \emph{Ann. Probab.} \textbf{46}, no.~5, (2018), 2495--2544.
\newblock
  \burlalt{doi:10.1214/17-AOP1231}{http://dx.doi.org/10.1214/17-AOP1231}.

\bibitem[GK96]{Istvan}
\textsc{I.~Gy\"ongy} and \textsc{N.~Krylov}.
\newblock Existence of strong solutions for {I}t\^o's stochastic equations via
  approximations.
\newblock \emph{Probab. Theory Related Fields} \textbf{105}, no.~2, (1996),
  143--158.
\newblock
  \burlalt{doi:10.1007/BF01203833}{http://dx.doi.org/10.1007/BF01203833}.

\bibitem[GPS16]{SougGess}
\textsc{B.~Gess}, \textsc{B.~Perthame}, and \textsc{P.~E. Souganidis}.
\newblock Semi-discretization for stochastic scalar conservation laws with
  multiple rough fluxes.
\newblock \emph{SIAM J. Numer. Anal.} \textbf{54}, no.~4, (2016), 2187--2209.

\bibitem[GR17]{GR17}
\textsc{B.~Gess} and \textsc{M.~R\"{o}ckner}.
\newblock Stochastic variational inequalities and regularity for degenerate
  stochastic partial differential equations.
\newblock \emph{Trans. Amer. Math. Soc.} \textbf{369}, no.~5, (2017),
  3017--3045.
\newblock \burlalt{doi:10.1090/tran/6981}{http://dx.doi.org/10.1090/tran/6981}.

\bibitem[GS15]{GS14}
\textsc{B.~Gess} and \textsc{P.~E. Souganidis}.
\newblock Scalar conservation laws with multiple rough fluxes.
\newblock \emph{Commun. Math. Sci.} \textbf{13}, no.~6, (2015), 1569--1597.

\bibitem[GS17a]{GS14-2}
\textsc{B.~Gess} and \textsc{P.~E. Souganidis}.
\newblock Long-time behavior, invariant measures, and regularizing effects for
  stochastic scalar conservation laws.
\newblock \emph{Comm. Pure Appl. Math.} \textbf{70}, no.~8, (2017), 1562--1597.

\bibitem[GS17b]{GS16-2}
\textsc{B.~Gess} and \textsc{P.~E. Souganidis}.
\newblock Stochastic non-isotropic degenerate parabolic--hyperbolic equations.
\newblock \emph{Stochastic Process. Appl.} \textbf{127}, no.~9, (2017),
  2961--3004.

\bibitem[GS19]{GS17-2}
\textsc{B.~Gess} and \textsc{S.~Smith}.
\newblock Stochastic continuity equations with conservative noise.
\newblock \emph{J. Math. Pures Appl. (9)} \textbf{128}, (2019), 225--263.
\newblock
  \burlalt{doi:10.1016/j.matpur.2019.02.002}{http://dx.doi.org/10.1016/j.matpur.2019.02.002}.

\bibitem[GT14]{GT14}
\textsc{B.~Gess} and \textsc{J.~M. T{\"o}lle}.
\newblock Multi-valued, singular stochastic evolution inclusions.
\newblock \emph{J. Math. Pures Appl. (9)} \textbf{101}, no.~6, (2014),
  789--827.

\bibitem[HKRSs18]{KarlsenRisebro}
\textsc{H.~Hoel}, \textsc{K.~H. Karlsen}, \textsc{N.~H. Risebro}, and
  \textsc{E.~B. Storr\o~sten}.
\newblock Path-dependent convex conservation laws.
\newblock \emph{J. Differential Equations} \textbf{265}, no.~6, (2018),
  2708--2744.
\newblock
  \burlalt{doi:10.1016/j.jde.2018.04.045}{http://dx.doi.org/10.1016/j.jde.2018.04.045}.

\bibitem[Hof13]{HOF2}
\textsc{M.~Hofmanov\'a}.
\newblock Degenerate parabolic stochastic partial differential equations.
\newblock \emph{Stochastic Process. Appl.} \textbf{123}, no.~12, (2013),
  4294--4336.
\newblock
  \burlalt{doi:10.1016/j.spa.2013.06.015}{http://dx.doi.org/10.1016/j.spa.2013.06.015}.

\bibitem[KO82]{KO82}
\textsc{K.~Kawasaki} and \textsc{T.~Ohta}.
\newblock Kinetic drumhead model of interface. i.
\newblock \emph{Progress of Theoretical Physics} \textbf{67}, no.~1, (1982),
  147--163.

\bibitem[KR79]{KR79}
\textsc{N.~V. Krylov} and \textsc{B.~L. Rozovski{\u\i}}.
\newblock Stochastic evolution equations.
\newblock In \emph{Current problems in mathematics, {V}ol. 14 ({R}ussian)},
  71--147, 256. Akad. Nauk SSSR, Vsesoyuz. Inst. Nauchn. i Tekhn. Informatsii,
  Moscow, 1979.

\bibitem[Kry13]{K_Ito}
\textsc{N.~V. Krylov}.
\newblock A relatively short proof of {I}t\^o's formula for {SPDE}s and its
  applications.
\newblock \emph{Stoch. Partial Differ. Equ. Anal. Comput.} \textbf{1}, no.~1,
  (2013), 152--174.

\bibitem[KS91]{Kar}
\textsc{I.~Karatzas} and \textsc{S.~E. Shreve}.
\newblock \emph{Brownian motion and stochastic calculus}, vol. 113 of
  \emph{Graduate Texts in Mathematics}.
\newblock Springer-Verlag, New York, second ed., 1991.
\newblock
  \burlalt{doi:10.1007/978-1-4612-0949-2}{http://dx.doi.org/10.1007/978-1-4612-0949-2}.

\bibitem[Kun97]{Kun}
\textsc{H.~Kunita}.
\newblock \emph{Stochastic Flows and Stochastic Differential Equations}.
\newblock Cambridge Studies in Advanced Mathematics. Cambridge University
  Press, 1997.

\bibitem[LL06a]{LL06}
\textsc{J.-M. Lasry} and \textsc{P.-L. Lions}.
\newblock Jeux \`a champ moyen. {I}. {L}e cas stationnaire.
\newblock \emph{C. R. Math. Acad. Sci. Paris} \textbf{343}, no.~9, (2006),
  619--625.

\bibitem[LL06b]{LL06-2}
\textsc{J.-M. Lasry} and \textsc{P.-L. Lions}.
\newblock Jeux \`a champ moyen. {II}. {H}orizon fini et contr\^ole optimal.
\newblock \emph{C. R. Math. Acad. Sci. Paris} \textbf{343}, no.~10, (2006),
  679--684.

\bibitem[LL07]{LL07}
\textsc{J.-M. Lasry} and \textsc{P.-L. Lions}.
\newblock Mean field games.
\newblock \emph{Jpn. J. Math.} \textbf{2}, no.~1, (2007), 229--260.

\bibitem[LPS13]{LPS13}
\textsc{P.-L. Lions}, \textsc{B.~Perthame}, and \textsc{P.~E. Souganidis}.
\newblock Scalar conservation laws with rough (stochastic) fluxes.
\newblock \emph{Stoch. Partial Differ. Equ. Anal. Comput.} \textbf{1}, no.~4,
  (2013), 664--686.

\bibitem[LPS14]{LPS14}
\textsc{P.-L. Lions}, \textsc{B.~Perthame}, and \textsc{P.~E. Souganidis}.
\newblock Scalar conservation laws with rough (stochastic) fluxes: the
  spatially dependent case.
\newblock \emph{Stoch. Partial Differ. Equ. Anal. Comput.} \textbf{2}, no.~4,
  (2014), 517--538.

\bibitem[LR15]{LR15}
\textsc{W.~Liu} and \textsc{M.~R{\"o}ckner}.
\newblock \emph{Stochastic partial differential equations: an introduction}.
\newblock Universitext. Springer, Cham, 2015.

\bibitem[LS98a]{Soug1}
\textsc{P.-L. Lions} and \textsc{P.~E. Souganidis}.
\newblock Fully nonlinear stochastic partial differential equations.
\newblock \emph{C. R. Acad. Sci. Paris S\'{e}r. I Math.} \textbf{326}, no.~9,
  (1998), 1085--1092.
\newblock
  \burlalt{doi:10.1016/S0764-4442(98)80067-0}{http://dx.doi.org/10.1016/S0764-4442(98)80067-0}.

\bibitem[LS98b]{Soug2}
\textsc{P.-L. Lions} and \textsc{P.~E. Souganidis}.
\newblock Fully nonlinear stochastic partial differential equations: non-smooth
  equations and applications.
\newblock \emph{C. R. Acad. Sci. Paris S\'{e}r. I Math.} \textbf{327}, no.~8,
  (1998), 735--741.
\newblock
  \burlalt{doi:10.1016/S0764-4442(98)80161-4}{http://dx.doi.org/10.1016/S0764-4442(98)80161-4}.

\bibitem[MR18]{MR18}
\textsc{I.~Munteanu} and \textsc{M.~R\"{o}ckner}.
\newblock Total variation flow perturbed by gradient linear multiplicative
  noise.
\newblock \emph{Infin. Dimens. Anal. Quantum Probab. Relat. Top.} \textbf{21},
  no.~1, (2018), 1850003, 28.

\bibitem[Par75]{P75}
\textsc{{\'E}.~Pardoux}.
\newblock {E}quations aux d\'eriv\'ees partielles stochastiques non lin\'eaires
  monotones.
\newblock \emph{PhD thesis} (1975).

\bibitem[PR07]{PR07}
\textsc{C.~Pr{\'e}v{\^o}t} and \textsc{M.~R{\"o}ckner}.
\newblock \emph{A concise course on stochastic partial differential equations},
  vol. 1905 of \emph{Lecture Notes in Mathematics}.
\newblock Springer, Berlin, 2007.

\bibitem[Sou16]{Soug3}
\textsc{P.~E. Souganidis}.
\newblock Fully nonlinear first- and second-order stochastic partial
  differential equations.
\newblock \emph{CIME lecture notes} \textbf{327}, (2016), 1–37.

\bibitem[SY04]{SY04}
\textsc{P.~E. Souganidis} and \textsc{N.~K. Yip}.
\newblock Uniqueness of motion by mean curvature perturbed by stochastic noise.
\newblock \emph{Ann. Inst. H. Poincar\'e Anal. Non Lin\'eaire} \textbf{21},
  no.~1, (2004), 1--23.

\bibitem[T{\"{o}}l18]{T18}
\textsc{J.~M. T{\"{o}}lle}.
\newblock Estimates for nonlinear stochastic partial differential equations
  with gradient noise via {D}irichlet forms.
\newblock In \emph{Stochastic partial differential equations and related
  fields}, vol. 229 of \emph{Springer Proc. Math. Stat.},  249--262. Springer,
  Cham, 2018.

\end{thebibliography}
\bibliographystyle{Martin}
\end{document}